\documentclass[11pt]{amsart}
\usepackage[active]{srcltx}
\usepackage{mathrsfs}
\usepackage[T1]{fontenc}

\usepackage[bookmarks]{hyperref}\hypersetup{backref, colorlinks=true}
\hypersetup{backref, colorlinks=true, colorlinks   = true,
	urlcolor     = blue, linkcolor = blue, citecolor   = magenta}
\usepackage{latexsym,enumerate}
\DeclareMathOperator{\esslim}{ess\, lim}
\usepackage{amsmath,amssymb,amsthm,amsfonts,latexsym}
\usepackage[bookmarks]{hyperref}

\hypersetup{backref, colorlinks=true}

\def\adh#1{\overline{#1}}

\let\=\partial

\newtheorem {pro}{Proposition}[section]
\newtheorem {thm}[pro]{Theorem}
\newtheorem {cor}[pro]{Corollary}
\newtheorem{lem}[pro]{Lemma}

\theoremstyle{definition}
 \newtheorem {rem}[pro]{Remark}
\newtheorem {dfn}[pro]{Definition}

\newtheorem {exa}[pro]{Example}

\newtheorem {ste}{Step}

 \newtheorem {rems}[pro]{Remarks and examples}

\thanks{Research supported by the Narodowe Centrum Nauki (Poland) under  grant 
	2021/43/B/ST1/02359.}

\newcommand{\dor}{\mathscr{V}}

\newcommand{\blj}{\mathbf{B}^{l,j}_\mu}

\newcommand{\jac}{\mbox{jac}\,}

\newcommand{\tra}{\mathbf{tr}}\newcommand{\trd}{\mathbf{Tr}}\newcommand{\res}{\mathbf{res}}

\newcommand{\R}{\mathbb{R}}

\newcommand{\N}{\mathbb{N}}

\newcommand{\dsc}{\mathscr{D}}
\newcommand{\cc}{\mathscr{C}}\newcommand{\C}{\mathcal{C}}
\newcommand{\st}{\mathscr{S}}
\newcommand{\poi}{\mathcal{P}}
\newcommand{\Hom}{\mathbf{Hom}}

\newcommand{\cfr} {\mathfrak{c}}
\newcommand{\lcs}{\mathbb{\R}}
\newcommand{\rfr}{\mathfrak{r}}\newcommand{\Rfr}{\mathfrak{R}}
\newcommand{\rfh}{{\hat{\mathfrak{r}}}}

\newcommand{\et}{\quad \mbox{and} \quad }

\newcommand{\hn}{\mathcal{H}}

\newcommand{\mba}{ {\overline{M}}}

\newcommand{\mep}{ {M^\ep}}
\newcommand{\uxo}{U_\xo}

\newcommand{\nep}{ {N^\ep}}

\newcommand{\wsc}{\mathscr{W}} \newcommand{\wca}{\mathcal{W}}
\newcommand{\cch}{\check{C}}\newcommand{\chh}{\check{\C}}
\newcommand{\hch}{\check{H}}\newcommand{\dfh}{{\hat{d}^F}}\newcommand{\dsh}{{\hat{d}^S}}

\newcommand{\F}{\mathcal{F}}
\newcommand{\ep}{\varepsilon}\newcommand{\epd}{\frac{\varepsilon}{2}}

\newcommand{\E}{\mathcal{E}}

\newcommand{\pa}{\partial}

\newcommand{\hh}{\mathcal{V}}

\newcommand{\bou}{\mathbf{B}}
\newcommand{\sph}{\mathbf{S}}
\newcommand{\orn}{{0_{\R^n}}}

\newcommand{\Bb}{\overline{ \mathbf{B}}}

\newcommand{\xo}{{x_0}}
\newcommand{\mc}{{\check{M}}}

\newcommand{\supp}{\mbox{\rm supp}}

\newcommand{\hsk}{ \hskip-1mm}

\title[]{$L^p$ Hodge theory for bounded subanalytic manifolds}

\makeatletter

\@addtoreset{equation}{section}

\makeatother

\author[ G. Valette]{ Guillaume Valette}

\address[G. Valette]{Instytut Matematyki Uniwersytetu
Jagiello\'nskiego, ul. S. \L ojasiewicza 6, Krak\'ow, Poland}\email{guillaume.valette@im.uj.edu.pl}

\keywords{$L^p$ Differential forms, Lefschetz duality, de Rham theorem, Hodge theory,   $p$-harmonic forms, currents,  subanalytic varieties, singularities.}

\subjclass[2020]{58A14, 32B20, 58A25, 57P10,  35J92, 14F40, 58A12,  14P10}

\begin{document}

\begin{abstract}
	Given a bounded subanalytic submanifold of $\mathbb{R}^n$, possibly admitting singularities within its closure,
 we study the cohomology of $L^p$ differential forms having an $L^p$  exterior differential (in the sense of currents) and satisfying  Dirichlet or Neumann condition. We show an $L^p$ Hodge decomposition theorem,  an $L^p$ de Rham theorem, as well as a Lefschetz duality theorem  between $L^p$ and $L^{p'}$ forms (with $\frac{1}{p}+\frac{1}{p'}=1$) in the case where $p$ is large or close to $1$. This is achieved by proving that de Rham's pairing between complementary $L^p$ differential forms induces a   pairing between cohomology classes which is nondegenerate (for such $p$).  The main difficulty to carry it out is to show the density (in  Sobolev spaces of differential forms) of  forms that vanish near some singularities and are smooth up to the closure of the underlying manifold in the case $p$ large (Theorem \ref{thm_dense_formes}). This result, which is of independent interest, also makes it possible to give  a trace theorem that leads us to some explicit characterizations of Dirichlet and Neumann conditions in terms of traces and residues.
\end{abstract}
\maketitle
\begin{section}{Introduction}
The famous Hodge decomposition theorem asserts that on a compact smooth Riemannian manifold,  every smooth differential form can be decomposed as the sum of an exact form, a coexact form, and a harmonic form \cite{warmer}. As a consequence, every de Rham cohomology class contains a unique harmonic representative. This also provides a nice interpretation of the Poincar\'e duality isomorphism which is then induced by the action of the Hodge star operator on harmonic forms. The situation is more complicated if we drop the compactness assumption, especially if singularities appear in the closure of the manifold, and this has much to do with the Lipschitz geometry of these singularities. An important step in the comprehension of Hodge  theory of singular varieties was accomplished by J. Cheeger who showed a de Rham theorem and a Hodge decomposition theorem for $L^2$ differential forms on varieties that admit only even codimensional metrically conical singularities \cite{c1,c2,c3,c4}. This theory turned out to be related with the famous intersection homology  introduced by M. Goresky and R. Macpherson \cite{ih1,ih2,cgm}. Unfortunately,   Cheeger's de Rham type theorem no longer holds for non metrically conical singularities and these results about $L^2$ cohomology fail \cite{bb}. In \cite{linfty}, the author of the present article established a counterpart of Cheeger's de Rham theorem for $L^p$ differential forms on subanalytic bounded manifolds (not necessarily compact) in the case $p=\infty$. This result was then extended to the case ``$p$ close to $1$'' or ``$p$ sufficiently large''  \cite[Theorems $2.9$ and $2.10$]{gvpoincare},  and some Poincar\'e duality isomorphisms between $L^p$ and $L^{p'}$ cohomologies were given (here $p'$ stands for the H\"older conjugate of $p$, i.e. $\frac{1}{p}+\frac{1}{p'}=1$). The isomorphisms provided by the latter article were however not completely satisfying to carry out a nice Hodge theory on singular varieties as a boundary phenomenon arose at singularities (see \cite[Section $5.4$]{gvpoincare}).

 It is on the other hand well-known that Hodge theory goes over compact manifolds with boundary, the Hodge star operator inducing isomorphisms between harmonic forms satisfying Dirichlet and Neumann conditions respectively, and that these isomorphisms correspond to Lefschetz duality. In this article we establish a Lefschetz duality theorem for $L^p$ forms on bounded subanalytic manifolds (not necessarily closed) when $p$ is large or close to $1$ (Theorem \ref{thm_lefschetz_duality}) as well as a de Rham type theorem for these forms (Theorem \ref{thm_derham}). This $L^p$ Lefschetz duality theorem, which generalizes the main result of \cite{gvpoincare}, requires to make sure that integration by parts of products of forms satisfying complementary conditions is licit. Namely, we have to  show that, given a subanalytic manifold $M$ and a closed subanalytic subset $A$ of $\delta M:=\mba \setminus M$, the natural pairing between $\wca^j_p(M,A)$, which is the space of $L^p$ forms that vanish in the Dirichlet sense on $A$,  and $\wca^{m-j}_{p'}(M,\delta M\setminus A)$ induces a mapping in cohomology (here $m=\dim M$, see Corollary \ref{cor_lsp_A}). This fact will be derived from a density theorem asserting that on a normal subanalytic manifold (i.e. connected at frontier points, see Definition \ref{dfn_normal}), the space of forms that are smooth up to the frontier and vanish nearby $A$ is dense in $\wca^j_p(M,A)$ for all $p$ sufficiently large (Theorem \ref{thm_dense_formes}). This theorem, which is of its own interest and is essential for our Lefschetz duality result, is the main technical difficulty of this article, and its proof occupies all section \ref{sect_density}.

 As a byproduct of our Lefschetz duality theorem, we derive a Poincar\'e inequality for differential forms (Corollary \ref{cor_image_fermee}) and the $L^p$ Hodge decomposition theorem (Theorem \ref{thm_hodge}). The idea is that since we are working with the $L^p$ norm, the $L^p$  Hodge decomposition theorem should give a $p$-harmonic representative of every cohomology class, i.e., a zero of the $p$-coboundary operator, defined for $\omega \in  \wca^j_p(M)$ as
  $$\delta_p \omega:=(-1)^{m(j-1)+1}*_{p'}d*_p \omega,\quad \mbox{where }\quad *_p\omega:=\frac{|\omega|^{p-2}}{||\omega||_{L^p(M)}^{p-2}}*\omega,$$
 where $*$ is the Hodge star operator and $d$ the exterior differentiation, $m=\dim M$. In the case $p=2$, we of course find the usual coboundary operator, which means that our $L^p$ Hodge theory naturally extends the classical one. Forms that are both closed and $p$-coclosed (i.e. $d\omega=0$, $\delta_p\omega=0$) are called $p$-harmonic. In the case $j=0$, these are nothing but the zeros of the $p$-Laplace operator $\Delta_p$ \cite{degen, gupel} which is a well-known generalization of the usual Laplace-Beltrami operator. Differential forms that are in the image of $\delta_p$ are called $p$-coexact.  The $L^p$ Hodge  theorem (Theorem \ref{thm_hodge}) thus provides, given  a subanalytic subset $A$ of $\delta M$, a decomposition of any $L^p$ form that has an $L^p$ exterior differential as the sum of an exact form vanishing in the Dirichlet sense on $A$, a $p$-coexact form vanishing  in the   Neumann sense on $\delta M\setminus A$, and a $p$-harmonic form vanishing in the Dirichlet sense on $A$ and  in the Neumann sense on $\delta M\setminus A$. We also compute the relative $L^p$ cohomology groups for $p$ large (Theorem \ref{thm_derham}), which shows that they are finitely generated (Corollary \ref{cor_finitely_generated}) and topological invariants of $\mba$, establishing that the set of $p$-harmonic forms is a $\cc^0$ manifold  for such $p$ (the $p$-Laplace operator is not linear, see Corollary \ref{cor_hodge}).

We will prove a residue formula (see (\ref{eq_residue_formula})) which will provide an efficient way to perform integration by parts on noncompact subanalytic manifolds. Thanks to our density theorems (Theorems \ref{thm_dense_formes} and \ref{thm_pprime}), this will enable us to characterize more explicitly functions that vanish on a definable subset of $\delta M$ in the Dirichlet sense (Corollaries \ref{cor_vanishing_res} and \ref{cor_trace_formes}).
This residue theory for $L^1$ differential forms on subanalytic varieties could prove useful for other purposes, such as the study of Plateau problems or to put boundary conditions for PDE's on singular domains \cite{gupel}.

 It is worthy of notice that we will also establish that the  boundary conditions that we put are vacuous  for all $j>\dim \delta M+1$ if  $p$ is large, and all $j<m-\dim \delta M-1$ if $p$ is close to $1$  (see (\ref{eq_tra_0}) and  (\ref{eq_res_0})). We thus get an $L^p$ Hodge decomposition theorem without any boundary condition at singular points for all such $j$ and $p$ (Corollary \ref{cor_hodge_deltaM_ptit}). In particular, in the case where the closure of $M$ admits only isolated singularities we get a Hodge decomposition of $\wca^j_p(M)$ for all $j\le m-2$ (for $p$ close to $1$). These results somehow unravel the duality pointed out in \cite[Corollary $2.12$]{gvpoincare}, which applied to every $j<m-\dim \delta M-1$, for all $p$ close to $1$.

\begin{subsection}{Some notations and conventions} By ``manifold'', we will always mean $\cc^\infty$ submanifold of $\R^n$. A submanifold of $\R^n$ will always be endowed with its canonical measure, provided by volume forms.  As integrals will always be considered with respect to this measure, we will not match the measure when integrating on a submanifold of $\R^n$.   Throughout this article, the letter
$M$  stands for a {\it bounded oriented subanalytic} submanifold of $\R^n$ and $m$ for its dimension.  We set $\dim \emptyset:=-1$.

 $0_{\R^n}\;$: origin of $\R^n$. When the ambient space will be obvious from the context, we will however omit the subscript $\R^n$.

  $x\cdot y$ and $|x|$:  euclidean inner product of $x\in \R^n$ and $y\in \R^n$ and euclidean norm of $x$.

   $|\omega(x)|$: euclidean norm of a differential form $\omega$ at $x$, i.e.  $|\omega(x)\wedge *\omega(x)|^{1/2}$, where $*$ is the Hodge operator. 

   $\sph(x,\ep),\bou(x,\ep)$, and $\adh{\bou}(x,\ep)$: sphere, open ball, and closed ball of radius $\ep$ that are centered  at $x\in \R^n$ (in the euclidean norm).

  $\adh{A}$: closure of $A\subset \R^n$. We also set $\delta A=\adh{A}\setminus A$.

$\Sigma_A$:  stratification induced by $\Sigma$ on $A$ (if $\Sigma $ is a stratification compatible with  $A$), Definition \ref{dfn_stratifications}.

``$\xi\lesssim \zeta$ on $B$'' (or also ``$\xi(x)\lesssim \zeta(x)$ for $x\in B$''): means that there is  $C>0$ such that $\xi(x) \le C\zeta(x)$ for all $x\in B$ ($\xi$ and $\zeta$ being  two nonnegative functions  on a set $A\supset B$).

$X_{reg}$:  regular locus of a subanalytic set $X$, section \ref{sect_sub}.


  $Dh$ and $\pa u$: derivative of the mapping $h$ and gradient of a function $u$.

$\jac\, h$: jacobian of  the mapping $h$, i.e.,   absolute value of the determinant of $Dh$.

 $||v||_{L^p(S)}$: $L^p$ norm  (possibly infinite) of the restriction to $S$ of a measurable mapping $v:M\to \R^k$, if $S$ is  a submanifold of $M$.

  $W^{1,p}(M)$:   Sobolev space, i.e. $W^{1,p}(M)=\{u\in L^p(M):|\pa u| \in L^p(M)\} $.

$\trd^p \alpha$ : $p$-trace of a form,  section \ref{sect_trd}.

$\tra_S $: trace operator on a set $S$, defined in section \ref{sect_trace_operators} (in the case of functions, i.e. $0$-forms) and section \ref{sect_trace_forms} (in the case of $j$-forms, $j\ge 0$).


 $L_p^j(M)$: space of measurable $L^p$ differential $j$-forms  on $M$.

$\wca^j_p(M)=\{\omega\in L^j_p(M):d\omega \in L^{j+1}_p(M)\}$  (see (\ref{eq_wcap})).

 $\wca^j_p(M,A)$: space of elements of $\wca^j_p(M)$ satisfying Dirichlet condition on $A\subset \delta M$ (section \ref{sect_forms_with_boundary conditions}).


 $H^j_p(M,A)$ (denoted $H^j_p(M)$ if $A=\emptyset$): $j^{th}\,$  cohomology group of $(\wca^i_p(M,A),d)_{i\in \N}$.

 $<\alpha,\beta>\;$: de Rham's pairing of differential forms (see (\ref{eq_pairing})).

$\omega_\xi(x)(\zeta):=\omega(x)(\xi(x) \otimes \zeta)$, if $\omega$ is a differential form and $\xi$ a vector field, see (\ref{eq_om_xi}).

$\cc^{j,\infty}(E)$: for $E\subset \mba$, elements of $\cc^{j,\infty}(M)$ that extend smoothly to a neighborhood of $E$ in $\R^n$. In particular, $\cc^{j,\infty}(\mba)\subset \cc^{j,\infty}(M)$ is the space of  $j$-forms that are smooth up to $\delta M$.  When $j=0$, i.e., in the case of functions, we omit $j$.


$\supp\, \omega$: support  of a  differential form $\omega$ on $M$.

$\supp_U \omega$: closure in $U\supset M$ of  $\supp\,\omega$.  We then set for $U\supset M$:
\begin{equation}\label{eq_ccjU}
\cc^{j,\infty}_{U}(\mba):=\{\omega\in \cc^{j,\infty}(\overline{M}): \supp_\mba \omega \subset U \}.
\end{equation}
Alike  $\cc^{j,\infty}(\mba)$, the space $\cc^{j,\infty}_U(\mba)$ will be regarded as a subspace of $\cc^{j,\infty}(M)$.  Note that $$\cc^{j,\infty}_{U}(\mba)=\{\omega\in \cc^{j,\infty}(\overline{M}): \supp_U \omega \mbox{ is compact} \}.$$  The spaces $\cc^{j,\infty}_U(M)$ and $\wca^j_{p,U}(M,A)$ are then defined analogously. If $U=M$, we use the more standard notation $\cc_0^{j,\infty}(M)$. 

As usual \cite[Chapter $3$]{grihar}, a {\bf $j$-current} on an open subset $U$ of $\R^m$ will be a linear functional $T:\cc^{m-j,\infty}_0(U)\to \R$ which is continuous in the $\cc^k$ topology for all $k$, and a current on a manifold $M$ will be  a functional which gives rise to a current on an open subset of $\R^m$ after a push-forward under the coordinate systems of $M$. We will denote by $\dsc^j(M)$ the space of $j$-currents on $M$.

$\hn^k$ and $L^p(A,\hn^k)$ :    $k$-dimensional Hausdorff measure  and space of $L^p$ functions  with respect to the measure $\hn^k$.

 $p'\;$: H\"older conjugate of $p$, i.e., $p'=\frac{p}{p-1}$, $1'=\infty$, $\infty'=1$. Given a topological vector space $V$, we denote by $V'$ the topological dual.

$<T,\beta>'$: duality bracket between a continuous linear functional $T:V\to \R$ on a topological vector space $V$ and a vector $\beta \in V$, i.e. $<T,\beta>'=T(\beta)$.

$\Hom(V_1,V_2)$: space of homomorphisms (of $\R$-vector spaces) from $V_1$ to $V_2$.

A mapping $h:A\to B$, $A\subset \R^n, B\subset \R^k$, is {\bf Lipschitz} if there is a constant $C$ such that $|h(x)-h(x')|\le C|x-x'|$ for all $x$ and $x'$. It is {\bf bi-Lipschitz} if it is a homeomorphism and if in addition $h$ and $h^{-1}$ are both Lipschitz.

 We shall several times make use of the following function.  Let $\psi:\R \to [0,1]$  be a nondecreasing $\cc^\infty$ function
such that $\psi\equiv 0$ on $(-\infty,\frac{1}{2})$ and $\psi\equiv 1$ near $[1,\infty)$, and set for $\eta$ positive and $s\in \R$, \begin{equation}\label{eq_psieta}\psi_\eta(s):=\psi\left(\frac{s}{\eta}\right).\end{equation}
Observe that we have for $\eta>0$:
\begin{equation}\label{eq_psieta_ineq}
	\sup \psi_\eta=1 \et  \sup \left|\frac{d\psi_\eta}{dt}\right| \lesssim \frac{1}{\eta}, \quad \mbox{as well as }\;\,  \supp\, \frac{d\psi_\eta}{dt}\subset [\frac{\eta}{2},\eta).
\end{equation}
\end{subsection}

\end{section}

\section{Definitions and main results}

\begin{subsection}{Subanalytic sets}
 We refer the reader to  \cite{bm, ds, loj, livre} for all the basic facts about subanalytic geometry.  Actually, \cite{livre}  gives in addition a detailed presentation of the Lipschitz properties of these sets, so that the reader can find there all the needed facts to understand the result achieved in the present article.

\begin{dfn}\label{dfn_semianalytic}
A subset $E\subset \R^n$ is called {\bf semi-analytic} if it is {\it locally}
defined by finitely many real analytic equalities and inequalities. Namely, for each $a \in   \R^n$, there are
a neighborhood $U$ of $a$ in $\R^n$, and real analytic  functions $f_{ij}, g_{ij}$ on $U$, where $i = 1, \dots, r, j = 1, \dots , s_i$, such that
\begin{equation}\label{eq_definition_semi}
E \cap   U = \bigcup _{i=1}^r\bigcap _{j=1} ^{s_i} \{x \in U : g_{ij}(x) > 0 \mbox{ and } f_{ij}(x) = 0\}.
\end{equation}

The flaw of the  semi-analytic category is that  it is not preserved by analytic mappings, even when they are proper. To overcome this problem, we prefer working with   subanalytic sets, which are defined as the projections of semi-analytic sets.

A subset $E\subset \R^n$  is  {\bf  subanalytic} if 
 each point $x\in\R^n$ has a neighborhood $U$ such that $U\cap E$ is the image under the canonical projection $\pi:\R^n\times\R^k\to\R^n$ of some relatively compact semi-analytic subset of $\R^n\times\R^k$ (where $k$ depends on $x$).
   
   A subset $Z$ of $\R^n$ is  {\bf globally subanalytic} if $\hh_n(Z)$ is a subanalytic subset of $\R^n$, where $\hh_n : \R^n  \to (-1,1) ^n$ is the homeomorphism defined by $$\hh_n(x_1, \dots, x_n) :=  (\frac{x_1}{\sqrt{1+|x|^2}},\dots, \frac{x_n}{\sqrt{1+|x|^2}} ).$$

   We say that {\bf a mapping $f:A \to B$ is  subanalytic} (resp. globally subanalytic), $A \subset \R^n$, $B\subset \R^m$ subanalytic (resp. globally subanalytic), if its graph is a  subanalytic  (resp. globally subanalytic) subset of $\R^{n+m}$. In the case $B=\R$, we say that  $f$ is a (resp. globally) {\bf  subanalytic function}. For the sake of simplifying statements, globally subanalytic sets and mappings will be referred in this article as {\bf definable} sets and mappings (this terminology is often used by o-minimal geometers \cite{vdd,costeomin}).

   \end{dfn}

The globally subanalytic category is very well adapted to our purpose.  It is stable under intersection, union, complement, and projection. It  constitutes an o-minimal structure \cite{vdd, costeomin}, from which it comes down that definable sets enjoy a large number of finiteness properties (see \cite{costeomin,livre} or section \ref{sect_sub} below for more). The results of this article are actually valid on any polynomially bounded o-minimal structure (expanding $\R$).

\end{subsection}

\begin{subsection}{$L^p$ differential forms}\label{sect_lpforms}
	Given $p \in [1,\infty)$, we say that a (measurable) differential $j$-form $\omega$
		on $M$    is $L^p$  if so is $x\mapsto |\omega(x)|$, i.e.
	$$||\omega||_{L^p(M)}:=\left(\int_{x\in M} |\omega(x)|^p\right)^\frac{1}{p}<\infty.$$
 We say that $\omega$ is $L^\infty$ if   $||\omega||_{L^\infty(M)}:= \mbox{ess sup}_{x\in M} |\omega(x)|<\infty$ (the essential supremum).
	We denote by $L^j_p(M)$ the $\R$-vector space constituted by the $L^p$ $j$-forms on $M$ and define
	\begin{equation}\label{eq_wcap}\wca^j_{p} (M):=\{\omega\in L^j_p(M): d\omega \in L^{j+1}_p(M) \},\end{equation}
that we endow with its natural norm:
\begin{equation}\label{eq_norm_wca}
	||\omega||_{\wca^j_p(M)}:=||\omega||_{L^p(M)} +||d\omega||_{L^p(M)}.
 \end{equation}
Here, the exterior differential is considered in the sense of currents. It is worthy of notice that, since bounded subanalytic manifolds have finite measure \cite[Chapter $4$]{livre}, $\wca^j_p(M)\subset \wca^j_q(M)$, as soon as $p\ge q$.

If $\omega$ is a differential form  on $M$ and  $U$ a subset   of $\R^n$ containing $M$, we denote by $\supp_U \omega$ the closure in $U$ of the support of $\omega$, and by $\wca^j_{p,U}(M)$ the subspace of $\wca^j_p(M)$ constituted by the form that {\bf have compact support in $U$}, i.e. the forms $ \omega\in \wca_p^j (M) $ for which $\supp_U \omega$ is compact. Clearly,
\begin{equation}\label{eq_wcapU}\wca^j_{p,U}(M)=\{ \omega\in \wca_p^j (M) : \supp_\mba\, \omega\subset U\}.\end{equation}

\subsection{The operator $\trd^p$.}\label{sect_trd}Given  $\alpha \in L^j_p(M)$ and $\beta \in L^{m-j}_{p'}(M)$, $p\in [1,\infty]$ (with $1/p+1/p'=1$),  we set (recall that $M$ is oriented and $m=\dim M$):
\begin{equation}\label{eq_pairing}<\alpha,\beta>:=\int_M \alpha \wedge \beta. \end{equation}
We then define a continuous operator
$$\trd^p:\wca^j_p (M)\to \wca^{m-j-1}_{p'}(M)' $$ by setting for $\beta\in \wca^{m-j-1}_{p'}(M)$, if  $\alpha\in \wca^j_p(M)$:
$$<\trd^p \alpha,\beta>':=\int_M d(\alpha\wedge \beta)=<d\alpha,\beta>+(-1)^j<\alpha, d\beta>.$$
 Clearly, the functional $\trd^p\alpha$ vanishes on $M$ for all $p\in [1,\infty]$, in the sense that it vanishes on the forms $\beta$ that are compactly supported in $M$.
It is therefore natural to say that  $\trd^p\alpha:\wca^{m-j-1}_{p'} (M)\to \R$ {\bf vanishes on a subset} $A$ of $\delta M=\mba\setminus M$ if for all $\beta\in \wca^{m-j-1}_{p'}(M)$ having compact support in $ M\cup A$, we have
$$<\trd^p\alpha,\beta>'=0.$$
We will also sometimes shortcut it as ``$\trd^p \alpha\equiv 0$ on $A$''.
Notice that elements $\alpha$ of $\wca^j_p(M)$ for which $\trd^p \alpha$ vanishes on $A$ are the forms that satisfy for all $\beta\in \wca^{m-j-1}_{p', M\cup A}(M)$:
\begin{equation}\label{eq_lsp_j}
 <d\alpha,\beta>=(-1)^{j+1}<\alpha,d\beta>.
\end{equation}
We call $\trd^p\alpha$, the ``$p$-trace'' of $\alpha$. Theorems \ref{thm_trace_formes} and \ref{thm_residue_formula} show that in the case $p$ large, the $p$-trace coincides with a current induced by  restriction  of $\cc^\infty$ forms to the strata of a stratification of $\delta M$, providing a natural and more explicit notion of trace.

\subsection{$L^p$ Forms with boundary conditions.}\label{sect_forms_with_boundary conditions}
If $A\subset \delta M=\mba\setminus M$, we then set
\begin{equation}\label{eq_wmz}
\wca^j_p(M,A):=\{\alpha\in \wca_p^j(M): \trd^p \alpha\equiv 0 \mbox{ on $A$} \}. \end{equation}
 Observe that if   $\trd^p\alpha$ vanishes on $A$ then (\ref{eq_lsp_j}) immediately entails that so does $\trd^p d\alpha$, which means that $(\wca^j_p(M,A),d)_{j\in \N}$ is a cochain complex. The elements of $\wca^j_p(M,A)$ will be said to satisfy {\bf the Dirichlet condition on $A$}.

 For the moment, for the sake of simplicity of definitions, elements of $\wca^j_p(M,A)$ are thus forms  that vanish on $A$ ``in the weak sense'' (see (\ref{eq_lsp_j})). We will make it more explicit in sections \ref{sect_trace_forms} and \ref{sect_residues}, which will respectively establish that $\wca^j_p(M,A)$ gathers, when $p$ is large, the closure in $\wca^j_p(M)$ of the space of smooth forms that are zero in  restriction to $A$ (Theorem \ref{thm_trace_formes} and Corollary \ref{cor_trace_formes}) and,  when $p$ is close to  $1$, all the elements that vanish in the residue sense (Theorems \ref{thm_welldefined} and \ref{thm_residue_formula}, and Corollary \ref{cor_vanishing_res}).

Given  $A\subset \delta M$,  the  cohomology groups of  $(\wca^j_{p} (M,A) , d)_{j\in \N}$
are called the {\bf
	$L^p$ cohomology groups of $(M,A)$} and will be denoted by
$H^j _{p}(M,A)$ ($H^j_p(M)$ if $A=\emptyset$).
\end{subsection}

\begin{subsection}{The main theorems}
Given a definable subset $A\subset\delta M=\mba\setminus M$, we define the Lefschetz-Poincar\'e duality morphism as (recall that $M$ is oriented and $m=\dim M$):
$$\poi_{M,A}^j :\wca^j_{p} (M,A)\to \Hom (\wca^{m-j}_{p'}(M,\delta M \setminus  A),\R)  $$
\begin{equation}\label{eq_poima}\alpha\mapsto    [\beta \mapsto \poi_{M,A}^j(\alpha,\beta):=\int_{M} \alpha\wedge \beta],
\end{equation}
where $\Hom(\cdot,\cdot)$ stands for the corresponding space of homomorphisms of $\R$-vector spaces.
We will show in section \ref{sect_proof} for $A\subset\delta M$ definable:

\begin{thm}\label{thm_lefschetz_duality}(Lefschetz duality for $L^p$ cohomology) For $p\in [1,\infty]$ sufficiently large or sufficiently close to $1$, the $L^p$
  Lefschetz-Poincar\'e homomorphisms $$\poi_{M,A}^j : H^j_{p} (M,A)\to \Hom (H^{m-j}_{p'}(M,\delta M\setminus A),\R)  $$ 
   are isomorphisms for all $j$.
\end{thm}

 Given $j\le m$, $p\in [1,\infty]$, and $A$ as above, we denote by $ E^j_p(M,A)$ the space of {\bf $L^p$-exact $j$-forms} of $ \wca^j_{p}(M,A)$, i.e.:
 \begin{equation*}
 E^j_p(M,A):=\{d\omega :\omega \in \wca^{j-1}_{p}(M,A)\}.
 \end{equation*}
Our Lefschetz duality then immediately yields: 
 \begin{cor}\label{cor_image_fermee}(Poincar\'e inequality for $L^p$ forms) 
  For every $p\in [1,\infty]$ sufficiently large or sufficiently close to $1$, the space $E^j_p (M,A)$ is closed in $L^j_p(M)$ for each $j$, and consequently for every such $p$ there is a constant  $C$ such that for all $\omega\in \wca^j_p(M,A)$:
  \begin{equation}
   \inf_{\theta \in \wca^{j}_p(M,A), d\theta=0} ||\omega-\theta||_{L^p(M)} \le C||d\omega||_{L^p(M)}.
  \end{equation}
 \end{cor}
\begin{proof}
 By Theorem \ref{thm_lefschetz_duality}, $\alpha\in E^j_p (M,A)$ if and only if \begin{equation*}\label{eq_pf_PI}<\alpha,\beta>=0,
 \end{equation*}                                                                                    for all $\beta \in \wca^{m-j}_{p'}(M,\delta M \setminus A)$. This condition clearly defines a closed subspace, which yields the first part of the statement. The mapping $d: \wca^{j}_p(M,A) \to \wca^{j+1}_p(M,A)$, induced by the exterior differential operator, having closed image, it must give rise to an isomorphism of Banach spaces from  $\wca^{j}_p(M,A)/\ker d$ to $E^{j+1}_p (M,A)$, by the Inverse Mapping Theorem.
\end{proof}

We now come to our Hodge decomposition theorem. We denote by   $*$ the Hodge operator of $M$,
 and, given $\omega\in L_p^j(M)$, $p\in (1,\infty)$, $j\le m$,  we let for $x\in M$
  \begin{equation}\label{eq_*_p}
 *_p \omega(x):=\frac{|\omega(x)|^{p-2}}{||\omega||_{L^p(M)}^{p-2}}\cdot *\omega(x),
\end{equation}
with $*_p \omega(x) :=0$ if $\omega(x)=0$.
 It is worthy of notice that we have $*_2=*$ and that   $*_p*_{p'}$ is, up to sign, the identity map. Observe also that $||*_p\omega||_{L^{p'}(M)}=||\omega||_{L^p(M)}$.

As well-known, the $L^2$ coboundary operator is defined as $\delta \omega:=(-1)^{m(j-1)+1}*d*\omega$, if $\omega$ is a $j$-form on $M$. We shall  work with those forms $\omega$ for which $d *_p\omega$ is $L^{p'}$, so that  
\begin{equation}\label{eq_deltap}
	\delta_p \omega:=(-1)^{m(j-1)+1} *_{p'} d *_p\omega=(-1)^j*_p^{-1}d*_p \omega
	\end{equation}
    will be an $L^p$ form.
 We thus set
\begin{equation}\label{eq_wscpj}\wsc^p_j(M):=\{ \omega\in L^j_p(M):  d *_p\omega\in L^{m-j+1}_{p'}(M)\}=*_p\wca^{m-j}_{p'}(M),   \end{equation}
and more generally, given a subanalytic subset $A$ of $\delta M$ 
\begin{equation}\label{eq_wscpj_A}\wsc^p_j(M,A):=*_p\wca^{m-j}_{p'}(M,A).   \end{equation}
  The elements of $\wsc_j^p(M,A)$ will be said to satisfy the {\bf Neumann condition on $A$} (in the case of $1$-forms, if $p=2$, the required property is exactly the usual  Neumann condition). Clearly, $\delta_p \wsc_j^p(M,A)\subset \wsc_{j-1}^p(M,A)$.

   For simplicity, we denote by $\wca^j_p(M,A,\delta M \setminus  A)$ the set of forms of $\wca^j_p(M)$ satisfying the Dirichlet condition on $A$ and the Neumann condition on $\delta M\setminus  A$:
   $$\wca^j_p(M,A,\delta M \setminus  A):=\wca^j_p(M,A)\cap\wsc^p_j(M,\delta M\setminus A). $$

We will establish that, for $p$ large or close to $1$, each element of $\wca^j_p(M)$  can be decomposed as the sum of an exact form, a $p$-coexact form, and a $p$-harmonic form (i.e. $d\omega=0$, $\delta_p \omega=0$) with Dirichlet and Neumann conditions respectively (see section \ref{sect_hodge} for the proof):

\begin{thm}\label{thm_hodge}($L^p$ Hodge decomposition)
 For all $p\in (1,\infty)$ sufficiently large or sufficiently close to $1$ and each integer $j$, we have:
\begin{equation}\label{eq_hc}
 \wca^j_p(M)=E^j_p (M,A)\oplus cE^j_p (M,\delta M \setminus A) \oplus \mathscr{H}_p^j (M,A),
\end{equation}
where
$$\mathscr{H}_p^j (M,A):=\{ \,\omega \in \wca^j_p(M,A,\delta M\setminus  A):\;d\omega=0,\;\; \delta_p \omega=0\,\} $$
and
\begin{equation*}\label{eq_cE}
cE^j_p(M,\delta M \setminus A):=\{\omega\in \wca^j_p(M,\delta M \setminus  A): \exists \beta \in \wsc_{j+1}^p(M,\delta M\setminus  A), \,\, \omega=\delta_p \beta\}.
\end{equation*}
\end{thm}

This decomposition is orthogonal in the sense of the semi-inner product introduced in (\ref{eq_bracket}). Let us emphasize that since $\delta_p$ and $\Delta_p$ are nonlinear operators, the last two spaces that appear in (\ref{eq_hc}) are not vector subspaces. The sum is direct in the sense that pairwise intersections are  $
\{0\}$ and the resulting decomposition of a form of $  \wca^j_p(M)$ is unique. We indeed have (again, see section \ref{sect_hodge} for the proof): 

\begin{cor}\label{cor_hodge} For all $p\in (1,\infty)$ sufficiently large or sufficiently close to $1$ and each integer $j$, every cohomology  class of $H^j_p(M,A)$ contains a unique element of $\mathscr{H}_p^j (M,A).$ This element minimizes the $L^p$ norm in the class.  As a matter of fact, $\mathscr{H}_p^j (M,A)$ is a $\cc^0$ manifold of dimension $\dim H^j_p(M,A)<\infty$.
\end{cor}

That $\dim H^j_p(M,A)<\infty$  comes from our de Rham Theorem (see Corollary \ref{cor_finitely_generated}).  Corollary \ref{cor_hodge} shows that $*_p: \mathscr{H}_p^j (M,A)\to  \mathscr{H}_{p'}^{m-j} (M,\delta M\setminus A)$ as defined in (\ref{eq_*_p}) is an isomorphism that identifies the classes that are Lefschetz-Poincar\'e dual to each other.

\end{subsection}

\begin{section}{Some material from subanalytic geometry}\label{sect_sub}
\subsection{\L ojasiewicz's inequality and stratifications.} Many of our results about $L^p$ forms will be valid {\it for $p$ sufficiently large} (resp.  {\it  sufficiently close to $1$}), which means that there will be $p_0\in (1,\infty)$ such that the claimed fact will be true for all $p> p_0$ (resp. $p \in (1,p_0)$). We will specify if $p$ can be $1$ or infinite if it is not obvious from the context.  This number $p_0$ will often be provided by \L ojasiewicz's inequality, which is one of the main tools of subanalytic geometry. It originates in the fundamental work of S. \L ojasiewicz \cite{lojdiv}, who established this inequality in order to answer a problem  about distribution theory.  We shall make use of the following version whose proof can be found in  \cite[Theorem $2.2.5$]{livre}:


 \begin{pro}\label{pro_lojasiewicz_inequality}(\L ojasiewicz's inequality)
Let $f$ and $g$ be two globally subanalytic functions on a  globally subanalytic set $A$ with $\sup\limits_{x\in A} |f(x)|<\infty$. Assume that
 $\lim\limits_{t \to 0} f(\gamma(t))=0$ for every
globally subanalytic arc $\gamma:(0,\ep) \to A$ satisfying $\lim\limits_{t \to 0} g(\gamma(t))=0$.
Then there exist $\nu \in \N$ and $C \in \R$ such that for any $x \in A$:
$$|f(x)|^\nu \leq C|g(x)|.$$
\end{pro}

Our smoothing process of differential forms near singularities will rely on stratification theory (in section \ref{sect_density}). We now introduce the needed material on this topic.
\begin{dfn}\label{dfn_stratifications}
	A {\bf
		stratification} of a subset of $ \R^n$ is a finite partition of it into
	definable $\cc^\infty$ submanifolds of $\R^n$, called {\bf strata}. A stratification is {\bf compatible} with a set if this set is the union of some strata. A {\bf refinement} of a stratification $\Sigma$ is a stratification $\Sigma'$ compatible with the strata of $\Sigma$.
	
	   If $\Sigma$ is a stratification of a set $X$, we will sometimes say that $(X,\Sigma)$ is a {\bf stratified space}.  A {\bf subspace} of $(X,\Sigma)$ is a set $A\subset X$ which is the union of some strata endowed with the induced stratification, that we will then denote by $\Sigma_A$.
\end{dfn}

As well-known \cite{livre},  given a definable set $X$, the set $X_{reg}$ of points of $X$ at which this set is an analytic manifold (the set of {\bf regular points of $X$})  is definable and  dense in $X$ \cite{tamm, kureg, livre}, which makes it possible to construct stratifications of a given definable set.  One can  produce  stratifications satisfying various regularity conditions \cite{livre} such as Whitney's ones or of Lipschitz type \cite{livre, ngv, mos, par}, like the following one.
\begin{dfn}\label{dfn_trivial}
	A  stratification $\Sigma$ of a set $X$ is {\bf locally definably bi-Lipschitz trivial} if for every $S\in \Sigma$, there are an open neighborhood $V_S$ of $S$ in $X$ and a smooth  definable retraction $\pi_S:V_S\to S$  such that every $x_0\in S$ has an open neighborhood $W$ in $S$ for which there is a definable bi-Lipschitz homeomorphism  $$\Lambda:\pi_S^{-1}(W)\to \pi_S^{-1}(x_0) \times W, $$ satisfying:
\begin{enumerate}[(i)]
	\item  $\pi_S(\Lambda^{-1}(x,y))= y$, for all $(x,y)\in \pi_S^{-1}(x_0)\times  W$.
	\item   $\Sigma_{x_0}:=\{  \pi_S^{-1}(x_0)\cap Y:Y\in \Sigma\} $ is a stratification of $ \pi_S^{-1}(x_0)$,  and $\Lambda(\pi_S^{-1}(W)\cap Y)=(\pi_S^{-1}(x_0)\cap Y)\times W$, for all $Y\in \Sigma$.
\end{enumerate}
\end{dfn}


\begin{thm}\label{thm_existence_stratifications}\cite[Corollary $3.2.16$]{livre}
	Every definable set $X$ admits a locally definably bi-Lipschitz trivial stratification. This stratification can be assumed to be compatible with finitely many given definable subsets of $X$.
\end{thm}

	\begin{rem}\label{rem_M_stratum} We can always assume that our definably bi-Lipschitz trivial stratifications of $\mba$ are such that the only stratum of maximal dimension is $M$, possibly taking together all the strata included in $M$ (see \cite[Remark $1.6$]{lprime}). \end{rem}

      \begin{rem}\label{rem_frontier}
      It directly follows from their definition that  locally definably bi-Lipschitz trivial stratifications  satisfy the frontier condition in the following sense: if $S$ and $S'$ are two distinct strata such that $\adh S\cap S'\ne \emptyset$ then $S'\subset \delta S$. For such a couple of strata, we will write $S'\prec S$. It is worthy of notice that if $(X,\Sigma)$ is a stratified space   then the maximal strata of $\Sigma$ for this partial strict order relation constitute an open dense subset of $X$.
      \end{rem}

\subsection{Stratified differential forms.}\label{sect_stokes_formula}
Stratified forms were introduced in \cite{stokes, livre} in order to generalize Stokes' formula on bounded subanalytic sets, possibly singular. We give definitions in this section and then recall this generalized Stokes' formula  which will be of service in section \ref{sect_density}. 

\begin{dfn}\label{dfn_stratified_form}
	Let  $(X,\Sigma)$ be a  stratified space. 	
	A {\bf stratified differential $0$-form on $(X,\Sigma)$}\index{stratified form} is a collection of functions $\omega_S:S\to \R$, $S \in \Sigma$, that glue together into a continuous function on $X$.  	
	A {\bf stratified differential $k$-form on $(X,\Sigma)$}, $k>0$, is a collection $(\omega_S)_{S \in \Sigma}$  where, for every $S$, $\omega_S$ is a continuous differential $k$-form on $S$ such that for any $(x_i ,\xi_i)\in \otimes^k  TS$, with $x_i$ tending to $x\in S'\in \Sigma$ and $\xi_i$ tending to $\xi \in \otimes^k  T_xS'$, we have
	\begin{equation}\label{eq_stratified_form}\lim \omega_S(x_i)\xi_i=\omega_{S'}(x)\xi.\end{equation}
	The {\bf support of a stratified form}\index{support of a stratified form} $\omega$ on $(X,\Sigma)$ is the closure in $X$ of the  set
	$$\bigcup_{S\in \Sigma}\{x\in S: \omega_S(x)\ne 0 \}.$$
	When this closure is compact, $\omega$ is said to be {\bf compactly supported}\index{compactly supported}.
	We say that a stratified form $\omega=(\omega_S)_{S \in \Sigma}$ is {\bf differentiable}  if $\omega_S$ is $\cc^1$ for every $S\in \Sigma$ and if $d\omega:=(d\omega_S)_{S\in \Sigma}$ is a stratified form.
\end{dfn}

	Given a stratified $j$-form $(\omega,\Sigma)$ on a stratified space $(X,\Sigma)$ and a definable subset $Y$ of $X$ of dimension not greater than $j$, we then can define the integral of $\omega$ on $Y$ as follows. Take a refinement $\Sigma'$ of $\Sigma$ compatible with $Y_{reg}$, denote by $\Sigma''$ the stratification of $Y_{reg}$ induced by $\Sigma'$, and set
	$$\int_Y\omega:=\sum_{S\in \Sigma'', \dim S=j}\int_S \omega_S  $$
	(here we assume that $Y_{reg}$ is oriented and that every stratum $S$ is endowed with the induced orientation). As explained in \cite{stokes,livre}, this integral
	is independent of the chosen stratifications.
	
	\begin{dfn}\label{dfn_normal}
We say that a subanalytic set $A$ is {\bf connected at $x\in \overline{A}$}\index{connected at $x$} if  $\bou(x,\ep)\cap A$ is connected for all  $\ep>0$ small enough, and that it is {\bf connected along $Z\subset \adh A$} if it is connected at each point of $Z$.
We say that $A$ is {\bf normal} if it is connected at each $x\in \overline{A}$, and that it is {\bf weakly normal} if it is connected along  some definable set $E\subset \delta A$ satisfying $\dim (\delta A\setminus E)\le \dim A-2$.  
\end{dfn}
We then have the following Stokes' formula that applies to any weakly normal definable manifold, with possibly singularities within its closure.

\begin{thm}\label{thm_stokes_leaves}
	Let $P$ be an oriented $k$-dimensional weakly normal definable $\cc^0$ manifold (without boundary) and let $\Sigma$ be a stratification of $\adh{P}$. For any compactly supported differentiable stratified $(k-1)$-form  $\omega$ on $(\adh{P},\Sigma)$,   we have:
	\begin{equation}\label{eq_stokes_stratified_leaf}
		\int_P d\omega =\int_{\delta P} \omega,
	\end{equation}
	where $(\delta P)_{reg}$ is endowed with the induced orientation.
\end{thm}

We will say that a differential form $\beta$ on $M$ is {\bf stratifiable} if there is a stratified form  $(\omega, \Sigma)$ such that $\omega=\beta$ on the strata of dimension $\dim M$ (which constitute a dense subset). 
Stratifiable forms arise naturally when we pull-back a smooth differential form $\omega$ on a manifold $M$ under a definable Lipschitz mapping $h:M'\to M$, $M'$ definable manifold. Since definable Lipschitz mappings can be horizontally $\cc^1$ stratified \cite[Proposition $2.6.12$]{livre}, the pull-back $h^*\omega$ (which is only defined almost everywhere when $h$ is not smooth) is then stratifiable (see \cite{livre} right after Definition $4.6.3$ for more). By (\ref{eq_stokes_stratified_leaf}), $h^*d\omega=dh^*\omega$ if $\omega \in \cc^{j,\infty}(M)$. In particular,   a definable bi-Lipschitz mapping $h:M'\to M$ induces (by density of smooth forms),  linear isomorphisms  \begin{equation}\label{eq_pullback}h^*: \wca^j_p(M)\to \wca^j_p(M'), \quad \omega \mapsto h^*\omega,\end{equation}
that yield  isomorphisms in cohomology.

Note also that if $h:[0,1]\times M \to M$ is a  definable Lipschitz homotopy and if we set  for $x\in M$ and  $\omega\in \cc^{j,\infty}(M)$:
$$\hn\omega(x):=\int_0 ^1 (h^*\omega)_{\pa_t}(t,x)dt,$$
where  $\pa_t$ denotes the constant vector field $(1,0)$ on $[0,1]\times M$,
then (\ref{eq_stokes_stratified_leaf}) yields
\begin{equation}\label{eq_chain_homotopy}
	d \hn \omega+\hn d\omega=h_1^* \omega -h_0 ^*\omega,
\end{equation}
where $h_i:M\to M$, $i=0,1$, is defined by $h_i(x)=h(i,x)$.

\subsection{Normalizations.}In the case where the underlying manifold is not normal, we can ``normalize'' it. This will be of service in section \ref{sect_lp_coh}.

\begin{dfn}\label{dfn_normalization}
	{\bf A $\cc^\infty$ normalization of $M$ }
	is  a definable $\cc^\infty$ diffeomorphism $h: \mc\to M$ satisfying $\sup_{x\in \mc} |D_x h|<\infty$ and  $\sup_{x\in M} |D_x h^{-1}|<\infty$,  with $\mc$ normal $\cc^\infty$ submanifold of $\R^k$, for some $k$. We will sometimes say that $\mc$ is a normalization of $M$.
\end{dfn}

The following proposition gathers Propositions $3.3$ and $3.4$ of \cite{trace}, yielding existence and uniqueness of $\cc^\infty$ normalizations.

\begin{pro}\label{pro_normal_existence}\begin{enumerate}
		\item\label{item_normal_existence}
		Every bounded definable $\cc^\infty$ manifold admits a $\cc^\infty$ normalization.
		\item  \label{item_unique}   Every $\cc^\infty$ normalization  $h: \mc\to M$  extends continuously to a mapping from $\adh{\mc}$ to $\mba$.
		Moreover, if $h_1:\mc_1 \to M$ and $h_2:\mc_2 \to M$ are two $\cc^\infty$ normalizations of $M$, then $h_2^{-1} h_1$  extends to a homeomorphism between $\adh{\mc_1}$ and $\adh{\mc_2}$.                             \end{enumerate}
\end{pro}

\subsection{Trace operators of subanalytic manifolds.}\label{sect_trace_operators}	
As well-known \cite{adams}, Sobolev spaces of functions are very well-behaved on domains that have metrically conical boundary. Part of the theory was extended to subanalytic manifolds \cite{poincwirt, trace, poincfried, lprime, gupel}, which  may be cuspidal. We briefly mention some results of \cite{trace, lprime, gupel}  in this section about the case  ``$p$ large''  that we will generalize to differential $j$-forms in sections \ref{sect_density}  and \ref{sect_further}. Since functions are $0$-forms, this corresponds to the case $j=0$, on which we will rely when  arguing by induction on $j$ in section \ref{sect_proof_dense_formes}.
\begin{thm}\label{thm_trace}\cite{trace}
	Assume that $M$ is normal and let $A$ be any subanalytic subset of $\delta M$. For all $p\in [1,\infty)$ sufficiently large, we have:
	\begin{enumerate}[(i)]
		\item
		$\cc^\infty(\adh{M})$ is dense in $W^{1,p}(M)$.
		\item The linear operator
		\begin{equation*}\label{trace}
			\cc^\infty(\adh{M})\ni \varphi \mapsto \varphi_{|A}\in L^p(A,\hn^k), \qquad k:=\dim A,
		\end{equation*}
		is continuous in the norm $||\cdot ||_{W^{1,p}(M)}$ and thus extends to a mapping $\tra_A:W^{1,p}(M)\to L^p(A,\hn^k)$.
		\item If  $\st$ is a stratification of $A$, then $\cc^\infty_{\adh{M}\setminus\adh{A}}(\mba)$ is a dense subspace of
		$\bigcap\limits_{Y\in\st}\ker \tra_Y$.
\end{enumerate}\end{thm}
It was shown in \cite{lprime} that for $p$ large, elements of $W^{1,p}(M)$ are always H\"older continuous with respect to the inner metric of $M$, which is the metric on $M$  given by the length of the shortest path joining two given points.  This entails that a function $u \in W^{1,p}(M)$ may only have $\cfr_M(\xo)$ asymptotic values at a point $\xo$ of $\delta M$, where $\cfr_M(\xo)$ is the number of connected components of $M\cap \bou(\xo,\ep)$, $\ep>0$ small (one asymptotic value per connected component, see \cite[Corollary $1.3$]{kp} or \cite[Proposition 3.1.23]{livre}). It makes it possible to define an operator
\begin{equation}\label{eq_trace_fn_plarge}
	\tra: W^{1,p}(M)\to L^\infty(\delta M)^l,
\end{equation}
where $l=\max \cfr_M(\xo)$ (see \cite[section $5.2$]{gupel} for more details).

\subsection{Local conic structure of subanalytic sets.}	The proof of Theorem \ref{thm_trace} given in \cite{trace} relies on a theorem achieved in \cite{gvpoincare} (see also \cite[Theorem $3.4.1$]{livre} for full details). Since this result will also be needed to construct our homotopy operators in section \ref{sect_some_local_operators}, we recall it in this section (Theorem \ref{thm_local_conic_structure} below) and derive some consequences in the next one. 

In this theorem, $x_0* (\sph(x_0,\ep)\cap X)$ stands for the cone over $\sph(x_0,\ep)\cap X$ with vertex at $x_0$.

	\begin{thm}[Lipschitz Conic Structure]\label{thm_local_conic_structure}
		Let  $X\subset \R^n$ be subanalytic and $x_0\in X $. 
		For $\ep>0$ small enough, there exists a Lipschitz subanalytic homeomorphism
		$$H: x_0* (\sph(x_0,\ep)\cap X)\to  \Bb(x_0,\ep) \cap X,$$  
		satisfying $H_{| \sph(x_0,\ep)\cap X}=Id$, preserving the distance to $x_0$, and having the following metric properties:
		\begin{enumerate}[(i)] 
			\item\label{item_H_bi}     The natural retraction by deformation onto $x_0$ $$r:[0,1]\times  \Bb(x_0,\ep)\cap X \to \Bb(x_0,\ep)\cap X,$$ defined by $$r(s,x):=H(sH^{-1}(x)+(1-s)x_0),$$ is Lipschitz.   
			Indeed, there is a constant $C$ such that  for every fixed $s\in [0,1]$, the mapping $r_s$ defined by $x\mapsto r_s(x):=r(s,x)$, is $Cs$-Lipschitz.
			\item \label{item_r_bi}  For each $\delta>0$,
			the restriction of $H^{-1}$ to $\{x\in X:\delta \le |x-x_0|\le \ep\}$ is Lipschitz and, for each $s\in (0,1]$, the map  $r_s^{-1}:\Bb(x_0,s\ep) \cap X\to \Bb(x_0,\ep) \cap X$ is Lipschitz. 
		\end{enumerate}
	\end{thm}


  This theorem is the Lipschitz counterpart of the $\cc^0$ conic structure of sets definable in an o-minimal structure (see \cite[Theorem $4.10$]{costeomin}). It can be proved for sets definable in a polynomially bounded o-minimal structure.
 \begin{rem}\label{rem_preserve}
	It follows from \cite[Remark $3.4.2$ or  Corollary $3.4.5$]{livre} that we can require $r_s$ to preserve some given definable subset-germs of $X$ for all $s\in (0,1)$.
\end{rem}

\subsection{The retractions $r_s$ and $R_t$.}\label{sect_the_retraction_r_s_and_R_t}
We now wish to explain in which setting we will apply the above theorem and then mention a few facts that were established in \cite{trace, lprime}.
Let us assume that $0_{\R^n}\in\mba$ and 
set for $\eta>0$: \begin{equation}\label{eq_metaneta}M^{\eta}:=\bou(0_{\R^n},\eta)\cap M\;\; \et\;\; N^{\eta}:=\sph(0_{\R^n},\eta)\cap M.\end{equation}

 Apply Theorem \ref{thm_local_conic_structure} with $X=M\cup\{0_{\R^n}\}$ and $x_0=0_{\R^n}$. Fix $\ep>0$ sufficiently small for the statement of the theorem to hold and for $N^\ep$ to be a $\cc^\infty$ manifold, and  let $r$ as well as $H$ denote the mappings provided by this theorem. We will assume for convenience $\ep<\frac{1}{2}$. Observe  that since $r$ is subanalytic, it is $\cc^\infty$ almost everywhere.

  Since $r_s$ is bi-Lipschitz for every $s>0$, $\jac\, r_s$ can only tend to zero if $s$ is itself going to zero (see \cite[Remark $1.5$]{trace}). Hence, by \L ojasiewicz's inequality (Proposition \ref{pro_lojasiewicz_inequality}),
there are  $\nu\in \N$ and  $\kappa>0$ such that for all $s\in (0,1)$ we have for almost all $x\in M^\ep$:
\begin{equation}\label{eq_jacr_s}
 \jac \,r_s(x) \ge \frac{s^{\nu}}{\kappa}.  
\end{equation}

The definition of $r_s(x)$ that we gave in Theorem \ref{thm_local_conic_structure} actually makes sense for each $(s,x)\in [0,\infty)\times \mep$ satisfying $s\le \frac{\ep}{|x|}$. We therefore can define a mapping $R:P \to M$, where $P=\{(t,x)\in \R\times M: 1\le t\le \frac{\ep}{|x|} \}$, by 
$$R(t,x):=H(tH^{-1}(x)).$$
We will denote by $r^\eta_s:N^\eta \to N^{s\eta}$ and $R^\eta_t: N^{\eta} \to N^{t\eta}$ the respective restrictions of $r(s,\cdot)$ and
 $R(t,\cdot)$.
We do not use the notation $r(s,x)$ when $s>1$,  since the properties of $R$ will be very different from the properties of $r$ listed in  Theorem \ref{thm_local_conic_structure}. 
The mapping $R$ should rather be regarded as an inverse,  $R_t^\eta:N^\eta\to N^{t\eta}$   being the inverse of  $r_{1/t}^{t\eta}: N^{t\eta}\to N^{\eta}$.  In particular, by (\ref{eq_jacr_s}), we have
\begin{equation}\label{eq_jacRt}
	\jac R_t \le \kappa\, t^\nu.
\end{equation}
It is worthy of notice that since $r$ has bounded derivative, (\ref{eq_jacr_s}) holds for $r_s^\eta$, and hence (\ref{eq_jacRt}) holds to $R_t^\eta$, i.e.
\begin{equation}\label{eq_jac_rReta}
 \jac \,r_s^\eta(x) \ge \frac{s^{\nu}}{\kappa}\et 	\jac R_t^\eta \le \kappa\, t^\nu,
\end{equation} 
for some possibly bigger constant $\kappa$ (independent of $\eta<\ep$).

\subsection{A few facts about the conic structure.}\label{sect_key_facts}  There is a constant $C$ such that:

\begin{enumerate}
\item For almost all $x\in M^\ep$ we have for almost all $s\in (0,1)$,  \begin{equation}\label{eq_der_r_s}
       \left|\frac{\pa r}{\pa s}(s,x)\right|\le C|x|,
      \end{equation}
     and,  for almost all $t\in [1,\frac{\ep}{|x|}]$,
\begin{equation}\label{eq_par} 
 \left|\frac{\pa R}{\pa t}(t,x)\right| \le C|x|.
\end{equation}
       \item For each $v\in L^{p}(M^\ep)$,  $p\in [1,\infty)$, we have for all $\eta\in (0,\ep]$, \begin{equation}\label{eq_coarea_sph}
    \left(\int_0 ^\eta ||v||_{L^p (N^\zeta)}^p d\zeta \right)^{1/p}  \le    ||v||_{L^p (M^\eta)} \leq C \left(\int_0 ^\eta ||v||_{L^p (N^\zeta)}^p d\zeta \right)^{1/p}.
       \end{equation}
\item\label{item_lambda} The angle between $\frac{\pa r}{\pa s}(s,x)$ (resp. $\frac{\pa R}{\pa t}(t,x)$) and the tangent space to $N^{s|x|}$ (resp. $N^{t|x|}$) at $r_s(x)$ (resp. $R_t(x)$) is bounded below away from zero independently of almost every $s$ (resp. $t$) and $x$.
\end{enumerate}
\begin{proof} Inequalities (\ref{eq_der_r_s}) and (\ref{eq_coarea_sph}) were proved in \cite[Inequalities (2.2) and (2.3)]{trace}, while (\ref{eq_par}) is equivalent to Inequality $(3.2)$ of \cite{lprime}.  To show 
 (\ref{item_lambda}), notice that, if $\rho:M\to \R$ stands for the restriction of the distance function to the origin, then, as $\rho(R(t,x))=t |x|$, we must have:
 \begin{equation}\label{eq_dero}
 	<\pa \rho(R_t(x)), \frac{\pa R}{\pa t}(t,x)>= |x|\overset{(\ref{eq_par})}\ge \frac{ \left|\frac{\pa R}{\pa t}(t,x)\right|}{C} ,
 \end{equation} 
yielding the needed fact for $R$. The same applies to $r$, using  (\ref{eq_der_r_s}).
\end{proof}

\end{section}


\begin{section}{Local homotopy operators for $L^p$ forms}\label{sect_some_local_operators}
Throughout this section, we will work locally near the origin, assuming $\orn \in \mba$. 	We are going to define some homotopy operators   for differential forms (see (\ref{eq_rfr}) and (\ref{eq_Rfr})) and show that they preserve $\wca_p^j$ forms vanishing on a set (Proposition \ref{pro_r_preserve_vanishing}). This will be needed to establish our $L^p$ Poincar\'e Lemma (Lemma \ref{lem_poinc}) as well as our Lefschetz duality  (section \ref{sect_proof}).   

 Given a vector field $\xi$ and a $j$-form $\omega$ on  $M$, we denote by $\omega_\xi$ the differential $(j-1)$-form defined for every $x\in M$ and $\zeta \in \otimes^{j-1} T_x M$ by \begin{equation}\label{eq_om_xi}\omega_\xi(x)\zeta:=\omega(x)\xi(x) \otimes \zeta,\end{equation}
 with the convention that this form is zero if $j=0$. We will sometimes use this notation when $\xi$ is a vector (instead of a vector field) regarding $\xi$ as a constant vector field.

\begin{subsection}{The operators $\rfr$ and $\Rfr$} Let $M^\eta$, $N^\eta$, $\ep$, $r$, and $R$, be as in subsection  \ref{sect_the_retraction_r_s_and_R_t} (we recall that we assume $\orn \in \mba$ throughout section \ref{sect_some_local_operators}).
 The retractions $r$ and $R$  give rise to operators on $L^p$ differential forms as follows. For $\omega\in L^j_p(\mep)$, $p\in [1,\infty]$, $j\ge 0$, and almost every $x\in \mep$, let
\begin{equation}\label{eq_rfr}
 \rfr\omega(x) :=\int_0 ^1 (r^* \omega)_{\pa_s}(s,x)\,ds,
\end{equation}
as well as
\begin{equation}\label{eq_Rfr}
 \Rfr\omega(x) :=\int^1 _\frac{\ep}{|x|} (R^* \omega)_{\pa_t}(t,x)\,dt ,
\end{equation}
where $\pa_s$ and $\pa_t$ respectively stand for the gradient of $[0,1]\times \mep\ni (s,0)\mapsto s$  and $[1,\infty) \times \mep\ni (t,x)\mapsto t$.
In particular, $\rfr\omega=\Rfr \omega=0$ if $j=0$.

It follows from Theorem \ref{thm_local_conic_structure} that $R_t$ is bi-Lipschitz for every $t\ge 1$ with a constant that stays bounded away from infinity if $t$ remains bounded. Hence, by (\ref{eq_par}), $(R^*\omega)_{\pa_t}$ is $L^1$ on $[1,\frac{\ep}{|x|} ]$ for almost every $x\in \mep$,  for all $\omega\in L^j_p(\mep)$, $p\ge 1$, which  means that $\Rfr\omega$ is well-defined for all such $\omega$. We will however mainly use it in the study of the case ``$p$ is close to $1$'' (see Proposition \ref{pro_borne_R}).
 Formula (\ref{eq_rfr}) needs the function $s\mapsto |(r^*\omega)_{\pa_s}(s,x)|$ to be  $L^1([0,1])$ (for
 almost every $x$), which is true  for  $\omega\in L^j_p(\mep)$ provided $p$ is sufficiently large:
\begin{lem}\label{lem_r_well_dfnd}
For $p\in [1,\infty)$ sufficiently large, if $\omega\in L^j_p(\mep)$, $j\ge 1$, then $|(r^*\omega)_{\pa_s}(s,x)|\in L^1( [0,1] \times\mep)$, and hence $\rfr\omega$ is well-defined.  Moreover, if $j=m$ then this is true for all $p\in [1,\infty)$.
\end{lem}
\begin{proof}
Since $r_s$ is  Lipschitz, thanks to Minkowski's integral inequality, we have for every $\omega\in L^j_p(\mep)$, $j\le m$, $p\in [1,\infty)$, and almost every $x\in \mep$:
\begin{eqnarray*}
  || \int_0 ^1|r^*\omega|(s,x)ds||_{L^p(\mep)} &\le&  \int_0 ^1 ||r^*\omega(s,x)||_{L^p(\mep)}ds\lesssim \int_0 ^1 ||\omega\circ r_s||_{L^p(\mep)}ds\end{eqnarray*}
$$\overset{(\ref{eq_jacr_s})}\lesssim\int_0 ^1 s^{-\frac{\nu}{p}}\left( \int_{\mep} |\omega\circ r_s|^p\cdot \jac r_s\right)^{\frac{1}{p}}ds=\int_0 ^1 s^{-\frac{\nu}{p}}||\omega||_{L^p(M^{s\ep})}ds,$$
which, by H\"older's inequality, gives if $p> \nu+1$ (for such $p$, the function $s^{-\frac{\nu}{p}} $ is  $L^{p'}$)
$$ || \int_0 ^1|r^*\omega|(s,x)ds||_{L^p(\mep)}  \le ||s^{-\frac{\nu}{p}}||_{L^{p'}([0,1])}\cdot\left(\int_0^1 ||\omega||_{L^p(M^{s\ep})}^pds\right)^\frac{1}{p}\lesssim  ||\omega||_{L^p(M^{\ep})}.$$
 This implies that
$r^*\omega\in L^p((0,1)\times \mep)$, which, in virtue of H\"older's inequality,  means that  it is an $L^1$ form.

We now prove the last statement, i.e. we assume $j=m$. Fix $\omega\in L^m_p(\mep)$. 
Let $\xi_s(x):=\frac{\frac{\pa r}{\pa s}(s,x)}{| \frac{\pa r}{\pa s}(s,x)|}$ and let $\xi_s^*$ be the $1$-form defined by this vector field.   It follows from
 (\ref{item_lambda}) of section \ref{sect_key_facts} that if $\zeta_1,\dots,\zeta_{m-1}$ is an orthonormal basis of $T_{r_s(x)} N^{s\eta}$ (for some $(s,x)\in (0,1)\times N^\eta$, $\eta<\ep$) then the distance between $\xi_s^*$ and the space $Z$ generated by $\zeta_1^*,\dots,\zeta_{m-1}^*$ is bounded below away from zero by some constant independent of $(s,x)$ and $\eta$. It makes it possible to decompose the $m$-form $\omega(r_s(x))$ as
 \begin{equation}\label{eq_xi}\omega(r_s(x))= \xi_s^*(x)\wedge\omega_{s}'(r_s(x)),\end{equation}
   with for every $s\le 1$,   $ \omega_{s}'$  $(m-1)$-form on $M^{s\eta}$  satisfying
\begin{equation}\label{eq_ome12r} 
\omega_{s,\pa \rho}'(r_s(x))=0,
\end{equation}
 where  $\rho:\mep\to \R$ is the restriction of the distance to the origin (i.e. $ \omega_{s}'\in \wedge^{m-1} Z$), and \begin{equation}\label{eq_omega2r}|\omega_{s}'(r_s(x))| \lesssim |\omega(r_s(x))|, \end{equation}
where the constant, which only depends on the bound given by (\ref{item_lambda}) of section \ref{sect_key_facts}, is  independent of both $\omega$ and $s$.
For almost every $x\in \mep$ and $s\in (0,1)$, we now can write if	$\eta:=|x|$ (we recall that $r_s^\eta:N^\eta\to N^{s\eta}$ is the mapping induced by $r_s$):
		\begin{eqnarray*}	\left|(r^*\omega)_{\pa s}(s,x)\right|
		 &\overset{(\ref{eq_xi})}\lesssim& \left| \frac{\pa r}{\pa s}(s,x) \right|\cdot  \left| r_s^{*}\omega_{s}'(x)\right|
		\overset{(\ref{eq_der_r_s})} \lesssim	\eta\cdot  \left| r_s^{*}\omega_{s}'(x)\right|\\
	&\overset{(\ref{eq_ome12r})}	\lesssim& 	\eta\cdot 	   \left| r_s^{\eta*}\omega_{s}'(x)\right|
				 	=	\eta\cdot \left|\omega_{s}'(r_s(x))\right|\cdot \jac r_s^{\eta}(x), \end{eqnarray*}
				 	and therefore, by (\ref{eq_omega2r}),
				 	\begin{eqnarray}\label{eq_estimate_r}
				\left|(r^*\omega)_{\pa s}(s,x)\right| 	&\lesssim&	  	\eta\cdot \left|\omega(r_s(x))\right|\cdot \jac r_s^{\eta}(x).
		 \end{eqnarray}
		Integrating with respect to $(s,x)\in [0,1]\times  N^\eta$, we get:
	$$ \int_{x\in N^\eta}	\int_0 ^1 \left|(r^*\omega)_{\pa s}(s,x)\right|ds 
\lesssim 	\eta \int_{x\in N^\eta}	\int_0 ^1 \left|\omega(r_s(x))\right|\cdot \jac r_s^{\eta}(x)ds =	\eta\int_0 ^1||\omega ||_{L^1(N^{s\eta})}ds. $$
Integrating with respect to $\eta$, by   (\ref{eq_coarea_sph}), we conclude that
$$ ||(r^*\omega)_{\pa s}||_{L^1([0,1]\times\mep)} \lesssim \int_0^\ep\eta\int_0 ^1||\omega ||_{L^1(N^{s\eta})}ds\, d\eta \overset{(\ref{eq_coarea_sph})}{\le} \int_0^\ep||\omega ||_{L^1(M^{\eta})}d\eta, $$
	 which is finite when $\omega$ is at least $L^1$.
 \end{proof}
\end{subsection}

\begin{subsection}{$L^p$ bounds for $\rfr$ and $\Rfr$.}\label{sect_lp_bounds} It is worthy of notice that the just above computations have also shown that $\rfr$ is continuous in the $L^p$ norm for $p> \nu+1$. The next proposition provides a related estimate (see also Remark \ref{rem_borne_Rr} below).   We will then give  analogous bounds for $\Rfr$ in the case where $p$ is close to $1$ (Proposition \ref{pro_borne_R}). We  will derive from these bounds that $\Rfr$ and $\rfr$  induce continuous homotopy operators on $\wca^j_p$ forms for $p$ sufficiently large or close to $1$ (Propositions \ref{pro_rR_def_sur_wj} and \ref{pro_r_preserve_vanishing}).
	 \begin{pro}\label{pro_borne_r}
 There is a constant $C$ such that for any $p\in [1,\infty)$ large enough  we have for each $\omega\in L^j_p(\mep)$, $j\ge 0$,  and  every $\eta\le \ep$:
\begin{equation}\label{eq_r_borne}
||\rfr \omega||_{L^p(N^\eta)}\le C \,\eta^\frac{p-1}{p} \,||\omega||_{L^p(M^\eta)} .\end{equation}
Moreover, when $j=m$, this inequality is true for all $p\in [1,\infty)$.
\end{pro}

\begin{proof} Given $\omega\in L^j_p (M^\ep)$, $j\ge 1$, we have for $\eta\le \ep$:
\begin{eqnarray*}\label{eq_mk1}
  \nonumber ||\rfr \omega||_{L^p(N^\eta)} &= & \left(\int_{x\in N^\eta}\left| \int_0 ^1 \left(r^*\omega\right)_{\pa_s}(s,x) ds\right|^p \right)^{1/p} \\
   \nonumber &\le&\int_0 ^1 \left(\int_{x\in N^\eta}\left| (r^*\omega)_{\pa_s} (s,x)\right|^p \right)^{1/p}ds \quad \mbox{(by Minkowski's  inequality)}\\
    \nonumber &\overset{(\ref{eq_der_r_s})}\lesssim&\eta \int_0 ^1\left( \int_{N^\eta}\left| \omega\circ r_s \right|^p \right)^{1/p}ds \\
    \nonumber &\overset{(\ref{eq_jac_rReta})}\lesssim&\eta \int_0 ^1  s^{-\frac{\nu}{p}}\left( \int_{N^\eta}\left| \omega\circ r_s \right|^p \cdot \jac r_s ^\eta \right)^{1/p}ds \\
    \nonumber &\lesssim& \eta \left( \int_0 ^1\int_{N^\eta}\left| \omega\circ r_s \right|^p \cdot \jac r_s ^\eta \,ds \right)^{1/p} \quad \mbox{(by H\"older, for $\frac{\nu}{p-1}\le \frac{1}{2}$)}\\
        \nonumber &\overset{(\ref{eq_coarea_sph})}\le& \eta^\frac{p-1}{p} ||\omega||_{L^p(M^{\eta})}.
  \end{eqnarray*}
 
To show that this estimate holds  for {\it all} $p$ in the case $j=m$, take $\omega\in L^m_p (M^\ep)$  and write 
     \begin{eqnarray*}\label{eq_mk2}
  \nonumber ||\rfr \omega||_{L^p(N^\eta)}   &= & \left(\int_{x\in N^\eta}\left| \int_0 ^1 (r^*\omega)_{\pa_s}(s,x) ds\right|^p \right)^{1/p} \\
   &\le& \int_0 ^1\left( \int_{x\in N^\eta}| (r^*\omega)_{\pa_s}(s,x) |^p \right)^{1/p}ds \quad \mbox{(by Minkowski's inequality)}\\
      &\overset{(\ref{eq_estimate_r})}\lesssim&\eta \int_0 ^1\left( \int_{N^\eta}| \omega\circ r_s|^p\cdot ( \jac r_s^\eta )^p \right)^{1/p}ds \\
     \nonumber &\le& \eta \left( \int_0 ^1\int_{N^\eta}\left| \omega\circ r_s \right|^p \cdot\left( \jac r_s^\eta  \right)^p  ds \right)^{1/p} \quad \mbox{(by H\"older)}\\
        \nonumber &\overset{(\ref{eq_coarea_sph})}\lesssim& \eta^\frac{p-1}{p} ||\omega||_{L^p(M^{\eta})} \quad \mbox{(since $ (\jac r_s^\eta)^p\lesssim \jac r_s^\eta$)}.
  \end{eqnarray*}
  \end{proof}
  \begin{rem}\label{reb_est_neta_infty}
  	In the case $p=\infty$, we have for any $\omega\in L^j_\infty(\mep)$, $j\le m$, and $\eta\le \ep$:
  	\begin{equation}\label{eq_neta_infty}
  		||\rfr\omega||_{L^\infty(N^\eta)} \le \int_0 ^1 || \left(r^*\omega\right)_{\pa_s}(s,x)||_{L^\infty(N^\eta)}  ds  \overset{(\ref{eq_der_r_s})}\lesssim  \eta || \omega||_{L^\infty(M^\eta)}.
  	\end{equation}
  	\end{rem}

 Inequality (\ref{eq_r_borne}) does not hold for $\Rfr$ (even if $p$ is close to $1$). We however have the following weaker inequalities. In this proposition (and in the sequel), we write $||\omega||_{L^p(M^{\sqrt{\eta}})}$ although $\omega$ is just defined on $\mep$ (and maybe $\sqrt{\eta}\ge \ep$) as a shortcut for $||\omega||_{L^p(M^{\sqrt{\eta}}\cap \mep)}$.
  
  \begin{pro}\label{pro_borne_R}
 For any  $p\in [1,\infty)$  close enough to $1$ and $ j< m$, we have for each $\omega\in L^j_p(\mep)$ and $\eta\le  \ep$:
\begin{equation}\label{eq_R_borne} ||\Rfr \omega||_{L^p(N^\eta)}\le C \,\eta^\frac{p-1}{p} \, \left(||\omega||_{L^p(M^{\sqrt{\eta}})} +\eta^\sigma||\omega||_{L^p(M^{\ep})}\right) ,
\end{equation}
where $C$ and $\sigma$ are positive real numbers (independent of $\eta$ and $\omega$). Moreover, for all $p\in [1,\infty)$ we have: if $j=1$ and $m>1$
\begin{equation}\label{eq_R_borne1} ||\Rfr \omega||_{L^p(N^\eta)}\le C \,\eta^\frac{\min(m,p)-1}{p} \,||\omega||_{L^p(\mep)}  ,
\end{equation}
and, if $j=m$  (with $m\ge 1$),
\begin{equation}\label{eq_R_borne_j=m} ||\Rfr \omega||_{L^p(N^\eta)}\le C\eta^\frac{-\nu}{p'}\cdot ||\omega||_{L^p(\mep)},
\end{equation}
where $\nu$ is given by (\ref{eq_jacRt}).
\end{pro}

\begin{proof}
We first focus on (\ref{eq_R_borne}).
 Let us first notice that, because $r_s$ is $Cs$-Lipschitz for all $s$, with $C$ independent of $s$, no eigenvalue of  $^{\bf t} DR^\eta_t(x)DR_t^\eta(x) $ can be smaller than $(Ct)^2$, for almost all $x\in \mep$ and $t< \frac{\ep}{|x|}$ (since $R_t^\eta$ is the inverse of $r^{t\eta}_{\frac{1}{t}}$), which entails that if $\alpha$ is any $(j-1)$-form on $N^\eta$   and $(t,x)\in(1,\frac{\ep}{\eta}) \times N^\eta $, $\eta<\ep$, then:
   \begin{equation}\label{eq_jac_m_2}
    \left| R_t ^{\eta*}\alpha\right(x)| \lesssim  \frac{\jac R_t^\eta(x)}{t^{m-j}}  |\alpha(R_t(x))|.
   \end{equation} 

Let now  for $(t,x)\in \R\times \mep\setminus \{0\}$ satisfying $t<\frac{\ep}{|x|}$, $\lambda_t(x):=\frac{\frac{\pa R}{\pa t}(t,x)}{\left|\frac{\pa R}{\pa t}(t,x)\right|},$ and let $\lambda^*_t$ be the differential $1$-form defined by $\lambda_t$.  It follows from
(\ref{item_lambda}) of section \ref{sect_key_facts} that if $\zeta_1,\dots,\zeta_{m-1}$ is an orthonormal basis of $T_{R_t(x)} N^{t\eta}$ (for some $x\in  N^\eta$ and $t\le \frac{\ep}{\eta}$) then the distance between $\lambda_t^*$ and the space $Z$ generated by $\zeta_1^*,\dots,\zeta_{m-1}^*$ is bounded below away from zero by some constant independent of (almost every) $(t,x)$ and $\eta$. If $\omega$ is a $j$-form on $\mep$, $j>0$, it makes it possible to decompose  $\omega(R_t(x))$ as
   \begin{equation}\label{eq_lambda}\omega(R_t(x))= \lambda_t^*(x)\wedge\omega_t'(R_t(x)),\end{equation}
   with for every $t$,   $ \omega_t'(R_t(x))$  $(j-1)$-form  satisfying
   \begin{equation}\label{eq_ome12}  \omega_{t,\pa \rho}'(R_t(x))=0,\end{equation}
 where  $\rho:\mep\to \R$ is the restriction of the distance to the origin   (i.e. $\omega'_t\in \wedge^{j-1} Z$), and \begin{equation}\label{eq_omega2}|\omega_t'(R_t(x))| \lesssim |\omega(R_t(x))|. \end{equation}

      For $\omega\in L^j_p(\mep)$ and almost any $x\in \mep$ and $t< \frac{\ep}{|x|}$, we now can write, setting for simplicity $\eta:=|x|$,
   \begin{eqnarray}\label{eq_Romega2}
   \left| (R^*\omega)_{\pa_t}(t,x)  \right|&	\overset{(\ref{eq_lambda})}\le & \left| R_t^{\eta*}  \omega_t'(x)   \right| \cdot \left|\frac{\pa R}{\pa t} (t,x) \right|\overset{(\ref{eq_par})}\lesssim \eta\left| R_t^{\eta*} \omega_t' (x)  \right|,\nonumber\\
&	\overset{(\ref{eq_jac_m_2})}\lesssim &\eta	\left|\omega_t' (  R_t^\eta(x))\right|\cdot  \frac{\jac R_t^\eta(x)}{t^{m-j}}
\overset{(\ref{eq_jacRt})}{\lesssim} \eta \,t^{l} \cdot \left|\omega_t' (R_t^\eta(x))\right|\cdot  (\jac R_t^\eta (x))^{1/p},
\end{eqnarray} where $l:= \frac{\nu}{p'}-m+j$ (with $l:=-m+j$ if  $p=1$),
which, by (\ref{eq_omega2}), gives,
\begin{equation}\label{eq_omega2R}  
	 \left| (R^*\omega)_{\pa_t}(t,x)  \right|\lesssim  \eta \,t^{l} \cdot \left|\omega (R_t^\eta(x)  )\right|\cdot  (\jac R_t^\eta(x))^{1/p}.
\end{equation}
   For  $\omega\in L^j_p(\mep)$, we now can write, using Minkowski's inequality:
    \begin{eqnarray*}||\Rfr\omega||_{L^p(N^\eta)} =  \left(\int_{N^\eta}\left| \int_1^{\frac{\ep}{\eta} } \left(R^*\omega\right)_{\pa_t}(t,x) dt\right|^p \right)^{1/p} 
   	\le \int_1 ^{\frac{\ep}{\eta} }\left(\int_{N^\eta}\left| (R^*\omega)_{\pa_t} (t,x)\right|^p \right)^{1/p}dt
   \end{eqnarray*} 
       \begin{eqnarray}\label{eq_K1_neta}
       	\overset{(\ref{eq_omega2R})}\lesssim  \eta\int_1 ^{\frac{\ep}{\eta} }t^l\left(  \int_{N^\eta}\left|\omega \circ  R_t^\eta\right|^p\cdot  \jac R_t ^\eta\right)^{1/p}dt
       	       	= \eta\int_1 ^{\frac{\ep}{\eta} }t^l||\omega||_{L^p(N^{t\eta})}\,dt.
       	   \end{eqnarray} 
As, for $j<m$ and $p'$ large enough, we have $lp'=\nu-(m-j)p'<-1$, we see that $||t^l||_{ L^{p'}([1,\infty))}$ is bounded independently of $p$ sufficiently close to $1$ for such $j$. As a matter of fact,  we can write for $\eta <\ep^2$ (assuming $j<m$, $p'$ large enough), using H\"older's inequality (if $p=1$, just write $t^l\lesssim 1$):
\begin{equation}\label{eq_K1_neta1}
\int_1 ^{\frac{1}{\sqrt \eta} }t^l||\omega||_{L^p(N^{t\eta})}\,dt \lesssim  \left(\int_1 ^{\frac{1}{\sqrt\eta} }||\omega||_{L^p(N^{t\eta})}^p\,dt\right)^{\frac{1}{p}}\overset{(\ref{eq_coarea_sph})}\lesssim \eta^{-\frac{1}{p}}   ||\omega||_{L^p(M^{\sqrt\eta})}.
  \end{equation}
Moreover,  by H\"older's inequality (if $p=1$, just use that $t^l \le \eta^{-l/2}$), we have for $\eta \le \ep^2$:
\begin{equation*}
 \int_{\frac{1}{\sqrt\eta} } ^{\frac{\ep}{\eta} }t^l||\omega||_{L^p(N^{t\eta})}\,dt \le  \left(\int_{\frac{1}{\sqrt\eta} } ^{\frac{\ep}{\eta} } t^{lp'} \, dt\right)^\frac{1}{p'} \left(\int_{\frac{1}{\sqrt\eta} } ^{\frac{\ep}{\eta} }||\omega||^p_{L^p(N^{t\eta})}\,dt\right)^\frac{1}{p}
\overset{(\ref{eq_coarea_sph})} \lesssim  \eta^{\sigma-\frac{1}{p}}\cdot ||\omega||_{L^p(M^{\ep})},
\end{equation*}
where $\sigma=-\frac{1}{2}(\frac{1}{p'}+l)$, which is positive for $p$ sufficiently close to $1$. Together with   (\ref{eq_K1_neta}) and (\ref{eq_K1_neta1}), this yields (\ref{eq_R_borne}) for $\eta\le\ep^2$. If $\eta\in [\ep^2,\ep]$ then (\ref{eq_K1_neta1}) with $\frac{\ep}{\eta}$ instead of $\frac{1}{\sqrt\eta}$ (together with  (\ref{eq_K1_neta})) already gives the desired estimate.

Let us now check (\ref{eq_R_borne1}). As observed at the very beginning of the proof, no eigenvalue of the mapping $^{\bf t} DR^\eta_t(x)DR_t^\eta(x) $ can be smaller than $(Ct)^2$, for almost all $x\in \mep$ and $t< \frac{\ep}{|x|}$, which entails that
\begin{equation}\label{eq_jac_R_below}
t^{m-1} \lesssim	\jac R_t^\eta(x). 
\end{equation}
 When $j=1<m$, $\omega_t'$ is a $0$-form (see (\ref{eq_lambda}) for $\omega'_t$), which makes it possible to rewrite the above computation (\ref{eq_Romega2}) as
   \begin{eqnarray}\label{eq_Romega21}
	\left| (R^*\omega)_{\pa_t}(t,x)  \right|&	\overset{(\ref{eq_lambda})}\le & \left| R_t^{\eta*}  \omega_t'(x)   \right| \cdot \left|\frac{\pa R}{\pa t} (t,x) \right|\overset{(\ref{eq_par})}\lesssim \eta\left| R_t^{\eta*} \omega_t' (x)  \right|=\eta	\left|\omega_t' (  R_t^\eta(x))\right|,\nonumber \\
	&	\overset{(\ref{eq_jac_R_below})}\lesssim &\eta	\left|\omega_t' ( R_t^\eta(x))\right|  \frac{(\jac R_t^\eta (x))^{1/p}}{t^\frac{m-1}{p}}
		\overset{(\ref{eq_omega2})}\lesssim \eta	\left|\omega (  R_t^\eta(x))\right|  \frac{(\jac R_t^\eta (x))^{1/p}}{t^\frac{m-1}{p}}.
\end{eqnarray} 
Repeating the computation (\ref{eq_K1_neta}) with now $l=\frac{1-m}{p}$ (replacing (\ref{eq_omega2R}) with (\ref{eq_Romega21})),  we deduce
\begin{eqnarray*}||\Rfr\omega||_{L^p(N^\eta)} \lesssim \eta\int_1 ^{\frac{\ep}{\eta} }t^\frac{1-m}{p}||\omega||_{L^p(N^{t\eta})}\,dt\le  \eta\cdot ||t^{\frac{1-m}{p}}||_{L^{p'}([1,\frac{\ep}{\eta}])}\cdot\left(\int_1 ^{\frac{\ep}{\eta} }||\omega||^p_{L^p(N^{t\eta})}\,dt \right)^\frac{1}{p},
\end{eqnarray*} 
using H\"older's inequality. A straightforward computation of integration shows that $$ ||t^{\frac{1-m}{p}}||_{L^{p'}([1,\frac{\ep}{\eta}])}\lesssim \eta^{\frac{\min(m,p)}{p}-1},$$ which together with (\ref{eq_coarea_sph}) enables us to derive the desired fact from the  above estimate.

 It remains to show (\ref{eq_R_borne_j=m}). Let $\omega$ be an $L^p$ $m$-form on $\mep$. Thanks to Minkowski's integral inequality, we have
 \begin{eqnarray*}||\Rfr\omega||_{L^p(N^\eta)} =  \left(\int_{N^\eta}\left| \int_1^{\frac{\ep}{\eta} } \left(R^*\omega\right)_{\pa_t}(t,x) dt\right|^p \right)^{1/p} 
 	\le \int_1 ^{\frac{\ep}{\eta} }\left(\int_{N^\eta}\left| (R^*\omega)_{\pa_t}(t,x) \right|^p \right)^{1/p}dt
 \end{eqnarray*}
 \begin{eqnarray*}\overset{(\ref{eq_lambda})}\lesssim  \int_1 ^{\frac{\ep}{\eta} }\left(\int_{N^\eta} \left|\frac{\pa R}{\pa t}\right|\cdot \left| R_t^{\eta*}\omega'_t  \right|^p \right)^{1/p}\,dt\overset{(\ref{eq_par})}{\lesssim}  \eta\int_1 ^{\frac{\ep}{\eta} }\left( \int_{N^\eta} \left|\omega' _t\circ  R_t^\eta\right|^p\cdot  (\jac R_t ^\eta)^p \right)^{1/p}\,dt.
 \end{eqnarray*} 
 As a matter of fact by (\ref{eq_omega2}), we get
 \begin{eqnarray*}\label{eq_K1_neta_bis}\nonumber||\Rfr\omega||_{L^p(N^\eta)}
 	&\lesssim &\eta\int_1 ^{\frac{\ep}{\eta} }\left( \int_{N^\eta} \left|\omega \circ  R_t^\eta\right|^p\cdot  (\jac R_t ^\eta)^p \right)^{1/p}\,dt \\
 	&\overset{(\ref{eq_jacRt})}\lesssim & \eta\int_1 ^{\frac{\ep}{\eta} }t^\frac{\nu(p-1)}{p}\left(  \int_{N^\eta}\left|\omega \circ  R_t^\eta\right|^p\cdot  \jac R_t ^\eta\right)^{1/p}dt\\
 	\nonumber &=& \eta\int_1 ^{\frac{\ep}{\eta} }t^\frac{\nu(p-1)}{p}||\omega||_{L^p(N^{t\eta})}\,dt \\
 	\nonumber &\le& \eta\left(\int_1 ^{\frac{\ep}{\eta} }t^{\nu}\,dt\right)^{1/p'}\cdot \left(\int_1 ^{\frac{\ep}{\eta} }||\omega||^p_{L^p(N^{t\eta})}\,dt\right)^{1/p}\qquad\mbox{(by H\"older)}\\
 	\nonumber &\lesssim& \eta^{1-\frac{\nu+1}{p'}}\cdot \left(\int_1 ^{\frac{\ep}{\eta} }||\omega||^p_{L^p(N^{t\eta})}\,dt\right)^{1/p}\\
 	&\overset{ (\ref{eq_coarea_sph})}\lesssim& \eta^{1-\frac{\nu+1}{p'}-\frac{1}{p}}\cdot  ||\omega||_{L^p(\mep)},
 \end{eqnarray*}
 which is exactly (\ref{eq_R_borne_j=m}) (in the case $p=1$, we assumed $\frac{1}{p'}=0$).
  \end{proof}

   \begin{rem}\label{rem_borne_Rr}
Using (\ref{eq_coarea_sph}), we immediately deduce from the preceding propositions that there is a constant $C$ such that:
 \begin{enumerate}[(i)]
\item\label{item_r_borne} For any $p\in [1,\infty]$ large enough (use (\ref{eq_neta_infty}) for $p=\infty$),  we have for each $\omega\in L^j_p(\mep)$, $j\ge 0$,  and every $\eta\le \ep$:
\begin{equation}\label{eq_r_borne_m}
||\rfr \omega||_{L^p(M^\eta)}\le C \,\eta \,||\omega||_{L^p(M^\eta)} .\end{equation}
Moreover, when $j=m$, this inequality holds for all $p$. 
\item\label{item_R_borne} For any  $p\in[1,\infty)$  close enough to $1$, we have for each  $\omega\in L^j_p(\mep)$, with $0\le j< m$,  and every $\eta\le \ep$:
\begin{equation}\label{eq_R_borne_m} ||\Rfr \omega||_{L^p(M^\eta)}\le C \,\eta \, \left(||\omega||_{L^p(M^{\sqrt{\eta}})} +\eta^\sigma||\omega||_{L^p(M^{\ep})}\right) ,
\end{equation}
where $\sigma$ is a positive real number (independent of $\eta$ and $\omega$).  In the case $j=1<m$, (\ref{eq_R_borne1}) establishes that $\Rfr$ is continuous on $L^1_p(\mep)$ for all $p\in [1,\infty)$, and, in addition,  we have for every $\omega\in L^1_\infty(\mep)$
\begin{equation}\label{eq_R_borne_infty} ||\Rfr \omega||_{L^\infty(\mep)}\overset{(\ref{eq_par})}\lesssim \mbox{ess sup}_{x\in \mep} \left(|x| \,\int_1^\frac{\ep}{|x|} || \omega||_{L^\infty(M^\ep)}dt\right)\le || \omega||_{L^\infty(M^\ep)}.
\end{equation}
 In the case $j=m$, (\ref{eq_R_borne}) is not true, and indeed, it is not always true that $||\Rfr\omega||_{L^p(N^\eta)}$ tends to $0$ as $\eta \to 0$, even if $p$ is close to $1$. Nevertheless, by (\ref{eq_R_borne_j=m}),  $\Rfr$ is continuous on $L^m_p(\mep)$ for  $p\ge 1$ sufficiently close to $1$.
\end{enumerate}
\end{rem}

  \begin{pro}\label{pro_rR_hom_op}
 If $\omega\in \wca^j_1(\mep)$, $0\le j<m$, is identically equal to $0$ near $N^\ep$ then:
\begin{equation}\label{eq_R_homot_operator}d\Rfr\omega+\Rfr d\omega=\omega.\end{equation}
 If $p\in [1,\infty]$ is large enough and $j\ge 1$, then we have  for all  $\omega \in \wca_{p} ^j (\mep)$:
\begin{equation}\label{eq_r_homot_operator}d\rfr\omega+\rfr d\omega=\omega.\end{equation}
Furthermore, if $j=m$ then (\ref{eq_r_homot_operator}) is true for all $p\in [1,\infty]$.
\end{pro}
\begin{proof}We start with the case where $\orn\in \delta M$. As $\cc^{j,\infty}(\mep)\cap \wca^j_p(\mep)$ is dense in $\wca^j_p(\mep)$ for all $p$  and since we are interested in cases where $\Rfr$ and $\rfr$ are continuous in the $L^p$ norm (see Remark \ref{rem_borne_Rr}), it suffices to prove these identities for a smooth form.
 If $\omega\in \wca^j_1(\mep)$ is a smooth form that vanishes near $\nep$ then for almost each (nonzero) $x\in \mep$  and  each $t_0>0$ sufficiently close to $\frac{\ep}{|x|}$ (and not greater), we have $$\Rfr\omega(x) =\int_1 ^{t_0} (R^* \omega)_{\pa t}(t,x)\,dt,  $$
which means that (\ref{eq_R_homot_operator}) follows from (\ref{eq_chain_homotopy}).  Moreover, (\ref{eq_chain_homotopy}) also entails that if, for $s_0>0$, we set
 $$ \rfr_{s_0}\omega(x) :=\int_{s_0} ^1 (r^* \omega)_{\pa s}(s,x)\,ds,
$$
then $$d\rfr_{s_0}\omega+\rfr_{s_0}  d\omega =\omega-r_{s_0}^*\omega,$$
which means that (\ref{eq_r_homot_operator})  reduces to check that $r_{s_0}^*\omega$ tends to zero in the $L^p$ norm as $s_0$ goes to zero (for $p$ large).

In the case $1\le j<m$, we have for $s>0$ (since $r_s$ is $Cs$-Lipschitz for some $C$ independent of $s$)
$$||r_{s}^* \omega||_{L^p (\mep)} \lesssim \left(\int_{\mep} s^p |\omega\circ r_{s}|^p \right) ^{1/p}\overset{(\ref{eq_jacr_s})}{\lesssim} \left(\int_{\mep} s^{p-\nu}|\omega\circ r_{s}|^p  \jac r_{s}\right) ^{1/p} = s^{1-\frac{\nu}{p}} || \omega||_{L^p (M^{s\ep})},$$
 which tends to $0$ as $s$ tends to $0$ for each $p\ge \nu$, as needed. 

 In the case $j=m$, as $\jac r_s$ is bounded independently of $s$, we have for all $p\ge 1$:
$$||r_{s}^* \omega||_{L^p (\mep)}=\left(\int_{\mep} (\jac r_{s})^p\cdot  |\omega\circ r_{s}|^p \right) ^{1/p} \lesssim \left(\int_{\mep} \jac r_{s}\cdot  |\omega\circ r_{s}|^p \right) ^{1/p}=|| \omega||_{L^p (M^{s\ep})},$$
which tends to $0$ as $s$ tends to $0$, for all $p\in [1,\infty)$.

If $\orn\in M$ then, as $r$ is a Lipschitz retraction by deformation in the vicinity of $0$, the above argument still applies to show (\ref{eq_r_homot_operator}). Since $R$ is not defined at $0$, (\ref{eq_R_homot_operator}) is more delicate in this case, and  is indeed  not true as currents on $\mep$ if $j=m$. By the above, we nevertheless know that if  $\omega \in \wca^j_1(\mep)$ vanishes near $\nep$ then (\ref{eq_R_homot_operator}) holds on $M':=\mep\setminus \{0\}$. Hence, if for $\eta>0$ small we define a smooth function on $\mep$ by $\hat{\psi}_\eta(x):=\psi_\eta (|x|)$, where $\psi_\eta$ is as in (\ref{eq_psieta}),   (\ref{eq_R_homot_operator}) for  $ \omega_{|M'}$ entails that
$$  d\hat{\psi}_\eta\Rfr \omega+\hat{\psi}_\eta\Rfr d \omega= \hat{\psi}_\eta  \omega+ d\hat{\psi}_\eta\wedge \Rfr\omega,$$
as currents on $M'$. Since all these forms are $0$ nearby the origin, this equality continues  to hold on $\mep$. As $\hat\psi_\eta$ tends to $1$ (pointwise), it thus suffices to show that $d\hat{\psi}_\eta\wedge \Rfr\omega$ goes to zero (as current on $\mep$) as $\eta\to 0$, which is clear since (\ref{eq_psieta_ineq}) and (\ref{eq_R_borne_m}) establish that this form goes to zero in the $L^1$ norm for each $j<m$.
\end{proof}
\begin{rem}\label{rem_r_hom_op_j=0}
 Equality (\ref{eq_r_homot_operator}) is not true for $j=0$. It however follows from Lemmas 2.1 and 2.4 of \cite{trace} that when $u$ extends continuously at $\orn$ (which happens as soon as $M$ is connected at $\orn$, see the paragraph just above (\ref{eq_trace_fn_plarge})) we have for all $p\in [1,\infty]$ sufficiently large and all $u\in \wca^0_p(\mep)$ (recall that $\rfr u=0$ by definition):
 \begin{equation}\label{eq_tra_et_r}
 \rfr du= u-\tra\, u(0),
 \end{equation}
where $\tra$ is the trace operator defined in (\ref{eq_trace_fn_plarge}). In particular, (\ref{eq_r_homot_operator}) holds if and only if $\tra u(0)=0$. It is also worthy of notice that the above proof has established that (\ref{eq_R_homot_operator}) holds for $j=m$ if $\orn\in \delta M$. Moreover, if it is true that proofs  can   be significantly simplified  when $\orn\in M$, we on the other hand wish to emphasize  that the argument we used still applies if $M$ is nonsmooth and $\orn$ is a singular point of it (which is a more delicate case). This means that Proposition \ref{pro_rR_hom_op} can  be helpful to integrate closed $L^1$ currents on singular varieties, which  can be valuable to investigate Plateau problems.
\end{rem}

\begin{pro}\label{pro_rR_def_sur_wj}
 Let $\mba^\ep:=\bou(\orn,\ep)\cap \mba$. For all $p\in [1,\infty]$ sufficiently large, $\rfr$ and $\Rfr$ induce for each $1\le j<m$ continuous operators
$$\rfr:\wca^j_p(\mep)\to \wca_{p}^{j-1}(\mep) \et \Rfr:\wca^{j}_{p',\mba^\ep}(\mep)\to \wca_{p',\mba^\ep}^{j-1}(\mep).$$
Moreover, $\rfr:\wca^m_p(\mep)\to \wca_{p}^{m-1}(\mep)$ is well-defined and continuous for all $p\in [1,\infty]$.
  \end{pro}
  \begin{proof}
   Proposition \ref{pro_borne_r} shows that $\rfr \omega$  is $L^p$ if so is $\omega$, for all $p$ sufficiently large (for all $p$ if $j=m$), as explained in Remark \ref{rem_borne_Rr}.  By (\ref{eq_r_homot_operator}), $d\rfr \omega=\omega-\rfr d\omega$ for all $\omega \in \wca^j_p(\mep)$ with $p$ sufficiently large (for any $p$ if $j=m$), showing that $d\rfr \omega$ is $L^p$ for all such $\omega$. The same argument applies to $\Rfr$, using (\ref{eq_R_homot_operator}).
  \end{proof}

\end{subsection}

\begin{subsection}{Some more properties of $\rfr$ and $\Rfr$.} We are still in the setting of subsection \ref{sect_the_retraction_r_s_and_R_t}.  We are going to provide formulas for $\Rfr$ and $\rfr$ (Proposition \ref{pro_Rfr_2eme_forme}) that will be useful to show that these operators are somehow dual to each other (Proposition \ref{pro_dual_Rr}), from which we will derive that they preserve forms vanishing on a subset  of the boundary (Proposition \ref{pro_r_preserve_vanishing}).

 Theorem \ref{thm_local_conic_structure}  provides for $\ep>0$ small a Lipschitz  subanalytic homeomorphism
  $$H: 0_{\R^n}* (\sph(\orn,\ep)\cap X)\to  \Bb(\orn,\ep) \cap X,\quad \mbox{with $X=\{\orn\} \cup M$},$$
preserving the distance to the origin and possessing properties $(i)$ and $(ii)$ of this theorem.
  In particular, $H$ gives rise to a  globally subanalytic homeomorphism:
  \begin{equation}\label{eq h pour K_nu}
h:(0,1) \times N^\ep \to \mep\setminus \{0\}, \qquad (\tau,z)\mapsto h(\tau,z):=H(\tau z).\end{equation}
We then denote by $\pa_\tau $ the constant vector field $(1,0_{\R^n})$   on $\R\times N^\ep$.
 
 \begin{pro}\label{pro_Rfr_2eme_forme} For $h$ be as above, we have
  \begin{enumerate}[(i)]
   \item for every $\omega\in L_1^j(\mep)$, $j\ge 1$:
   $$h^* \Rfr \omega(\tau,z)=\int^\tau _1 \left(h^* \omega (t,z)\right)_{\pa_\tau}dt.$$
   \item for each $p\in [1,\infty]$ sufficiently large and each $\omega\in L_p^j(\mep)$, $j\ge 1$: 
   \begin{equation*}\label{eq_r_deuxieme_forme}h^* \rfr \omega(\tau,z)=\int_0^\tau ( h^* \omega (t,z))_{\pa_\tau}dt.\end{equation*}
  \end{enumerate}
  Moreover, the  latter equality holds true for all $p\in [1,\infty]$  if $j=m$.
 \end{pro}

 \begin{proof} Observe first that it follows from the definitions of $R$ and $h$ that \begin{equation}\label{eq_rfr_h} R(t,h(\tau,z))=h(t\tau,z)= H(t\tau z),\end{equation}
from which it follows that:
\begin{equation}\label{eq_rfr_dh}
 DR_t (h(\tau,z))\frac{\pa h}{\pa \tau}(\tau,z)=	t\frac{\pa h}{\pa \tau}(t\tau,z)\et DR_t (h(\tau,z))\frac{\pa h}{\pa z}(\tau,z)=\frac{\pa h}{\pa z}(t\tau,z),
\end{equation}
as well as 
\begin{equation}\label{eq_rfr_dh2}
 \frac{\pa R}{\pa t}(t,h(\tau,z))=\tau\frac{\pa h}{\pa \tau}(t\tau,z).
\end{equation}

Let $\omega\in L_1^j(\mep)$, $j\ge 1$. By definition of $\Rfr$, we have for  all   $\zeta \in \otimes^{j-1} T_x \mep$, $x\in \mep\setminus \{0\}$:
$$\Rfr \omega(x)\zeta=\int^1 _\frac{\ep}{|x|} \omega(R(t,x))\left(\frac{\pa R}{\pa t}(t,x) \otimes DR_t(x) \zeta\right) dt.$$
 Hence, pulling back via $h$ (which is bi-Lipschitz on $[\eta,1]\times N^\ep$ for every $\eta>0$), we get for $(\tau,z)\in (0,1)\times N^\ep$ and each  $(j-1)$-multivector $\xi$ of $\R\times T_z N^\ep$ (note that if $h(\tau,z)=x$ then, since $|z|=\ep$, we must have $\tau=\frac{|x|}{\ep}$, which accounts for the fact that the new bounds of the integral below are $\frac{1}{\tau}$ and $1$):
 \begin{eqnarray}\label{eq_pullback_hrR1} h^*\Rfr \omega(\tau,z)\xi&=&\int^1 _\frac{1}{\tau} \omega(R(t,h(\tau,z)))\left(\frac{\pa R}{\pa t}(h(\tau,z)) \otimes DR_t(h(\tau,z)) Dh(\tau,z)\xi \right) dt\nonumber\\
  &\overset{(\ref{eq_rfr_dh2})}=& \int^1 _\frac{1}{\tau} \tau\cdot \omega(h(t\tau,z))\left(\frac{\pa h}{\pa \tau}(t\tau,z) \otimes  DR_t(h(\tau,z)) Dh(\tau,z)\xi \right) dt.
 \end{eqnarray}
Let us  decompose $\xi$ as $\pa_\tau\otimes\xi_1 +\xi_2$, with $\xi_1$ and $\xi_2$ respectively $(j-2)$ and $(j-1)$-multivectors of $ T_z N^\ep$. As
\begin{equation*}\label{eq_pullback_hrR2}
DR_t (h(\tau,z))Dh(\tau,z)\xi =DR_t (h(\tau,z)) (\frac{\pa h}{\pa \tau}(\tau,z)\otimes\frac{\pa h}{\pa z}(\tau,z)\xi_1) +DR_t (h(\tau,z))\frac{\pa h}{\pa z}(\tau,z) \xi_2,
\end{equation*}
 by (\ref{eq_rfr_dh}), we see that
\begin{equation}\label{eq_dRhxi}
	DR_t (h(\tau,z))Dh(\tau,z)\xi =t \frac{\pa h}{\pa \tau}(t\tau,z)\otimes\frac{\pa h}{\pa z}(t\tau,z)\xi_1 +\frac{\pa h}{\pa z}(t\tau,z) \xi_2.
\end{equation}
 Plugging it into (\ref{eq_pullback_hrR1}), we get (the first term of the resulting sum is equal to $0$):
  \begin{eqnarray*} h^*\Rfr \omega(\tau,z)\xi&=&\int^1 _\frac{1}{\tau} \tau\cdot \omega(h(t\tau,z))\left(\frac{\pa h}{\pa \tau}(t\tau,z) \otimes \frac{\pa h}{\pa z}(t\tau,z)\xi_2 \right) dt\\
 &=&\int^1 _\frac{1}{\tau} \tau\cdot \omega(h(t\tau,z))Dh(t\tau,z)(\pa_\tau \otimes\xi)  dt, \end{eqnarray*}
  since $\pa_\tau \otimes \xi=\pa_\tau\otimes \pa_\tau\otimes \xi_1 +\pa_\tau\otimes \xi_2$ (so that again, one term is zero).
 After a change of variable $t':=\tau t$ in this Riemannian integral, we already get $(i)$.
 
  The proof of $(ii)$ is analogous, just replace $R$ with $r$. The argument however requires to check that $\left(h^*\omega (\cdot, z)\right)_{\pa_\tau}$ is $L^1$ on $[0,1]$, which holds for every $L^p$ form $\omega$ and almost every $z\in \nep$, provided $p$ is sufficiently large (for all $p$ if $j=m$). This fact can be seen as follows. 
 
We start with the case $j<m$.  Since $h$ is bi-Lipschitz on $(\eta,1)\times N^\ep$ for each $\eta>0$, by \L ojasiewicz's inequality, there is an integer $l$ such that for almost all  $(\tau,z)\in (0,1) \times \nep$:
 \begin{equation}\label{eq_jac_htau}\tau ^l \lesssim \jac h(\tau,z). \end{equation}
 Consequently, thanks to Minkowski's integral inequality, we can write:
 \begin{eqnarray*}
 	\left(\int_{\nep} \left(\int_0 ^1 |h^*\omega(\tau,z)|d\tau\right)^p \right)^{\frac{1}{p}}
 	&\le& \int_0 ^1 \left( \int_{\nep} |h^*\omega(\tau,z)|^p  \right)^\frac{1}{p}d\tau\\\
& \overset{(\ref{eq_jac_htau})}\lesssim& \int_0 ^1 \tau^\frac{-l}{p}\left( \int_{\nep}|\omega (h(\tau,z))|^p\jac h(\tau,z) \right)^\frac{1}{p}d\tau\\
& \le& || \tau^\frac{-l}{p}||_{L^{p'}([0,1])}\left( \int_0 ^1\int_{\nep}|\omega (h(\tau,z))|^p\jac h(\tau,z) d\tau \right)^\frac{1}{p},
\end{eqnarray*}
by H\"older's inequality. This yields the needed fact for all $p$ sufficiently large for  $\tau^\frac{-l}{p}$ to be $L^{p'}$ on $[0,1]$.

Finally, in the case $j=m$, analogously, we have, by Minkowski's integral inequality,
 	$$\left(\int_{\nep} \left(\int_0 ^1 |h^*\omega(\tau,z)|d\tau\right)^p \right)^{\frac{1}{p}}\le \int_0 ^1 \left( \int_{\nep} |h^*\omega(\tau,z)|^p  \right)^\frac{1}{p}d\tau$$
$$ = \int_0 ^1 \left( \int_{\nep}|\omega (h(\tau,z))|^p\jac h(\tau,z)^p\right)^\frac{1}{p}d\tau
  \lesssim \left( \int_0 ^1\int_{\nep}|\omega (h(\tau,z))|^p\jac h(\tau,z)d\tau \right)^\frac{1}{p}<\infty,$$
applying again H\"older's inequality (as $h$ is Lipschitz we have  $\jac h(\tau,z)^p\lesssim \jac h(\tau,z)$). 
\end{proof}

\begin{pro}\label{pro_dual_Rr}
For all $p\in [1,\infty]$ sufficiently large and $j\ge 1$, if $\alpha\in L^j_p(\mep)$ and $\beta\in L^{m-j+1}_{p'}(\mep)$ then:
\begin{equation}\label{eq_rR_adjoints}
 <\rfr\alpha,\beta>=(-1)^j<\alpha,\Rfr \beta>.
\end{equation}
This formula holds for all $p$ if $j=m$. 
\end{pro}
\begin{proof}
By density of $L^\infty$ forms in the $L^p$ spaces (and continuity of $\rfr$ and $\Rfr$ for suitable $j$ and $p$, see Remark \ref{rem_borne_Rr}), we can assume that $\alpha$ and $\beta$ are bounded, which means that so are $h^*\alpha$ and $h^*\beta$. Observe then that for bounded forms $\alpha$ and $\beta$,  sufficient integrability conditions hold to write (for relevant orientation on $\nep$) 
\begin{eqnarray*}
 <\rfr\alpha,\beta>&=&\int_{(0,1)\times \nep} h^*\rfr\alpha  \wedge h^* \beta \qquad(\mbox{pulling back via $h$}) \\
 &=& \int_{(0,1)\times \nep} \left(\int_0 ^\tau  (h^*\alpha(t,x))_{\pa_\tau}\wedge h^*\beta(\tau,x)dt\right)\quad\mbox{(by Proposition \ref{pro_Rfr_2eme_forme} $(ii)$})\\
 &=&(-1)^j \int_{x\in \nep} \left(\int_0^1 \int_0 ^\tau  (h^*\alpha)_{\pa \tau}(t,x)\wedge (h^*\beta)_{\pa_\tau}(\tau,x) \,dt\,d\tau \right)\\
  &=&(-1)^j \int_{x\in \nep} \int_{0<t\le \tau<1}  (h^*\alpha)_{\pa_\tau}(t,x)\wedge(h^*\beta)_{\pa_\tau}(\tau,x)\,dt\,d\tau.
\end{eqnarray*}
Making the same computation for $< \alpha ,\Rfr\beta>$  and applying Fubini's Theorem then yields (\ref{eq_rR_adjoints}).
\end{proof}

The latter two propositions make it possible to establish the following relative version of Proposition \ref{pro_rR_def_sur_wj}. We recall that $\mba^\ep=\mba\cap \bou(0,\ep)$ and set $A^\ep :=A\cap \bou(0,\ep)$.


\begin{pro}\label{pro_r_preserve_vanishing} Let $A$ be a definable subset-germ of $\delta M$ (at $\orn\in \mba$). If $r_s$ preserves $A^\ep$ for all $s\in (0,1)$ then for every $p\in [1,\infty]$ sufficiently large and all $j\ge 1$, $\rfr$ and $\Rfr$ induce continuous operators $$\rfr:\wca^{j}_p (M^\ep,A^\ep)\to\wca^{j-1}_p (M^\ep,A^\ep) \et \Rfr:\wca^{m-j}_{p',\mba^\ep} (M^\ep,A^\ep)\to\wca^{m-j-1}_{p',\mba^\ep} (M^\ep,A^\ep).$$
	Moreover, in the case $j=m$,  $\rfr$ actually induces such an operator for every $p\in [1,\infty]$.
\end{pro}
\begin{proof}We first focus on $\rfr$.
	By Proposition \ref{pro_rR_def_sur_wj},
	we just have to show that (for suitable $p$), given $\omega \in \wca^{j}_p (M^\ep,A^\ep) $, $j\ge 1$,  $\trd^p \rfr\omega$  vanishes on $A^\ep$.
	Let for this purpose $\beta\in \wca^{m-j}_{p', \mep \cup A^\ep}(\mep)$, $1\le j\le m$. If $p$ is sufficiently large, we have for such $\omega$:
	\begin{equation}\label{eq_omega_beta}(-1)^j<\rfr\omega,d\beta>\overset{(\ref{eq_rR_adjoints})}{ = }<\omega,\Rfr d \beta>\overset{(\ref{eq_R_homot_operator})}{=}<\omega ,\beta> -<\omega,d \Rfr\beta>.\end{equation}
Let $\psi_\eta$ be the function defined in (\ref{eq_psieta}) and set $\hat\psi_\eta (x):=\psi_\eta(|x|)$.	Since  $r_s$ was required to preserve $A$, so does $R_t$,  and since $\beta$ is zero in the vicinity of $\delta \mba^\ep\setminus A^\ep$, so does $\hat\psi_\eta \Rfr \beta$ (by definition of $\Rfr$, see (\ref{eq_Rfr})). As $\trd^p \omega$ vanishes on $A^\ep$, we therefore must have:
	$$< \omega,d (\hat\psi_\eta\Rfr \beta)> =(-1)^{j+1}<d \omega, \hat\psi_\eta \Rfr\beta>.$$
	Thanks to (\ref{eq_R_borne_m}) and (\ref{eq_psieta_ineq}), we know that  $\hat\psi_\eta\Rfr \beta$ tends to $\Rfr \beta$ in the $\wca_{p'}^{m-j-1}$ norm for $p$ large (observe that $\hat\psi_\eta$ tends to $1$ pointwise). As a matter of fact, passing to the limit as $\eta\to 0$ in the above identity, we get:
	$$< \omega,d \Rfr \beta> =(-1)^{j+1}<d \omega,  \Rfr\beta>\overset{(\ref{eq_rR_adjoints})}{=}<\rfr d \omega,\beta>\overset{(\ref{eq_r_homot_operator})}= <\omega,\beta>-<d\rfr \omega, \beta >.$$
	Plugging this equality into (\ref{eq_omega_beta}) provides the desired fact (i.e. (\ref{eq_lsp_j}) for the pair $\rfr\omega,\beta$).
	
	To prove  the analogous fact for $\Rfr$
	just apply the same argument, replacing $\Rfr$ with $\rfr$, (\ref{eq_R_borne_m}) with (\ref{eq_r_borne_m}), and switch (\ref{eq_R_homot_operator}) with (\ref{eq_r_homot_operator}) (for $p=\infty$, if $\beta\in \wca^{j}_{p, \mep \cup A^\ep}(\mep)$ then $\hat\psi_\eta \rfr \beta$ then  does not converge in $\wca^{j-1}_\infty(\mep)$ but, by (\ref{eq_neta_infty}), is bounded in this space and converges pointwise, which is enough).
		Finally, in the case $j=m$, (\ref{eq_r_homot_operator}) holds for all $p$ which makes it possible to write for  $\beta \in \wca_{p', \mep\cup A^\ep}^{0}(M^\ep)$
	$$<\rfr\omega,d\beta>\overset{(\ref{eq_rR_adjoints})}=(-1)^m<\omega,\Rfr d\beta>\overset{(\ref{eq_R_homot_operator})}=(-1)^m <\omega,\beta >\overset{(\ref{eq_r_homot_operator})}=(-1)^m <d\rfr \omega,\beta >,$$
	showing that $\trd^p \rfr \omega$ vanishes on $A^\ep$.  
\end{proof}
%
\begin{rem}
If  $\omega\in \wca^{m}_{p',\mba^\ep} (M^\ep,A^\ep)$ (i.e. when $j=0$ in the latter proposition) then $\Rfr\omega$ may however fail to belong to $ \wca^{m-1}_{p'} (M^\ep,A^\ep)$, even if $p$ is large (the problem is at the origin, compare with Remark \ref{rem_r_hom_op_j=0}).
\end{rem}

\end{subsection}
\end{section}


\begin{section}{Density of $\cc^{j,\infty}(\mba)$ for $p$ large}\label{sect_density}
Our Lefschetz duality result requires to show that  $\poi_{M,A}$ induces a mapping on the cohomology groups, which amounts to show that no ``residue phenomenon'' arises at singularities when integrating by parts the product of an exact form $d\alpha$, $\alpha \in \wca^{j-1}_p(M,A)$, with a form $\beta\in \wca^{m-j}_{p'}(M,\delta M\setminus A)$ (where $A$ is a definable subset of $\delta M$). We will show it for $p$ sufficiently large or close to $1$ (Corollary \ref{cor_lsp_A}). This will be achieved  by establishing the density of smooth bounded forms that vanish near $A$ in the space $\wca^j_p(M,A)$ for each $p$ large (Corollary \ref{cor_densite_non_normal}). We first show that in the case where $M$ is connected along $\delta M\setminus A$ (Definition \ref{dfn_normal}), functions that vanish near $\adh A$ and that are smooth up to the boundary are dense in this Sobolev space (Theorem \ref{thm_dense_formes}), which can be regarded as a partial generalization of Theorem \ref{thm_trace}  to differential forms. This theorem will also be of service in section \ref{sect_further} to extend the other properties  listed in Theorem \ref{thm_trace} (like existence of natural trace operators) as well as to address the case where  $p$ is close to $1$ (subsection \ref{sect_pprime}) and to provide some more explicit characterizations of $ \wca^j_p (M,A)$  (Corollaries \ref{cor_vanishing_res} and \ref{cor_trace_formes}).

We recall that $\cc^{j,\infty}_U(\mba)$ is considered as a subspace of $\cc^{j,\infty}(M)$ (the smooth $j$-forms on $M$) and that the subscript $U$ means that we only take  forms that are compactly supported in $U$ (see (\ref{eq_ccjU})).

\begin{thm}\label{thm_dense_formes} Let $A$ and $E$ be definable subsets of $\delta M$. If $M$ is connected along $\delta M\setminus  A$ then
for every integer $j>\dim E$ and every $p\in [1,\infty)$ sufficiently large, the space $\cc^{j,\infty}_{\mba \setminus \adh {E\cup A}}(\mba)$ is dense in $ \wca^j_p (M,A)$.
\end{thm}
 The proof of this result constitutes the main difficulty of this article  and occupies sections \ref{sect_j=0},  \ref{sect_mol}, and \ref{sect_proof_dense_formes}.   In section \ref{sect_j=0}, we address the case $j=0$,  in section \ref{sect_mol}, we construct ``mollifying operators'', to finally provide  in seven steps our approximation process (section \ref{sect_proof_dense_formes}). We first derive some consequences of the above theorem, among which are the aforementioned corollaries.

 A direct consequence of this theorem is the following fact, referred by the author as the {\it specialization property of the vanishing of the trace}: for any definable subset $A$ of $\delta M$ (assuming, as in the above theorem, that $M$ is connected along $\delta M\setminus A$) we must have for every $p\in [1,\infty)$ sufficiently large and all $j\le m$ (since the elements of $\cc^{j,\infty}_{\mba \setminus \adh { A}}(\mba)$  must vanish in the vicinity of $\adh A$)
\begin{equation}\label{eq_specialization}
	\wca^j_p (M,A)=\wca_p^j(M,\adh A).
\end{equation}

Note also that in the case   $A=E=\emptyset$, Theorem \ref{thm_dense_formes} gives:
\begin{cor}\label{cor_A_vide_forme}For every $p\in [1,\infty)$ sufficiently large, if $M$ is normal then  $\cc^{j,\infty}(\mba)$ is dense  in $ \wca^j_p (M)$ for all $j$.
\end{cor}

In the case where $M$ is not connected along $\delta M \setminus  A$, it is not difficult to see that the conclusion of the above theorem never holds (see \cite[Corollary 3.10]{trace} and its proof).
Of course, we can always normalize $M$ (Proposition \ref{pro_normal_existence} (\ref{item_normal_existence})), which means that we have the following fact, satisfying for many purposes.
\begin{cor}\label{cor_densite_non_normal}
 Let $E$ and $A$ be  definable subsets of $\delta M$. If $\dim E<j$ then, for all $p\in [1,\infty)$ sufficiently large, $\wca^j_\infty(M)\cap \cc_{\mba \setminus A\cup E}^{j,\infty}(M)$ is dense in $\wca^j_p(M,A)$.
\end{cor}

\begin{rem}\label{rem_p=infty}
Theorem \ref{thm_dense_formes} and Corollary \ref{cor_densite_non_normal} (of course) do not hold when $p=\infty$. The successive steps of the proof of the latter theorem however provide, given $\omega\in \wca^j_\infty(M,A)$, families of forms $\omega_\mu$ tending to $\omega$ pointwise (with pointwise convergence of $d\omega_\mu$) and satisfying $||\omega_\mu||_{\wca^j_\infty (M)} \le C$ for some constant $C$ (depending on $\omega$ but not on $\mu$), the  final obtained family having the property of vanishing near $A\cup E$.   This is enough to derive (\ref{eq_specialization}), Corollary \ref{cor_lsp_A} below, as well as Lefschetz duality, in the case $p=\infty$ or $1$. Hodge theorem is less clear in these cases as $L^1$ and $L^\infty$ spaces are not uniformly convex.
\end{rem}

  We make a point here that the above corollary provides approximations in the space $\wca^j_\infty(M)\cap \cc_{\mba \setminus A\cup E}^{j,\infty}(M)$ which are therefore not necessarily smooth up to $\delta M$. This is due to the fact that normalizations do not need to extend smoothly to the closure, but just continuously. 
This however suffices to derive our Lefschetz duality pairing:
 \begin{cor}\label{cor_lsp_A}
 For every $p\in [1,\infty]$ sufficiently large or close to $1$,  we have for every couple of forms $\alpha\in \wca^{j-1}_p(M,A)$ and $\beta\in \wca^{m-j}_{p'}(M,\delta M\setminus A )$, $1\le j\le m$:
 \begin{equation}\label{eq_lsp_A}
 	<d\alpha,\beta>=(-1)^j<\alpha,d\beta>.
 \end{equation}
  Consequently,  $\poi_{M,A}^j$ (see (\ref{eq_poima})) induces a mapping $$\poi_{M,A}^j: H^j_{p} (M,A)\to \Hom (H^{m-j}_{p'}(M,\delta M\setminus A),\R), $$ for every $j$ and each such $p$.
 \end{cor}
\begin{proof}
	It is of course enough to prove the result for $p$ large. By the preceding corollary (with $E:=\emptyset$), there are arbitrarily close approximations of any given $\alpha\in \wca^{j-1}_p(M,A)$ by forms that are compactly supported in $\mba \setminus A$. Since we can pass to the limit, it thus suffices to establish (\ref{eq_lsp_A}) for such a form, which directly comes down from the definition  of $ \wca^{m-j}_{p'}(M,\delta M\setminus A )$ (see Remark \ref{rem_p=infty} if $p=\infty$).
\end{proof}

\subsection{The case $j=0$.}\label{sect_j=0} The proof of Theorem \ref{thm_dense_formes}, which will be inductive on $j$ (see step \ref{ste_j=0} below), will require a lemma in the case $j=0$ (Lemma \ref{lem_traTra}), which  already establishes the theorem in this case (see Remark \ref{rem_traTra}) and will also be of service in section \ref{sect_lp_coh}.
 The following  observation will be useful.

\begin{rem}\label{rem_vanishing_local}
 Vanishing on a subset $A$ of $\delta M$ is a local condition in the following sense: given $\alpha\in \wca^j_p(M)$, $\trd^p \alpha$ vanishes on $A$  if every point $\xo$ of this set has an open neighborhood $\uxo$ in $\mba$ such that $\alpha$ satisfies (\ref{eq_lsp_j}) for all $\beta\in \wca^{m-j-1}_{p',\uxo\cap (M\cup A)}(M)$. This is due to the fact that we can use a partition of unity subordinated to a finite covering of the support (in $\mba$) of a given element of $ \wca^{m-j-1}_{p',M\cup A}(M)$.
	\end{rem}

The proof of Lemma \ref{lem_traTra} will require the following fact, which will also be helpful to define our trace operator in the proof of Theorem \ref{thm_trace_formes}.

\begin{lem}\label{lem_beta}
Assume $\orn\in \delta M$ and let $\ep, N^\eta, M^\eta$ be as in section \ref{sect_the_retraction_r_s_and_R_t}. There is a closed stratifiable differential form $\beta\in \wca^{m-1}_{p}(M^\ep),$ with $p>1$, vanishing in the vicinity of every point of $ \bou(0,\ep)\cap\delta M\setminus \{0\}$ and  such that  for $\eta\in (0,\ep)$:
	\begin{equation}\label{eq_beta_S_geta}
		\int_{N^\eta} \beta =1  \et 	\int_{N^\eta} |\beta| \lesssim 1.
	\end{equation}
\end{lem}
\begin{proof}
	The mapping $H:0*N^\ep \to M^\ep\cup \{0\}$ (provided by applying Theorem \ref{thm_local_conic_structure} to the germ of $M\cup \{0\}$, as explained in section  \ref{sect_the_retraction_r_s_and_R_t}) induces a mapping  $h :(0,1)\times N^\ep \to \mep$, $h (t,x):=H(t x)$. 
	Take then any $\beta_N\in \cc_0^{m-1,\infty}(N^\ep)$  satisfying $\int_{N^\ep}\beta_N=1$,  
	and set \begin{equation}\label{eq_beta_S}\beta:=h^{-1*}\,\pi^*\beta_N,\end{equation} 
	where $\pi: (0,1)\times N^\ep \to N^\ep $ is the canonical projection.
	As $\beta_N$ is compactly supported,  $\beta$ must vanish in the vicinity of every point of $ \bou(0,\ep)\cap\delta M\setminus \{0\}$, and since $ \beta_N$ is smooth, $\beta$ is a stratifiable form (see section \ref{sect_stokes_formula}). 
By construction, for each $\eta<\ep$ we have	$$\int_{N^\eta}\beta=\int_{N^\ep}\beta_N=1. $$
	  Observe also that for every $t \in (0,1) $, $ h$ induces a bi-Lipschitz map $h_{t}:N^\ep \to N^{t\ep}$, $x\mapsto  h(t, x)$. As a matter of fact,  since $\beta_N$ is an $(m-1)$-form on $N^\ep$, we have for almost every $x\in N^\eta$, with $\eta<\ep$
	\begin{equation}\label{eq_hetay}|\beta(x )|\overset{(\ref{eq_beta_S})}{=} | \beta_N \left(h^{-1}_{\eta/\ep}(x)\right)|\cdot  \jac h_{\eta/\ep}^{-1}(x),
	\end{equation} 
which yields the estimate claimed in (\ref{eq_beta_S_geta}). 

 It just remains to check that $\beta\in \wca^{m-1}_{p}(\mep)$ for some $p>1$, which, since this form is closed (for so has to be $\beta_N$), reduces to show that it is $L^{p}$ for such a $p$.
 Since $h$ is bi-Lipschitz on $(t_0 ,1)\times N^\ep$ for every $t_0\in (0,1)$,  in virtue of
	\L ojasiewicz's inequality, there must be $k>0$ such that for almost every $(t,x)\in  (0,1)\times N^\ep $
	\begin{equation}\label{eq_jac_heta}
		t^k\lesssim \jac h_{t}(x).                                                                                                                                                                                                                                                                                                                                                                                       \end{equation}
		We thus can estimate the $L^{p}$ norm of $\beta$ as
	\begin{eqnarray*}||\beta||^{p}_{L^{p}(\mep)}&\overset{(\ref{eq_coarea_sph})}{\lesssim}& \int_0^\ep \int_{N^\eta} |\beta(x)|^{p}\,dx\,d\eta \\
		&\overset{(\ref{eq_hetay})}\lesssim& \int_0^\ep \int_{ N^\eta} |\beta_N \left(h_{\eta/\ep}^{-1}(x)\right)|^{p} \cdot (\jac h_{\eta/\ep}^{-1}(x))^{p}\,dx\,d\eta \\
		&=& \int_0^\ep  \int_{ N^\ep} |\beta_N \left(x\right)|^{p}\cdot  \jac  h_{\eta/\ep}(x)^{1-p}\,dx\,d\eta\\
		& \overset{(\ref{eq_jac_heta})}\lesssim & \int_0^\ep  \eta^{k(1-p)} \left(\int_{x\in N^\ep} |\beta_N (x)|^{p} \,dx\right)d\eta\\
		& =&\left( \int_0^\ep  \eta^{k(1-p)}  d\eta\right)\cdot    ||\beta_N||_{L^{p}(N^\ep)}^{p},
	\end{eqnarray*}
	which is finite for $p<1+\frac{1}{k}$. 
\end{proof}

In   (\ref{eq_trace_fn_plarge}), we defined a trace operator for functions in the case ``$p$ large''. It is natural to wonder whether for such $p$, given $u\in \wca^0_p(M)$, the vanishing of $\trd^p u$ is equivalent to the vanishing of $\tra\, u$. The answer is  positive:

\begin{lem}\label{lem_traTra} Let $A$ be a definable subset of $\delta M$.
For each $p\in [1,\infty]$ sufficiently large we have for all  $u\in W^{1,p}(M)$: $\trd^p u\equiv 0$ on $A$ if and only if $\tra\, u=0$ on this set.
\end{lem}
\begin{proof}Taking a normalization if necessary, we can assume that $M$ is normal. Of course, we have to take $p$ sufficiently large for $\tra$ to be well-defined on $W^{1,p}(M)$ (see (\ref{eq_trace_fn_plarge})).

We first focus on the ``if'' part.  By Theorem \ref{thm_trace}, for every sufficiently large  $p$,  each function $u\in W^{1,p}(M)$ that satisfies $\tra\, u=0$ on $A$ can be approximated by a sequence of $\cc^\infty$ functions $\varphi_\nu$ that vanish in the vicinity of $\adh{A}$. Hence, for each $\omega\in \wca^{m-1}_{p',M\cup A}(M)$, equality (\ref{eq_lsp_j}) holds for the pair $\omega , \varphi_\nu$  (see Remark \ref{rem_vanishing_local}). Passing to the limit in this equality shows that it also holds for the pair $\omega,u$ for each such $\omega$, and thus that $\trd^p u\equiv 0$ on $A$.
 
 To prove the converse,  take  a definably bi-Lipschitz trivial stratification $\Sigma$ of $\mba$ compatible with $A$ and $\delta M$.                                                                                                                                                                                                                                                                                                                                                     Fix  $\xo\in \delta M$ and denote by $S$ the stratum that contains this point.

  By definition of definable bi-Lipschitz triviality, there are an open neighborhood $\uxo$ in $\mba$ and a definable bi-Lipschitz homeomorphism $\Lambda: U_\xo \to (\pi^{-1}(\xo)\cap \adh{M})\times W_\xo$, where $\pi:U_\xo \to S$ is a $\cc^\infty$ retraction and $W_\xo$ is a neighborhood of $\xo$ in $S$ (here $\dim S$ may be $0$, in which case $\Lambda$ is the identity).
 We will identify $U_\xo \cap M$ with $F\times W_{x_0}$ and $\pi$ with the canonical projection.   Let $F:=\pi^{-1}(\xo)\cap M$, $F^\eta:=F\cap \bou(\xo ,\eta)$, and $G^\eta:=F\cap  \sph(\xo,\eta)$.
  
 Let $\beta\in \wca^{m-k-1}_{q}(F^\ep)$, $q>1$, be obtained by applying  Lemma \ref{lem_beta} with $M=F$, which is a smooth definable manifold-germ at $x_0$ of dimension $(m-k)$, where $k:=\dim S$.  Below, we sometimes regard $\beta$ as a form on $F^\ep\times W_\xo$, identifying it with $\mu^*\beta$, where $\mu:F^\ep\times W_\xo\to F^\ep $ is the canonical projection. As the result is local in $\delta M$ (we can use a partition of unity), we can assume that   $u$ is compactly supported in $U_\xo$. 
 By Theorem \ref{thm_trace}, $u$ admits arbitrarily  close approximations  $u_l\in \cc^\infty(\mba)$, $l\in \N$, that we may assume  to be compactly supported in $\uxo$ (since so is $u$).  For every $\varphi\in \cc_0^{k,\infty} (W_\xo)$,  we then can write, assuming $p\ge q'$:
 \begin{eqnarray}\label{eq_traphi}
 	\int_S \tra\, u\cdot  \varphi &=&    \lim_{l\to \infty} 	\int_S \tra \,u_l \cdot \varphi\nonumber\\
 	&\overset{(\ref{eq_beta_S_geta})}=&  \lim_{l\to \infty} \int_{y\in W_\xo} \int_{x\in G^\eta}  \tra\,u_l(y)\cdot  \beta(x)\wedge \varphi(y)\quad \mbox{ (for $\eta>0$ small)}\nonumber\\
 	 	&=&  \lim_{l\to \infty} \lim_{\eta \to 0} \int_{y\in W_\xo} \int_{x\in G^\eta}  u_l(x,y)\cdot \beta(x)\wedge \pi^*\varphi(x,y),
 	  \end{eqnarray}
since $u_l$ is smooth.  As  $\beta$ is  stratifiable, so is the form $ u_l(x,y)\cdot \beta(x)\wedge \pi^*\varphi(x,y)$ (since $u_l$ and $\varphi$ are smooth), which  means that it satisfies Stokes' formula (\ref{eq_stokes_stratified_leaf}). Note that for $\eta>0$ small, $P^\eta:=F\times W_\xo \setminus (F^\eta \times W_\xo)$ is a manifold with boundary $G^\eta \times W_\xo$ on which $u_l(x,y)\cdot \beta(x) \wedge \pi^*\varphi(x,y)$ is compactly supported (for so is  $u_l$ in  $\uxo$  and $\adh{\supp\, \beta} $ can only meet $\delta M$ at points of $S$).  By Stokes' formula  (\ref{eq_stokes_stratified_leaf}), we have \begin{eqnarray} \label{eq_tr_geta}
\lim_{\eta\to 0}	\int_{G^\eta \times W_\xo}u_l(x,y)\cdot \beta(x) \wedge \pi^*\varphi(x,y) =\lim_{\eta\to 0}\int_{P^\eta} d(u_l(x,y)\cdot \beta(x) \wedge \pi^*\varphi(x,y))\nonumber\\
=\int_{\uxo} d(u_l(x,y)\cdot \beta(x) \wedge \pi^*\varphi(x,y))=< \trd^p u_l, \pi^*\varphi \wedge\beta>',
\end{eqnarray}
by definition of $\trd^p$.  Passing to the limit as $l$ tends to infinity, we get thanks to (\ref{eq_traphi})
   \begin{eqnarray*}\label{eq_int_betauphi}
 	\int_S \tra\, u\cdot  \varphi	=  \lim_{l\to \infty} < \trd^p u_l, \pi^*\varphi \wedge\beta>'=< \trd^p u, \pi^*\varphi \wedge\beta>'.
 \end{eqnarray*}
We claim that if $\trd^p u\equiv 0$ on $S$ then the right-hand side is equal to $0$ for all $\varphi\in \cc_0^{k,\infty} (W_\xo)$, which will prove that $\tra u=0$ on $S$.  To see this, let $\phi:U_\xo \to \R$ be a $\cc_0^\infty$ function which is $\equiv 1$ near  $\supp_{\mba} u$. As $\phi\cdot \pi^*\varphi \wedge\beta$ has compact support in $M\cup S$ and $\trd^p u$ vanishes on $S$ by assumption, (\ref{eq_lsp_j}) holds true for the pair $u,\phi\cdot \pi^*\varphi \wedge\beta$. But
 since $\phi$ is $1$ near $\supp_\mba u$, this shows (\ref{eq_lsp_j})   for the pair $u,  \pi^*\varphi \wedge\beta$, which is our claim.
\end{proof}

\begin{rem}\label{rem_traTra}
 Since $\wca^0_p (M)=W^{1,p}(M)$, the above lemma entails, in virtue of Theorem \ref{thm_trace}, that Theorem \ref{thm_dense_formes} holds in the case $j=0$.
\end{rem}

\begin{subsection}{Mollifying with parameters.}\label{sect_mol}
One ingredient of the smoothing
process of the proof of Theorem \ref{thm_dense_formes} will be  a ``mollifying operator along the strata''. We give preliminary details in this subsection.

Let $F\subset\R^n$ be a subanalytic manifold  and $p\in [1,\infty)$.
For  $\omega\in L^j_p(F\times\R^k)$ and $\psi\in L^1(\R^k)$,
we define a $j$-form 	$\omega*_k\psi$ on $F\times\R^k $ by setting for almost every $(x,y)\in F\times\R^k$ and $\xi\in \otimes^j (T_x F\times \R^k)$
\begin{equation}\label{mol}
	\omega*_k\psi(x,y)(\xi)=\int_{\R^k}\omega(x,z)(\xi)\cdot \psi(y-z)\,dz=\int_{\R^k}\omega_{\xi}(x,z)\cdot \psi(y-z)\,dz.
\end{equation}

\begin{lem}\label{lem_conv} Let  $\omega\in L^j_p(F\times\R^k)$, $p\in [1,\infty)$.
	\begin{enumerate}[(i)]
		\item For  $\psi\in L^1(\R^k)$,
		we have\begin{eqnarray}\label{eq_young_k}
			||\omega*_k\psi||_{L^p(F\times\R^k)}\lesssim ||\omega||_{L^p(F\times \R^k)} ||\psi||_{L^1( \R^k)}.
		\end{eqnarray}
		Moreover,  for  $\beta\in L^{l+k-j}_{p'}(F\times \R^k)$, $l:=\dim F$, we have if we let $\adh\psi(y):=\psi(-y)$
		\begin{equation}\label{eq_conv_currents}
			<\omega*_k\psi ,\beta>=	<\omega,\beta*_k\adh\psi>.
		\end{equation}		
			\item	If $\psi\in W^{1,1}(\R^k)$ then for all $i\le k$
 \begin{equation}\label{eq_conv_der_psi}\frac{\pa }{ \pa y_i}(\omega*_k\psi)=\omega*_k\frac{\pa \psi}{ \pa y_i}.\end{equation}
\item If $d\omega$ (resp. $\frac{\pa \omega}{\pa y_i}$ for all $i\le k$) is $L^p$ then \begin{equation}\label{eq_conv_der}
d(\omega*_k \psi)=d\omega*_k \psi \qquad \quad\mbox{(resp. $\frac{\pa}{\pa y_i}(\omega*_k\psi)=\frac{\pa \omega}{\pa y_i}*_k \psi$).} 		                                                                                         \end{equation}

		 \item\label{item_conv_lem} Let $\varphi:
\R^k\to \R$ be a nonnegative smooth compactly supported function satisfying $\int \varphi =1$, and set $\varphi_\sigma(x):=\frac{1}{\sigma^k}\varphi(\frac{x}{\sigma})$. If $d\omega$  is $L^p$ then $\omega*_k\varphi_\sigma$ tends to $\omega$ in $\wca^j_p(F\times \R^k)$ as $\sigma\to 0$.
	\end{enumerate}
\end{lem}
\begin{proof}
	In this proof, given any mapping $\alpha$ on $F\times \R^k$ and $x\in F$ (resp. $y\in \R^k$), we denote by $\alpha(x,\cdot)$ (resp. $\alpha(\cdot,y)$) the mapping defined as $\R^k\ni z\mapsto \alpha(x,z)$ (resp. $F\ni z\mapsto \alpha(z,y)$). Notice that for each $x\in F$, (\ref{mol}) may be rewritten as
	\begin{equation}\label{eq_ompsi_norm}
	\omega*_k\psi(x,\cdot)(\xi)=\omega_\xi(x,\cdot) * \psi,
\end{equation}where $*$ is the usual convolution product, so that we can write, using Young's inequality,
 	\begin{equation*}
\int_{\R^k}|\omega*_k\psi(x,y)(\xi)|^pdy\overset{(\ref{eq_ompsi_norm})}= \int_{\R^k}|(\omega_\xi(x,\cdot) * \psi)(y)|^pdy\le || \omega_\xi(x,\cdot) ||^p_{L^p(\R^k)}\cdot ||\psi||^p_{L^1( \R^k)}.
 	\end{equation*}
 Since we can take $\xi$ successively equal to all the vectors of some basis of  $\otimes^j (T_x F\times \R^k)$, this entails for almost every $x$:
 $$\int_{\R^k}|\omega*_k\psi(x,y)|^p \,dy\lesssim || \omega(x,\cdot) ||^p_{L^p(\R^k)}\cdot ||\psi||_{L^1( \R^k)}^p,$$
with a constant  independent to $x$ and $\omega$.	By Fubini's Theorem, after integrating with respect to $x$, we get (\ref{eq_young_k}). Moreover,  Fubini's  Theorem also  yields (\ref{eq_conv_currents}) ((\ref{eq_young_k}) establishes that  the involved forms are integrable).

 Equalities (\ref{eq_conv_der_psi}) and (\ref{eq_conv_der}) follow from the fact that we can differentiate under the integral (\ref{mol}), when suitable integrability conditions hold. Thanks to (\ref{eq_conv_der}),  $(\ref{item_conv_lem})$ reduces to prove that $\omega*_k \varphi_\sigma$ tends to $\omega$ in the $L^p$ norm as $\sigma$ tends to zero for all $\omega\in L^j_p(F\times \R^k)$ ((\ref{eq_young_k}) guarantees that $ \omega *_k \varphi_\sigma$ is $L^p$). By (\ref{eq_young_k}) and the density of $\cc^\infty_0$ in the $L^p$ spaces, we can assume $\omega\in \cc^{j,\infty}_0(F\times \R^k)$. For such $\omega$, as $(\omega_\xi(x.\cdot)*\varphi_\sigma)(y)$ tends to $\omega_\xi(x,y)$ uniformly in $\xi$ (unit multivector), $x\in F$, and $y\in\R^k$ (see \cite{adams}, p. 37), the needed fact  follows from (\ref{eq_ompsi_norm}).
\end{proof}
\end{subsection}

\begin{subsection}{Proof of Theorem \ref{thm_dense_formes}}\label{sect_proof_dense_formes} Most of the arguments used below apply for $p$ large enough, and, for the sake of being more concise, we will not mention it.
 Fix a definably bi-Lipschitz trivial stratification $\Sigma$ of $\mba$  compatible with $\adh E, A, \adh A$, and $\delta M$.  We will assume that $M$ is a stratum of $\Sigma$ (see Remark \ref{rem_M_stratum}).

The proof is split into seven steps.  The first six steps will be purely local at a boundary point, that we will assume to be the origin in order to simplify notations. In the last step, we will glue together these local approximations  by means of a partition of unity.

Let us thus start by introducing some notations necessary to the first six steps. Assume $\orn\in \delta M$ and let $S\in\Sigma$ be the stratum containing this point. As our problem will be local, we will identify $S$ with  $\{0\}\times (-1,1)^k$, where $k:=\dim S$.

By definition of definable bi-Lipschitz triviality, there is an open neighborhood $U$ of $\orn$ for which we can find a definable bi-Lipschitz homeomorphism
      \begin{equation} \label{eq_Lambda}
  \Lambda: U \to (\pi^{-1}(0)\cap \adh{M})\times (-1,1)^k,
      \end{equation}
 satisfying $\Lambda(\pi(x))=\pi(\Lambda(x))$, where $\pi:U \to U\cap S$ is a $\cc^\infty$ definable retraction. We then set
     \begin{equation} \label{eq_F}
F:=\pi^{-1}(0)\cap M \;\;\et \;\;F^\eta:=F\cap \bou(0 ,\eta),\quad \eta>0,
    \end{equation}
    as well as
 \begin{equation}\label{eq_UV} U^\eta:=F^\eta\times (-1,1)^k \et V^\eta :=(\adh F\cap \bou(0,\eta))\times (-1,1)^k.  \end{equation}
   Apply Theorem \ref{thm_local_conic_structure} to the germ of $\adh F$ at $0$, and let $r_s^F:\adh F\cap \bou(0,\ep) \to \adh F\cap \bou(0,\ep) $, $\ep>0$, be a retraction by deformation given by this theorem. We will assume that $r_s^F$  preserves the strata (for $s\in (0,1]$, see Remark \ref{rem_preserve}), which entails that if we set
 \begin{equation}\label{eq_Aep}
 	A^\ep := V^\ep \cap \Lambda(A) \,\;\et\,\; A_0^\ep :=A^\ep \cap \pi^{-1}(0)
 \end{equation}
 then $r_s^F$ preserves $A_0$. Note that, by Definition \ref{dfn_trivial} (ii), $\Lambda$ trivializes $A$ along $S$, i.e.
 \begin{equation}\label{eq_A_trivial}
 	A^\ep=A_0^\ep\times (-1,1)^k.
 \end{equation}

  We  work on $U^\ep$ during the first $6$ steps and  pull-back by means of $\Lambda$ in step $7$ the approximations constructed in the preceding steps. The idea of steps $2-6$  is to focus on the differential $j$-forms on $U^\ep$ whose partial derivatives along $S$ are $L^p$ (step \ref{ste_dense} establishes that such forms are dense), constructing some approximations of these forms on $U^\ep$ by induction on $j$ (in step \ref{ste_type_0}) and on the ``type''  of the considered form (in step \ref{ste_induc_type}), which is defined as follows.
 
\medskip

 \noindent{\bf Forms of type $l$.} We denote by $e_1,\dots, e_k$ the canonical basis of $\{0_{\R^n}\}\times \R^k$.
 We will say that $\omega\in \wca^j_p(U^\ep)$ is {\bf of type $l$} if for every $i>l$ we have for all $(x,y)\in U^\ep$ and all  $\xi\in\otimes^{j-1} T_{(x,y)}U^\ep$:
$$\omega(x,y) \xi \otimes e_i= 0.$$
Of course, if $\omega$ is of type $l$, it is also of type $l'$ for all $l'>l$. 
We can regard a form of type $0$ as a family of forms $\omega_y$ on $F^\ep$, parameterized by $y\in (-1,1)^k$, setting 
 \begin{equation}\label{eq_type_param} \omega_y(x):=\omega(x,y),\quad \mbox{ for   $(x,y)\in F^\ep\times (-1,1)^k$}.\end{equation}


 Given a measurable differential form $\omega(x,y)$ on $U^\ep$ 
and a multi-index $d\in \N^k$, we  set (as usual) $\frac{\pa^{|d|} \omega}{\pa y^d}(x,y):=\frac{\pa^{d_1}}{\pa y_1^{d_1}}\cdots \frac{\pa^{d_k}\omega}{\pa y_k^{d_k}} (x,y),$
with $\frac{\pa^0 \omega}{\pa y_i^0} (x,y):=\omega$.
 As explained just above, we will work with forms that have $L^p$ partial derivatives with respect to $y$ (at all orders). Namely, given $0\le j \le m$, $p\in [1,\infty)$, and $0\le l\le k$, we introduce:

\noindent{\bf The space $\E^{l,j}_p$.}   We denote by $\E^{l,j}_p$ the space of $j$-forms  $\omega $ of type $l$ on $U^\ep$ satisfying
\begin{enumerate}[(i)]
\item\label{item_der_par} $\frac{\pa^{|d|} \omega}{\pa y^d}(x,y) \in  \wca^j_p(U^\ep, A^\ep)$, for all $d\in \N^k$.
\item\label{item_vi} $\omega \equiv 0$ near $\{(x,y)\in F^\ep \times [-1,1]^k: \exists i\le k ,\;\;  y_i=-1\}$.
\end{enumerate}
We then put $\E^{l,j}_{p,V^\ep}:=\{\omega \in \E^{l,j}_p: \omega\mbox{ has compact support in $V^\ep$}\}$.

  The above condition (\ref{item_vi}) is required to establish (\ref{eq_A_l_hom_op}) in step \ref{ste_induc_type}. The definition of  $\E^{l,j}_p$ does not demand forms to be compactly supported in $V^\ep$ (which would be less technical than the above  assumption (\ref{item_vi})) for this requirement would not be preserved by the operator $\mathbf{A}_l$ constructed in step \ref{ste_induc_type}.

 The strategy of the first step is that, since $U^\ep=F^\ep \times (-1,1)^k$, we can approximate every form of  $ \wca^j_{p,V^\ep}(U^\ep, A^\ep)$ by elements of  $\E^{k,j}_{p,V^\ep}$ using the mollifying operators constructed in section \ref{sect_mol}.


\medskip

\begin{ste}\label{ste_dense}
 $\E^{k,j}_{p,V^\ep}$ is dense in $ \wca^j_{p,V^\ep}(U^\ep, A^\ep)$.
\end{ste}
\begin{proof}Take  $\omega\in  \wca^j_{p,V^\ep}(U^\ep,A^\ep)$ and let $\varphi_\sigma$ be as in $(\ref{item_conv_lem})$ of Lemma \ref{lem_conv}. Thanks to this lemma, we just have to check that  $ \omega*_k \varphi_\sigma\in \E^{k,j}_{p,V^\ep}$ for $\sigma>0$ small (see (\ref{mol}) for $*_k$).  Since $\omega$ is compactly supported in $V^\ep$, so is $ \omega*_k \varphi_\sigma$ for $\sigma>0$ small,  which yields property (\ref{item_vi})  of the definition of $\E^{k,j}_{p}$ for  $ \omega*_k \varphi_\sigma$. Moreover, by (\ref{eq_conv_der_psi}) and (\ref{eq_conv_der}), $\frac{\pa^{|d|} }{\pa y^d}(\omega*_k \varphi_\sigma)$, as well as its exterior differential, is $L^p$, for all $d\in \N^k$.

It remains to show that $\trd^p\left(\frac{\pa^{|d|}}{\pa y^d}(\omega*_k  \varphi_\sigma)\right)\equiv 0$ on $A^\ep$ for  $\sigma>0$ small and $d\in \N^k$, which, thanks to (\ref{eq_conv_der_psi}), reduces to show that $\trd^p(\omega*_k  \phi)$ vanishes on $A^\ep$ for every  $\phi\in \cc_0^\infty(\R^k)$ supported in a small neighborhood of the origin. Take $\beta\in \wca^{m-j-1}_{p',U^\ep\cup A^\ep}(U^\ep)$ and such $\phi$, and note that since  $\beta$ is zero near $\delta U^\ep\setminus A^\ep$, by (\ref{eq_A_trivial}),  $ \beta *_k \phi$ must be zero as well near this set (if $\supp \phi$ is small enough). As $\trd^p \omega$ vanishes on $A^\ep$, we can write (see (\ref{eq_conv_currents}) for $\adh\phi$)
$$< \omega*_k \phi,d\beta>\overset{(\ref{eq_conv_currents})}=<\omega, d\beta*_k \adh\phi>\overset{(\ref{eq_conv_der})}=<\omega, d(\beta*_k \adh\phi)>$$
$$\overset{(\ref{eq_lsp_j})}=(-1)^{j+1}<d\omega, \beta*_k \adh\phi>\overset{(\ref{eq_conv_currents})}=(-1)^{j+1}<d\omega*_k \phi, \beta>\overset{(\ref{eq_conv_der})}=(-1)^{j+1}<d(\omega*_k \phi), \beta> ,$$
which shows that $\trd ^p (\omega*_k \phi)$ vanishes on $A^\ep$, completing  step \ref{ste_dense}.
\end{proof}


 To avoid any confusion, we will denote by $d^F$ (resp. $d^S$) the exterior differential operator of $F$ (resp. $S$), while $d$ will stand for the operator of $U^\ep$.
 Given a differential form $\omega(x,y)$  on $U^\ep$, we  denote by $\hat d^S$ the ``extension'' of $d^S$, defined as:
 $$\dsh \omega(x,y):=\sum_{i=1} ^k dy_i \wedge \frac{\pa \omega}{\pa y_i}(x,y).$$
 Set then
 \begin{equation}\label{eq_d_dlf}
 	\dfh \hskip-0.5mm\omega:=  d\omega-\dsh \omega.
 \end{equation}
 Recall that a form $\omega$ of type $0$ can be regarded as a family of forms $\omega_y$ on $F^\ep$, $y\in (-1,1)^k$ (see (\ref{eq_type_param})). Given such a form $\omega$, we then have:
 \begin{equation}\label{eq_dsf}\dfh \hskip-0.5mm\omega(x,y):=  d^F\omega_y (x), \quad \mbox{for}\quad (x,y)\in F^\ep \times (-1,1)^k.\end{equation}

Let us emphasize that if $\omega \in \E^{l,j}_p$, $l\le k$, then, due to (\ref{item_der_par}) of the definition of  $\E^{l,j}_p$, $\dsh \omega$ must be $L^p$, and thus, so is $ \dfh \hskip-0.5mm\omega$. We also have:

 \begin{ste}\label{ste_fa}
  If $\omega \in \E^{0,j}_p$ then for almost every $y\in (-1,1)^k$, we have $\omega_y \in \wca^j_p (F^\ep,A_0^\ep)$ (see (\ref{eq_type_param}) for $\omega_y$ and (\ref{eq_Aep}) for $A_0^\ep$).
 \end{ste}
\begin{proof}
 Take such a form $\omega$.  In virtue of Fubini's Theorem, $\omega_y \in \wca^j_p (F^\ep)$ for almost every $y\in (-1,1)^k$. We thus just have to check that it vanishes on $A_0^\ep$ (in the trace sense).
 
  For this purpose, take  $\beta\in \wca^{m-k-j-1}_{p', F^\ep \cup A_0^\ep}(F^\ep)$.
For $y\in S$ fixed and $\sigma>0$, define an $(m-j-1$)-form $\beta^y_\sigma$ on $F^\ep \times \R^k$ by setting
 \begin{equation}\label{eq_betaeta}\beta_{\sigma}^{y}(x,z):=\varphi_\sigma (y-z)\cdot \beta(x) \wedge dy_1\wedge \dots \wedge dy_k,  \end{equation}
   where $\varphi_\sigma$ is as in Lemma \ref{lem_conv} $(\ref{item_conv_lem})$.
 Observe that (since  $\dsh \beta_\sigma^y =0$) \begin{equation}\label{eq_dbetat}
 d \beta_{\sigma}^{y}(x,z)\overset{(\ref{eq_d_dlf})}=\dfh  \beta_\sigma^{y}(x,z)= \varphi_\sigma (y-z)d^F\beta(x) \wedge dy_1\wedge \dots \wedge dy_k,
             \end{equation}
so that, by (\ref{eq_A_trivial}),  $\beta^y_\sigma$ must belong to $ \wca^{m-k-j-1}_{p', U^\ep \cup A^\ep}(U^\ep)$, for $\sigma>0$ small.                             As $\trd^p \omega$ vanishes on $A^\ep$, this entails that  for almost every $y$ in $S$ we have
             \begin{equation}\label{eq_fep_van_sigma}<\omega,d\beta_\sigma^{y}>=(-1)^{j+1} <d\omega,\beta_\sigma^{y}>\overset{(\ref{eq_d_dlf})}=(-1)^{j+1}<\dfh \omega, \beta_\sigma^{y}>,             \end{equation}
since $\dsh \omega\wedge\beta_\sigma^y =0$. 
We first show the desired fact for the form $\omega*_k \varphi_\sigma$, as follows:
\begin{eqnarray}\label{eq_*y}
<(\omega*_k \varphi_\sigma)_y,d^F\beta>_{F^\ep}&\overset{(\ref{mol})}=&\int_{x\in F^\ep}\int_{\R^k} \varphi_\sigma(y-z)\omega(x,z)\wedge d^F\beta(x)\, dz \;\;\nonumber\\
&\overset{(\ref{eq_dbetat})} =&\int_{(x,z)\in F^\ep\times\R^k}\omega(x,z)\wedge d\beta_\sigma^y(x,z)\nonumber\\
 &\overset{(\ref{eq_fep_van_sigma})} =&(-1)^{j+1}\int_{ F^\ep\times\R^k }\dfh\omega\wedge \beta_\sigma^y \nonumber\\
 &\overset{(\ref{eq_betaeta})} =&(-1)^{j+1}\int_{x\in F^\ep} \left(\int_{\R^k}\varphi_\sigma(y-z)\dfh \omega_z(x)\wedge \beta(x) \,  dz\right)\nonumber\\
&\overset{(\ref{mol})}=&(-1)^{j+1}<(\dfh \omega*_k \varphi_\sigma)_{ y}, \beta>_{F^\ep}.
\end{eqnarray}
   Lemma \ref{lem_conv} $(\ref{item_conv_lem})$ shows that $\omega*_k \varphi_\sigma$  tends to $\omega$  in  $\wca_p^j(U^\ep)$, which entails that $ (\omega*_k \varphi_\sigma)_{y}$  tends to $\omega_{y}$  in  $\wca^j_p(F^\ep)$  for almost every $y$ in $(-1,1)^k$.
 Passing to the limit as $\sigma\to 0$ in (\ref{eq_*y}) thus yields $\omega_{y} \in \wca^j_p (F^\ep,A_0^\ep)$, for almost every $y$, completing  step \ref{ste_fa}.
%
 \end{proof}

Recall that we have applied Theorem \ref{thm_local_conic_structure} to $\adh F$, which provided a retraction by deformation $r^F$, which now gives rise  to an operator  $\rfr^F$ (defined as in  (\ref{eq_rfr}) with $r=r^F$).
Since forms of type $0$ can be regarded as families of forms on $F^\ep$ (like in  (\ref{eq_type_param})),  $\rfr^F$ induces in turn
  an operator $\rfh$ on $\E^{0,j}_p$  for all $j$ defined by:
 \begin{equation}\label{eq_rfh_dfn}
(\rfh \omega)_y:= \rfr^F\hsk \omega_y, \quad y\in(-1,1)^k.
 \end{equation}
Furthermore, by (\ref{eq_r_borne_m}) and Proposition \ref{pro_rR_hom_op},  there is  $C>0$ such that for each $\omega \in \E^{0,j}_p$, $j\ge 1$, we have for almost all $y\in (-1,1)^k$ and all $\eta\le \ep$:
\begin{equation}\label{eq_rf_ueta}||\rfr^F \hsk\omega_y||_{L^p(F^\eta)} \le C\eta ||\omega_y||_{L^p(F^\eta)}
\et  ||d^F\rfr^F\hsk \omega_y||_{L^p(F^\eta)} \le C ||\omega_y||_{\wca^j_p(F^\eta)}.\end{equation}
By (\ref{eq_rfh_dfn}) and (\ref{eq_dsf}), integrating with respect to $y$, we get for each such $\omega$:
\begin{equation}\label{eq_rfh_ueta}
 ||\rfh \omega||_{L^p(U^\eta)} \le C \eta  ||\omega||_{L^p(U^\eta)}  \et
 ||\dfh \rfh \omega||_{L^p(U^\eta)} \le C   ||\omega||_{\wca^j_p(U^\eta)}.
\end{equation}


\begin{ste}\label{ste_df}If  $\omega\in \E^{0,j}_p $, $j\ge 0$, then $\dfh \omega\in \E^{0,j+1}_p $ and $\rfh \omega\in \E^{0,j-1}_p $.
\end{ste}
\begin{proof} Fix  $\omega\in \E^{0,j}_p $, $ j\ge 0$. We first focus on $  \rfh  \omega$.
 As $\rfh$ commutes with $\frac{\pa^{|d|} }{\pa y^d} $ for all $d\in \N^k$,  (\ref{eq_rf_ueta}) implies that for $y\in (-1,1)^k$ and  all such $d$ we have:  $$||\,\big{(}\frac{\pa^{|d|}\rfh \omega}{\pa y^d}\big{)}_y\,||_{\wca^{j-1}_p(F^\ep)}\lesssim ||\,\big{(}\frac{\pa^{|d|} \omega}{\pa y^d}\big{)}_y\,||_{\wca^j_p(F^\ep)},$$
with a constant independent of $y$.
 This entails that  $\frac{\pa^{|d|}\rfh  \omega }{\pa y^d}$ belongs to $L^{j-1}_p(U^\ep)$ for all $d$.  This also shows that $\dsh \rfh  \omega$ as well as all its partial derivatives with respect to $y$ are $L^p$, which, thanks to (\ref{eq_d_dlf}) and (\ref{eq_rfh_ueta}), implies  that so is $d\rfh  \omega$ (as well as its partial derivatives).  
 
 It thus just remains to check that $\trd^p (\frac{\pa^{|d|}}{\pa y^d} \rfh  \omega)$  vanishes on $A^\ep$ for all $d\in \N^k$. As $\rfh$ commutes with  $\frac{\pa}{\pa y_i}$, we can assume $d=0$.
 Thanks to step \ref{ste_fa}, we already know that $\trd^p \omega_y\equiv 0$  on $A_0^\ep$ for almost every $y$, which,  by Proposition \ref{pro_r_preserve_vanishing} (applied with $M=F$ and hence $\rfr=\rfr^F$), entails that $\trd^p (\rfh \omega)_y\equiv 0$  on $A_0^\ep$. Hence, if $\alpha \in \wca^{m-k-j}_{p',F^\ep \cup A^\ep_0}(F^\ep)$ then we have for almost every $y\in S$:
\begin{equation}\label{eq_vanish0}
 <d^F\hsk(\rfh\omega)_y,\alpha>= (-1)^j < (\rfh \omega)_y,d^F\hsk\alpha >.
\end{equation}
We have to prove that for all $\beta\in\wca^{m-j}_{p',U^\ep\cup A^\ep} (U^\ep)$,
\begin{equation}\label{eq_step_df}
 <d\,\rfh  \omega,\beta>=(-1)^j<\rfh \omega, d \beta>.
\end{equation}
 Take such a form $\beta$ having compact support in $V^\ep$ (see Remark \ref{rem_vanishing_local}). Possibly replacing it with $\beta*_k \varphi_\sigma$, where $*_k$  is the convolution defined in (\ref{mol}) and $ \varphi_\sigma$ is as in Lemma \ref{lem_conv} $(\ref{item_conv_lem})$, and eventually passing to the limit as $\sigma\to 0$ (see Lemma \ref{lem_conv}), we see that we can assume that $\frac{\pa^{|d|} \beta}{\pa y^d}$ is $L^p$ for all $d\in \N^k$.  Moreover,   thanks to (\ref{eq_d_dlf}), it suffices to show (\ref{eq_step_df}) for the operators $\dfh$ and $\dsh $ (instead of $d$).  For the latter operator, as $\rfh$ commutes with $\frac{\pa}{\pa y_i}$ for all $i$ and $\rfh \omega$ is smooth with respect to $y$ (and $\beta$ was assumed to be compactly supported in $V^\ep$), this easily follows by integration by parts with respect to the $y_i$'s. To show it for the operator $\dfh$, observe that,
 since $\omega$ is of type $0$,  if we set $\bar\beta:=\beta_{e_1\otimes \dots\otimes e_k}$ (we recall that $e_1, \dots, e_k$ is the canonical basis of $\{\orn\} \times \R^k$), we have:
  $$(\dfh\hskip0mm\rfh\,\omega \wedge\beta)_{e_1\otimes \dots\otimes e_k}= \dfh\hskip0mm\rfh\,\omega\wedge \bar\beta \et \rfh\omega\wedge \dfh\hskip-1mm\beta=(\rfh\omega\wedge \dfh\hskip-1mm\bar\beta)_{e_1\otimes \dots\otimes e_k} .$$
As a matter of fact, (\ref{eq_step_df}) for $\dfh$ ensues after applying (\ref{eq_vanish0}) with $\alpha=\bar\beta_y$, for each $y\in S$,  and integrating it with respect to $y$ (since $(\rfh  \omega)_y=\rfr^F \omega_y$ and  $(\dfh\rfh  \omega)_y=d^F\rfr^F \omega_y$, see (\ref{eq_rfh_dfn})).

We now show that $\dfh\hsk\omega\in \E^{0,j+1}_p$. Since, by  (\ref{item_der_par}) of the definition of $\E^{0,j}_p$,  $\frac{\pa  \omega}{\pa y_i}(x,y) \in \wca^j_p(U^\ep, A^\ep)$, 
for all $i$, we see that $\dsh \omega$ belongs to $\wca^{j+1}_p(U^\ep, A^\ep)$, and consequently, so does $\dfh\hsk\omega=d\omega -\dsh \omega$. As $\dfh$ commutes with $\frac{\pa }{\pa y_i} $ for all $i$, the same argument applies to show $\frac{\pa^{|d|} \dfh \hsk\omega}{\pa y^d}(x,y) \in \wca^{j+1}_p(U^\ep, A^\ep)$, for all $d\in \N^k$, completing step \ref{ste_df}.
\end{proof}


  Given a (nonzero) differential form $\omega$ on an open subset of $M$, let then
\begin{equation}\label{eq_kappa}
  \kappa(\omega):= \min \{ i\in\{0,\dots,m\} :  \supp_\mba \, \omega  \cap S\ne\emptyset, \mbox{ with $S\in \Sigma$ and $\dim S=i$} \}.
\end{equation}
  In step \ref{ste_final}, we will argue by decreasing induction on $\kappa(\omega)\in \{0,\dots,m \}$. Below, we will use $\kappa(\omega)$ although $\omega$ is defined on $U^\ep$ (instead of $M$) to in fact mean $\kappa(\Lambda^* \omega)$ where $\Lambda $ is as defined in (\ref{eq_Lambda}). This is because we identify strata with their images under $\Lambda$.
 
 As step \ref{ste_final} will be carried out by decreasing induction on $\kappa(\omega)$,  we will then know, given a form $\omega$ on $U^\ep$, that we can approximate the forms $\alpha$ satisfying $\kappa (\alpha)>\kappa (\omega)=k=\dim S$.  This is why steps \ref{ste_j=0}, \ref{ste_type_0}, and \ref{ste_induc_type} consist of  proving by induction on $l$ (see (\ref{item_kappa}) below)
   that we can approximate a form $\omega\in \E^{l,j}_p$ by a form $\blj(\omega)$ that can be written as the sum of a form $\alpha_\mu$ satisfying $\kappa(\alpha_\mu)>k$ and a smooth form $\beta_\mu$:
 
 \medskip 
 
 \noindent {${(\mathbf H_l})$.} For each $j\in \{0,1,\dots, m\}$ and each $\mu>0$, there is a linear mapping $\blj: \E^{l,j}_p\to \E^{l,j}_p$ such that for every $\omega\in \E^{l,j}_p$, we have:
 \begin{enumerate}
    \item\label{item_conv} $ \blj(\omega)$ tends to  $\omega$ in ${\wca^j_p(U^\ep)}$  as $\mu$ tends to zero.
\item\label{item_kappa}
  $\blj(\omega)=\alpha_\mu +  \pi^*\beta_\mu  ,$
  with $\alpha_\mu \in\wca^j_p(U^\ep,A^\ep)$ such that $\kappa(\alpha_\mu)>k=\dim S$ and $\beta_\mu\in \cc^{j,\infty}                                                                                                                   (S)$. Moreover, $\beta_\mu=0$ if $S\subset \adh{A}$ ($\beta_\mu$ is indeed the trace of $\blj(\omega)$ on $S$).
  \item \label{item_pay} $\blj$ commutes with $\frac{\pa }{\pa y_i}$ for all $i$, that is to say, $\frac{\pa }{\pa y_i} \blj (\omega)=\blj (\frac{\pa \omega}{\pa y_i} )$.
 \end{enumerate}

We will first  define $\mathbf{B}^{0,j}_\mu$ claimed in {${(\mathbf H_{0}})$} by induction on $j$ (in steps  \ref{ste_j=0} and \ref{ste_type_0}) and then  define  $\mathbf{B}_\mu^{l,j}$  by induction on $l$ (in step  \ref{ste_induc_type}). We recall that the assumption ``$p$ large enough'' is omitted in this proof. The definition of $\mathbf{B}^{l,j}_\mu$  will involve the  ``cut-off function'' $\hat\psi_\mu:U^\ep \to \R$, defined for $(x,y)\in U^\ep$ as:
\begin{equation}\label{eq_hatpsi_mu}
 \hat\psi_\mu(x,y):=\psi_\mu(|x|),
\end{equation}
where $\psi_\mu$ is as displayed in (\ref{eq_psieta}).


 \begin{ste}\label{ste_j=0}  We construct the operator  $\mathbf{B}_\mu^{0,j}$ appearing in {${(\mathbf H_{0}})$} in the case $j=0$.
 \end{ste}
 \begin{proof}
            We claim that for each $\mu>0$ small the operator
$$ \mathbf{B}_\mu^{0,0}(u)(x,y):=\hat\psi_\mu(x,y) \cdot  \rfh\, \dfh u(x,y)+\tra\, u(y),\qquad u\in  \E^{0,0} _p,$$
for $(x,y)\in F^\ep \times (-1,1)^k$,  has the required properties (see (\ref{eq_hatpsi_mu}) for $\hat\psi_\mu$).

Note that it follows from step \ref{ste_df} (and Lemma \ref{lem_traTra}, which entails  that $\tra\, u=0$ on $S$ when $S\subset A$) that $ \mathbf{B}_\mu^{0,0}(u)\in  \E^{0,0} _p$.
 In order to check condition (\ref{item_conv}) of {${(\mathbf H_{0}})$}, let us first prove that
\begin{equation}\label{eq_lim_hatpsi_prB}
	\lim_{\mu \to 0} ||\hat\psi_\mu \cdot \rfh  \,\dfh u-\rfh \dfh u||_{\wca^0_p(U^\ep)}=0.
\end{equation}
 Indeed, since $\hat\psi_\mu \cdot \rfh  \dfh u$ tends to $\rfh \dfh u$ in the $L^p$ norm ($ \rfh  \dfh u$ is $L^p$, by  step \ref{ste_df}, and $\hat \psi_\mu$ tends to $1$ pointwise), it suffices to show the analogous fact for their respective exterior differentials, i.e. that  $d(\hat\psi_\mu\cdot  \rfh  d u)$ tends to $d\rfh \dfh u$.  Notice that, as  $\dsh \hat\psi_\mu=0$, we have
 $$d( \hat\psi_\mu \cdot \rfh  \dfh u) \overset{(\ref{eq_d_dlf})}=\dfh\hat\psi_\mu \cdot \rfh  \dfh u + \hat\psi_\mu \cdot d \rfh  \dfh u.$$
 It follows from (\ref{eq_psieta_ineq}) and the first inequality of (\ref{eq_rfh_ueta}) that the first term of the right-hand-side tends to zero in the $L^p$ norm as $\mu\to 0$. It follows from step \ref{ste_df} (which entails that  $d \rfh  \dfh u$ is $L^p$) and Lebesgue's theorem (since $\hat \psi_\mu$ tends to $1$ pointwise) that the second term converges to the desired limit, yielding (\ref{eq_lim_hatpsi_prB}).

We claim that $u$ extends continuously on $S$. Indeed, this is true if $S\subset \delta M\setminus A$, since $M$ is connected along this set by assumption of the theorem (see the paragraph right before (\ref{eq_trace_fn_plarge}); it is worthy of notice that this is the only place in the proof where we make use of the assumption ``$M$ connected along $\delta M\setminus A$'', which is yet essential), and this is also true if $S\subset A$, since then $\tra u=0$ on $S$, which means that $u$ extends by zero. This yields our claim which entails that the function $F^\ep \ni x\mapsto u(x,y)$ satisfies  (\ref{eq_tra_et_r}) for almost each $y$.

Equality  (\ref{eq_lim_hatpsi_prB}) and the definition of $\mathbf{B}_\mu^{0,0}$ imply that  $\mathbf{B}_\mu^{0,0}(u)(x,y)$ tends to
$$\rfh \dfh u(x,y)+\tra\, u(y)\overset{ (\ref{eq_tra_et_r})}=u(x,y)$$ in the $\wca^0_p$ norm, yielding (\ref{item_conv}) of {${(\mathbf H_{0}})$} for $\mathbf{B}_\mu^{0,0}$.

The first sentence of (\ref{item_kappa}) follows from the definition of $\mathbf{B}^{0,0}_\mu$ and step \ref{ste_df}. Moreover, $(iii)$ of Theorem \ref{thm_trace} and Lemma \ref{lem_traTra} entail $\tra u=0$ on $S$ when $S\subset \adh A$, yielding the second sentence.
As $\rfh$ and $\dfh$ commute with  $\frac{\pa}{\pa y_i}$, (\ref{item_pay}) is also fulfilled, and step \ref{ste_j=0} is  complete.
\end{proof}


Let us emphasize that, due to (\ref{eq_r_homot_operator}),  we have for every $\omega\in \E_p ^{0,j}$, $j\ge 1$:
\begin{equation}\label{eq_r_hom}
 \rfh \dfh\hskip-1mm \omega +\dfh\rfh \omega=\omega.
\end{equation}


\begin{ste}\label{ste_type_0} We prove
 $({\mathbf H_{0}})$.
 \end{ste}
 \begin{proof}
We proceed by induction on $j$.  The case $j=0$ being established by step \ref{ste_j=0}, we fix $j\ge 1$ and assume that $\mathbf{B}^{0,j-1}_\mu$ has been constructed.
We are going to prove that for $\mu>0$
$$\mathbf{B}^{0,j}_\mu(\omega):=\hat\psi_\mu\cdot  \rfh \dfh\hskip-1mm \omega +\dfh\mathbf{B}^{0,j-1}_\mu(\rfh\omega), \qquad \omega\in \E^{0,j}_p,$$
  has the required properties (see (\ref{eq_hatpsi_mu}) for $\hat\psi_\mu$). Note that step \ref{ste_df} ensures that $\rfh\omega$ belongs to $\E^{0,j-1}_p$, and consequently, that $\mathbf{B}^{0,j}_\mu(\omega)$ is well-defined. 
  Step \ref{ste_df} also yields $\mathbf{B}^{0,j}_\mu(\omega)\in \E^{0,j}_p$.
  
We turn to check conditions (\ref{item_conv}), (\ref{item_kappa}), and (\ref{item_pay}) of $(\mathbf{H}_{0})$. Since $\hat\psi_\mu$ vanishes in the vicinity of $S$, condition (\ref{item_kappa}) of $(\mathbf{H}_{0})$, which holds true for $\mathbf{B}^{0,j-1}_\mu$, also holds for $\mathbf{B}^{0,j}_\mu $. Moreover,
as $\rfh, \mathbf{B}^{0,j-1}_\mu$, as well as $\dfh$, commute with $\frac{\pa }{\pa y_i}$ for all $i$, and since $\hat\psi_\mu(x,y)$ is constant with respect to $y$, it is clear that condition (\ref{item_pay}) of $(\mathbf{H}_{0})$ is fulfilled as well.

To check  (\ref{item_conv}), observe first that, as $\hat\psi_\mu\cdot  \rfh \dfh\hskip-1mm \omega$ converges to $\rfh \dfh\hskip-1mm \omega$ in the $L^p$ norm as $\mu$ goes to zero (since $ \rfh \dfh\hskip-1mm \omega$ is $L^p$, by (\ref{eq_rfh_ueta})) and $\dfh\mathbf{B}^{0,j-1}_\mu(\rfh\omega)$ tends to $\dfh\rfh\omega$ (by (\ref{item_conv}) for $\mathbf{B}^{0,j-1}_\mu$),  (\ref{eq_r_hom}) shows that $\mathbf{B}^{0,j}_\mu(\omega)$  tends  to $\omega$ in this norm. We thus just have to establish the $L^p$ convergence of $d\mathbf{B}^{0,j}_\mu(\omega)$ to $d\omega$. As $\frac{\pa }{\pa y_i}$ commutes with $ \mathbf{B}^{0,j-1}_\mu$, $\dfh $, and $\rfh $ for all $i$ we have:
$$d\mathbf{B}^{0,j}_\mu(\omega)\overset{(\ref{eq_d_dlf})}=d\hat\psi_\mu\wedge  \rfh \dfh\hskip-1mm \omega +\hat\psi_\mu\cdot   \dfh\hskip-1mm \omega +\sum_{i=1}^kdy_i\wedge\left(\hat\psi_\mu\cdot  \rfh \dfh \frac{\pa \omega}{\pa y_i}+\sum_{i=1}^k  \dfh\hskip-0.8mm \mathbf{B}^{0,j-1}_\mu(\rfh\frac{\pa \omega}{\pa y_i})\right).$$
For simplicity, we will denote the first term of the above sum by $a_\mu$, the second by $b_\mu$, and the remaining part by $c_\mu$.  Observe now that, by (\ref{eq_rfh_ueta}) for $\dfh\hskip -0.5mm  \omega$ and (\ref{eq_psieta_ineq}), $||a_\mu||_{L^p(U^\ep)}$  and $||b_\mu-\dfh\hskip -0.5mm \omega||_{L^p(U^\ep)}$ both go to zero as $\mu$ goes to zero. Note also that it  follows from (\ref{eq_r_hom})  and our induction on $j$ that $c_\mu$ tends to $\sum_{i=1}^k dy_i \wedge \frac{\pa \omega}{\pa y_i} $, from which we conclude that $d\mathbf{B}^{0,j}_\mu(\omega)$ tends to
\begin{equation*}
 \dfh\hskip -0.5mm \omega+\sum_{i=1}^k dy_i \wedge \frac{\pa \omega}{\pa y_i} =d\omega,
\end{equation*}
which  completes step \ref{ste_type_0}.
\end{proof}


 \medskip
 
 \begin{ste}\label{ste_induc_type} We prove $(\mathbf{H}_{l})$ for all $l\le k $.

\end{ste}
 \begin{proof}
We argue by induction on $l$, the case $l=0$ being provided by the preceding step.  Fix $j\le m$, and let us define $\blj$, for $l>0$.

Take for this purpose any $\alpha\in \E^{l,j'}_p$, with $j'\in \N$ (the integer $j'$ is going to stand below for $j$ or $(j+1)$).   We denote by $\mathbf{A}_l \alpha$ the partial integral of  $\alpha$ in the direction $e_l$, namely, we define a $(j'-1)$-form by setting for almost every $(x,y)\in F^\ep\times (-1,1)^k$:
 $$ \mathbf{A}_l \alpha(x, y):=\int_{-1} ^{y_l}\alpha_{ e_l} (x,y_1,\dots,y_{l-1},t,y_{l+1},\dots, y_k) \,dt.$$
We first make sure that  $ \mathbf{A}_l\alpha\in \E^{l-1,j'-1}_p$. It follows from Minkowski's integral inequality that $ \mathbf{A}_l\alpha$ is $L^p$. As $ d\mathbf{A}_l\alpha\overset{(\ref{eq_chain_homotopy})}=\alpha-\mathbf{A}_ld\alpha$ (by (\ref{item_vi}) of the definition of $\E^{l,j'}_p$, $\alpha(x,y)=0$ if $y_l$ is close to $-1$), we deduce that $\mathbf{A}_l\alpha\in \wca^{j'-1}_p(U^\ep)$ .  In order to show that $\trd^p\mathbf{A}_l\alpha$ vanishes on $A^\ep$, take a smooth form  $\beta \in \wca^{m-j'}_{p',U^\ep \cup A^\ep}(U^\ep)$. Note that if $f:[-1,1]^2\to \R$ is any $L^1$ measurable function then
$$\int_{-1}^1\int_{-1}^t f(s,t)\,ds\,dt=\iint_{-1\le s\le t\le 1}f(s,t)\,ds\,dt=\int_{-1}^1\int^{1}_s f(s,t)\,dt\,ds. $$
Consequently, for every $\gamma\in L^{m-j'+1}_{p'}(U^\ep)$,
\begin{equation}\label{eq_Astar1}<\mathbf{A}_l \alpha, \gamma>=(-1)^{j'}< \alpha, \mathbf{A}_l^*\gamma>, \end{equation}
where, for $y\in (-1,1)^k$
$$\mathbf{A}_l^*\gamma(x,y_1,\dots,y_k)=\int_{1} ^{y_l}\gamma_{ e_l} (x,y_1,\dots,y_{l-1},t,y_{l+1},\dots, y_k) \,dt.$$
We claim that 
\begin{equation}\label{eq_Astar2}< \alpha, d\mathbf{A}_l^*\beta>=(-1)^{j'+1}< d\alpha, \mathbf{A}_l^*\beta>. \end{equation}
To see this, take a $\cc^\infty$ function $\phi:\R^n \to \R$ which is $1$ on a neighborhood of  $\supp_{\adh{U^\ep}}\alpha$ and $0$ near $\{(x,y)\in F^\ep \times [-1,1]^k: \exists i\le k ,\;\;  y_i=-1\}$ (we recall that (\ref{item_vi}) of the definition of  $\E^{l,j'}_p$ requires that $\alpha\equiv 0$ near this set), and notice that    $\phi \mathbf{A}_l^*\beta$ has compact support in $U^\ep \cup A^\ep$. As $\trd^p \alpha$ vanishes on $A^\ep$, we thus know that (\ref{eq_lsp_j}) holds for the pair $\alpha, \phi \mathbf{A}_l^*\beta$. But since $\phi$ is $1$ on a neighborhood of    $\supp_{\adh{U^\ep}}\alpha$, this indeed yields  (\ref{eq_lsp_j})  for the pair $\alpha,  \mathbf{A}_l^*\beta$, which is precisely (\ref{eq_Astar2}).

To prove that $\trd^p\mathbf{A}_l\alpha$ vanishes on $A^\ep$,   it now suffices to write
$$(-1)^{j'}<\mathbf{A}_l \alpha, d\beta>\overset{(\ref{eq_Astar1})}=<\alpha,\mathbf{A}_l^* d\beta> \overset{(\ref{eq_chain_homotopy})}=<\alpha, \beta>-< \alpha, d\mathbf{A}_l^*\beta>$$
$$\overset{(\ref{eq_Astar2})}=<\alpha, \beta>+(-1)^{j'}< d\alpha, \mathbf{A}_l^*\beta>\overset{(\ref{eq_Astar1})}=<\alpha-\mathbf{A}_l d\alpha, \beta> \overset{(\ref{eq_chain_homotopy})}=<d\mathbf{A}_l \alpha, \beta>,$$
which establishes that  $\trd^p \mathbf{A}_l\alpha\equiv 0$ on $A^\ep$. We then deduce that so does  $\trd^p \frac{\pa^{|d|} \mathbf{A}_l\alpha}{\pa y^d}=\trd^p \mathbf{A}_l\frac{\pa^{|d|} \alpha}{\pa y^d}$  for all $d$ (since $\frac{\pa^{|d|} \alpha}{\pa y^d}$ also belongs to $\E^{l,j'}_p$),   which means that
 $ \mathbf{A}_l\alpha\in \E^{l-1,j'-1}_p$.
 
Let then for simplicity
\begin{equation}\label{eq_dlf}
	d_l\alpha:=\dfh\hskip -0.5mm \alpha+ dy_l\wedge \frac{\pa \alpha}{\pa y_l}\overset{(\ref{eq_d_dlf})}{=}d\alpha-\sum_{i=1,i\ne l}^k dy_i\wedge \frac{\pa \alpha}{\pa y_i} ,
\end{equation}
and notice that if $\alpha\in  \E^{l,j'}_p$ then,
as $ \alpha$ vanishes near $y_l=-1$ (in virtue of (\ref{item_vi}) of the definition of $\E^{l,j'}_p$ ), we have
\begin{equation}\label{eq_A_l_hom_op}
	\mathbf{A}_l d_l \alpha+d_l\mathbf{A}_l\alpha=\alpha.
\end{equation}

Define now the desired operator $\blj$ by setting for
 $\omega\in \E^{l,j}_p$
$$\blj(\omega):= \mathbf{B}_\mu ^{l-1,j}\left( \mathbf{A}_ld_l\omega\right) + d_l\mathbf{B}_\mu ^{l-1,j-1}\left(\mathbf{A}_l\omega\right) ,$$
and let us check that it satisfies  conditions (\ref{item_conv}), (\ref{item_kappa}), and (\ref{item_pay}) of $(\mathbf{H}_{l})$.

 Condition (\ref{item_kappa}) obviously follows from $(\mathbf{H}_{l-1})$ (since $\pi^*$ commutes with $d_l$), and,
since $ \mathbf{A}_l$ commutes with $\frac{\pa }{\pa y_i}$ for all $i$,   so does (\ref{item_pay}). To check (\ref{item_conv}),
observe first that thanks to (\ref{eq_A_l_hom_op}) together with $(\mathbf{H}_{l-1})$, we already know that $\blj(\omega)$ tends to $\omega$ in the $L^p$ norm. We thus just have to check the $L^p$ convergence of $d\blj(\omega)$ to $d\omega$. As $d_l \circ d_l =0$, by (\ref{eq_dlf}),
\begin{equation}\label{eq_dblj}
 d\blj(\omega)=d \mathbf{B}_\mu ^{l-1,j}\left( \mathbf{A}_ld_l\omega\right) + \sum_{i=1,i\ne l}^k dy_i\wedge d_l \mathbf{B}_\mu ^{l-1,j-1}\left(\mathbf{A}_l\frac{\pa \omega}{\pa y_i}\right).
\end{equation}
It directly follows from $(\mathbf{H}_{l-1})$ that the first term of the right-hand-side of this equality tends in the $L^p$ norm to $d\mathbf{A}_ld_l \omega$, and we have (since (\ref{eq_A_l_hom_op}) for $d_l \omega$ gives $d_l\mathbf{A}d_l\omega=d_l\omega$ and $\mathbf{A}_l$ commutes with $\frac{\pa }{\pa y_i}$ for all $i$):
 \begin{equation}\label{eq_conv1}
d\mathbf{A}_ld_l \omega\overset{(\ref{eq_dlf})}=d_l \omega+  \sum_{i=1,i\ne l}^kdy_i \wedge \mathbf{A}_l d_l\, \frac{\pa \omega}{\pa y_i}.
 \end{equation}
 Moreover, by $(\mathbf{H}_{l-1})$ again, the second term of the right-hand-side of (\ref{eq_dblj}) tends to 
  \begin{equation}\label{eq_conv2}
   \sum_{i=1,i\ne l}^k dy_i\wedge d_l\mathbf{A}_l \,\frac{\pa \omega}{\pa y_i}.
 \end{equation}
 Hence, passing to the limit in (\ref{eq_dblj}) as $\mu\to 0$ and plugging (\ref{eq_conv1}) and (\ref{eq_conv2}) into the resulting equality, we get, thanks to (\ref{eq_A_l_hom_op}), that  $d\blj(\omega)$ tends to $$ d_l\omega +\sum_{i=1,i\ne l}^k dy_i\wedge \frac{\pa \omega}{\pa y_i}\overset{(\ref{eq_dlf})}=d\omega,$$
 which completes step \ref{ste_induc_type}.
 \end{proof}


 \medskip

 Since we have established $(\mathbf{H}_k)$, we now can perform:
\begin{ste}\label{ste_final}
 We prove the statement of the theorem.
\end{ste}
\begin{proof}As explained before step \ref{ste_j=0}, we  are going to construct approximations of a form $\omega\in \wca^j_p (M,A)$, $j>\dim E$, arguing by decreasing induction on $\kappa(\omega)$ (see (\ref{eq_kappa})).
 For $\kappa(\omega)=m$,  $\omega$ has compact support in $M$, which means that the desired fact follows from the properties of  de Rham's regularization operators \cite{derham} (the $L^p$ estimates are proved in \cite{goldshtein}).
 Let us therefore take $k< m$ and assume the result to be true for all the differential forms $\alpha$ satisfying $\kappa(\alpha)>k$.

   As we can use a partition of unity (of $\mba$), it is enough to approximate the restriction of $\omega$ to $O\cap M$, where $O$ is
    a neighborhood of a point $\xo\in \delta M$ that we can choose arbitrarily small, and we can assume $\omega$ to have compact support in $O$. 
   Let us  denote by $S$ the stratum that contains $\xo$.
  Note that if $\dim S<k$ then $x_0$ does not belong to $\supp_\mba \omega$, in which case the needed statement is a tautology, and, if $\dim S>k$, the required fact comes down from our induction on $\kappa(\omega)$ (since $\omega$ is supported in a small neighborhood of $\xo$). We thus can assume $\dim S=k$. Without loss of generality, we can identify $\xo$ with $\orn$.

We are reduced to the situation described at the very beginning of the proof.
Let $\Lambda$ be the mapping defined in (\ref{eq_Lambda}) and
 let  $F^\ep$ and  $U^\ep$ be as defined in (\ref{eq_F}) and (\ref{eq_UV}), $\ep>0$ small.    As explained in section \ref{sect_stokes_formula}, $\Lambda$ gives rise by pull-back (for each $j$ and $p$) to an isomorphism $\Lambda^*: \wca^j_p(F^\ep\times (-1,1)^k)\to \wca^j_p(O\cap M)$, where $O:=\Lambda^{-1}(U^\ep)$.

  By step \ref{ste_dense}, we can find   arbitrarily close  approximations  of $\Lambda^{-1*}\omega$ by forms
  $\omega_i\in \E_{p,V^\ep}^{k,j}$. Possibly replacing $\omega$ with the successive elements of the sequence $\Lambda^* \omega_i$ (which tends to $\omega$), we can assume $\Lambda^{-1*}\omega \in \E_{p,V^\ep}^{k,j}$.
  
 Let $\mathbf{B}_\mu ^{k,j}$ be the operator given by $(\mathbf{H}_{k})$ (which was established by step \ref{ste_induc_type}). By (\ref{item_kappa}) of $(\mathbf{H}_{k})$, we have
  \begin{equation*}\label{eq_dec_approx}
  	\mathbf{B}_\mu ^{k,j}(\Lambda^{-1*}\omega):=\alpha_\mu+\pi^* \beta_\mu, 
  \end{equation*}
 with $\alpha_\mu\in\wca^j_p(U^\ep,A^\ep)$ satisfying $\kappa(\alpha_\mu)>k$ and $\beta\in \cc^{j,\infty} (S)$. By (\ref{item_conv}) of $(\mathbf{H}_{k})$, the  form $(\Lambda^*\alpha_\mu+\Lambda^*\pi^* \beta_\mu)$ tends to $\omega$ as $\mu\to 0$.

 By induction on $\kappa$, we know that we can approximate $\Lambda^*\alpha_\mu$ by smooth forms that vanish in the vicinity of $O \cap\adh{A\cup E}$.  It therefore suffices to show that $\Lambda^*\pi^* \beta_\mu$ is a smooth form on $O$ that is identically zero near $O \cap\adh{A\cup E}$.
 
 We first check that $\Lambda^*\pi^* \beta_\mu$ is  smooth. As $\Lambda$ commutes with $\pi$ and is the identity on $S$, we have  $\Lambda^* \pi^*\beta_\mu= \pi^*\Lambda_{|S}^*\beta_\mu=\pi^*\beta_\mu,$
 which is smooth, since so are both $\beta_\mu$ and $\pi$.
 
  To check  that $\Lambda^*\pi^* \beta_\mu$ is identically zero near $O \cap\adh{A\cup E}$, we will distinguish three cases, showing that actually either $\beta_\mu\equiv 0 $ near this set (cases $1$ and $3$ below) or $O \cap\adh{A\cup E}=\emptyset$ if $\ep$ was chosen small enough (case $2$):

\noindent{\it Case 1:}  $ k\le \dim E$. Then $j>k$ (for $j>\dim E$, by assumption of the theorem), which entails $\beta_\mu \equiv 0$ (since $\beta_\mu$ is a $j$-form on $S$ and $\dim S=k<j$).

\noindent{\it Case 2:} $k>\dim E$ and $S\cap \adh {A} =\emptyset$. In this case $O \cap\adh{A}=\emptyset$ (for $\ep >0$ small), and the compatibility of the stratification with  $E$ entails that   $O \cap\adh{ E}=\emptyset$ (for $\ep >0$ small), which means that the needed statement is vacuous.

\noindent{\it Case 3:}   $S \cap \adh {A} \neq \emptyset$. Since $\Sigma$ is compatible with $\adh A$, we must have  $S\subset \adh A$ (see Remark \ref{rem_frontier}). In this case, (\ref{item_kappa}) of $(\mathbf{H}_{k})$ specifies that $\beta_\mu\equiv 0$.
%
%
\end{proof}
   \end{subsection}

\end{section}

\begin{section}{De Rham theorem for relative $L^p$ cohomology and proof of the main results}\label{sect_lp_coh}
The first step of the proof of our Lefschetz duality theorem is to compute the relative $L^p$ cohomology in the case ``$p$ large''. We shall achieve it by proving a de Rham type theorem (Theorem \ref{thm_derham}) that relates these cohomology groups to  simplicial cohomology groups that we are going to introduce and that we will call the {\it collapsing cohomology groups} (see the definition just below). We will derive that the $L^p$ cohomology groups of $(M,A)$ are finitely generated and topological invariant of $(\mba,M,A)$ (Corollary \ref{cor_finitely_generated}).

 Given $X \subset \R^n$,  a {\bf definable singular $k$-simplex of $X$}\index{singular simplex} will be a continuous  definable mapping $\sigma: \Delta_k \to X$, $\Delta_k$ being the $k$-simplex spanned by
 $0,e_1,\dots,e_k$, where $e_1,\dots, e_k$ is the canonical basis of
 $\R^k$.    We denote by $C_k(X)$
 the vector space of definable singular $k$-chains of $X$, i.e., finite linear combinations (with real coefficients) of  singular definable simplices of $X$, and by
 $\partial c$  the boundary of a chain $c$.

\begin{subsection}{Collapsing cycles}  Given a definable subset $ Z$ of $\delta M$, we say that $c\in C_j(\mba)$ is  {\bf collapsing on} $Z$ if
 $$\dim |c|\cap Z \le j-2 .$$
 We denote by $\cch_{j} (M,A)$ the $\R$-vector space constituted by the $j$-chains $c\in C_j(\mba, A) $ collapsing on $\delta M\setminus A$ such that the  $(j-1)$-chain $\pa c$  is also collapsing on $\delta M\setminus  A$.
We  endow this chain complex with the boundary operator $\pa$ and we will write $\hch_j(M,A)$ for the resulting homology groups, that will be called the {\bf collapsing homology groups}.

By definition, these homology groups coincide with the singular homology groups of $(\mba,\delta M)$ if $A=\delta M$, and with intersection homology groups of $\mba$ \cite{ih1,ih2} when $A=\emptyset$ (if $\mba$ is a pseudomanifold). The motivation for introducing the collapsing homology  is that  
 it is indeed the geometric counterpart of relative $L^p$ cohomology for $p$ large:
\begin{thm}\label{thm_derham}(De Rham theorem for relative $L^p$ cohomology) Let $A$ be a definable subset of $\delta M$.
 For all $p\in [1,\infty]$ sufficiently large, we have  for all $j$:
 $$H_p^j (M,A)\simeq \hch^j \left(M,A\right) .$$

\end{thm}

The proof of  this theorem  given in section \ref{sect_proof} will rely on a sheaf theoretic argument, which requires some ``Poincar\'e lemmas'' for both $L^p$ cohomology and collapsing homology that are established in  section \ref{sect_local_computations} just below. Together with Proposition \ref{pro_normal_isom_cohomologie}, the above theorem yields for $A\subset \delta M$ definable:
 \begin{cor}\label{cor_finitely_generated}
 For all $p\in [1,\infty]$ sufficiently large, $H_p^j (M,A)$ is finitely generated and just depends on the topology of the triplet $(\mba, M,A)$ for all $j$.
 \end{cor}
\end{subsection}

\begin{subsection}{Some preliminary computations of the cohomology groups.}\label{sect_local_computations}
We are going to  compute the $L^p$ cohomology of the germ of $(M,A)$ at a point of $\mba$ (Lemma \ref{lem_poinc}), as well as  the collapsing homology of such a germ (Proposition \ref{pro_loc_hch}). We then derive a global computation of the collapsing cohomology in terms of the cohomology groups of a normalization (Proposition \ref{pro_normal_isom_cohomologie}).

Given $\xo\in \R^n$ and $\ep>0$, as well as a subset $Z\subset \R^n$, we set for simplicity:
\begin{equation}\label{eq_Zxoep}Z_\xo^\ep:= \bou(\xo,\ep)\cap Z.\end{equation}
We then let $\cfr_M(\xo,A)$ denote the number of connected components $E$ of $M_\xo^\ep$, $\ep>0$ small, for which $\adh E\cap A_\xo^\ep$ is empty.

\begin{lem}\label{lem_poinc}(Poincar\'e Lemma for relative $L^p$ cohomology)  Given $\xo\in \mba$,  we have for all $\ep>0$ sufficiently small and  $p\in [1,\infty]$ sufficiently large:
\begin{equation}\label{eq_poinc_0}
             H^0_p(M_\xo ^\ep, A_\xo^\ep)\simeq \R^k,\quad \mbox{where} \quad k=\cfr_M(\xo,A),
            \end{equation}
 as well as for all  $1\le j\le m$,
   \begin{equation}\label{eq_poinc_1}
  H^j_p (M_\xo ^\ep, A_\xo^\ep)\simeq 0\simeq H^{m-j} _{p',\mba_\xo^\ep}  (M_\xo ^\ep,A_\xo^\ep).
  \end{equation}
\end{lem}
\begin{proof}
 Given a closed form $\omega\in \wca^j_p (M_\xo^\ep)$ (with $p$ large), by (\ref{eq_r_homot_operator}), we have $d\rfr \omega=\omega$,  so that, by Proposition \ref{pro_r_preserve_vanishing}, we see that the first  isomorphism of (\ref{eq_poinc_1}) is clear.   To prove the second one,  just replace $\rfr$ with $\Rfr$.
 Finally, a closed $0$-form  $u \in \wca^0_p(M_\xo^\ep,A^\ep_\xo)$ is an element of $W^{1,p}(M_\xo^\ep)$ which is constant on every connected component of $M_\xo^\ep$, and, 
 by Lemma \ref{lem_traTra}, has to satisfy $\tra u \equiv 0$ on $A^\ep_\xo$. Such a function  must be identically zero on each connected component   of $M_\xo^\ep$ whose closure meets $ A_\xo^\ep$, which yields (\ref{eq_poinc_0}).
\end{proof}

 For $U$ open subset of $\mba$, let
$$\cch_{j}^{M,A}(U):=\{c\in\cch_{j} (M,A):|c|\subset U  \}. $$ 
We stress the fact that the support of a chain $c\in \cch_j(M)$ is a subset of $\mba$. We then define a presheaf
$$\cch^{j}_{M,A}(U):=\Hom (\cch_{j}^{M,A}(U),\R),$$
 that we endow with the usual coboundary operator $\pa^*$.
We will denote by $\C^j_{M,A}$ and $\check{\C}_{M,A}^j$ the complex of sheaves respectively derived from the complexes of presheaves  $U\mapsto C^j(U,U\cap A)$ and  $U\mapsto \cch^j_{M,A}(U)$, $U$ open subset of $\mba$.

The sheaves $\C^j_{M,A}$ and $\check{\C}_{M,A}^j$ being soft, they are acyclic, and,  as well-known (see for instance \cite[Section I.7 p. 26]{bredon} or \cite[section 5.32]{warmer})   $H^j(\C^\bullet_{M,A}(\mba))\simeq H^j(\mba,A)$. Let us check that the same applies to  $\check{\C}_{M,A}^j$, i.e., that we have for all $j$:
	\begin{equation}\label{eq_cch_same_hom} H^j( \check{\C}_{M,A}^\bullet(\mba) )\simeq H^j(\check{C}^\bullet(M,A) )=\check{H}^j(M,A).\end{equation}
\begin{proof}[Proof of (\ref{eq_cch_same_hom}).]  If $\mathcal{U}$ is an open covering of $\mba$, we denote by $ \cch^j(\mathcal{U},A)$ the subspace of cochains based on $\mathcal{U}$-small chains of $\cch_{j} (M,A)$ (i.e. on the chains $\Sigma \alpha_i \sigma_i$, where $\sigma_i$ is for each $i$ a simplex that fits in some element of $\mathcal{U}$).  A subdivision argument (see \cite[section 5.32]{warmer} for instance) then shows that the natural surjection $j_\mathcal{U}: \cch^{j} (M,A) \to \cch^j(\mathcal{U},A)$ induces a cohomology isomorphism, which, via a classical exact sequence \cite[section I, Theorem 6.2]{bredon}, yields (\ref{eq_cch_same_hom}).
\end{proof}

\begin{pro}\label{pro_loc_hch}
Given $\xo \in \mba$, we have for $\ep>0$ sufficiently small:
  $$H^j ( \chh _{M, A, \xo}^\bullet )\simeq\begin{cases}\label{eq_loc_hch_j}
   0 \qquad \mbox{if } \quad j\ge 1,\\
  \R^k \quad \mbox{ if} \quad j=0,\quad \mbox{with} \quad k=\cfr_M(\xo,A),
  \end{cases}$$
  where  $\chh _{M, A, \xo}^\bullet$ stands for the stalk of the sheaf $ \chh^\bullet _{M, A}$ at $\xo$.
 \end{pro}
\begin{proof}We start with the case $j\ge 1$. 
   For $\ep>0$ small, there is a definable continuous retraction by deformation $r_s:\mba_\xo^\ep \to  \mba_\xo^\ep$, $s\in [0,1]$, $r_0 \equiv \xo$, $r_1=Id$, that may be required to  preserve $M$ and $A$ for $s\in (0,1]$  (applying Theorem \ref{thm_local_conic_structure} for instance, we will however not need the Lipschitzness properties of $r_s$, which are indeed the main features of this theorem). This retraction by deformation provides, given a cycle  $c\in C_j(\mba_\xo^\ep)$, $j\ge 1$, a chain $a$ satisfying $\pa a=c$. By definition of collapsing chains, if $c$ is collapsing on $\delta M\setminus  A$ then so is $a$, which means that for every $j\ge 1$ we have for $\ep>0$ small:
\begin{equation}\label{eq_ccch_loc}H^j( \cch^\bullet_{M,A} (\mba_\xo^\ep))\simeq 0.\end{equation}

To address the case where $j=0$, take a connected component $E$ of $M_\xo^\ep$, $\ep>0$ small.
 In the case $\adh E\cap  A\ne \emptyset$, every point of $E$ may be joint to some point of $A_\xo^\ep$ by  a definable continuous path $\gamma:[0,1]\to \mba_\xo^\ep$ satisfying $\gamma([0,1))\subset E$, which means that such a connected component does not contribute to $H^0 ( \chh _{M, A, \xo}^\bullet )$.  Assume thus $\adh E\cap  A= \emptyset$. Then every collapsing $0$-chain  $c$ in $\adh E$ satisfies $|c|\cap \delta E=\emptyset$. Consequently, if $c\in \cch^{M,A}_0 (\mba_\xo^\ep)$ satisfies $|c|\subset \adh E$ then $|c|\subset E$. As  any two points in $E$ can be joint by a continuous path in $E$, we therefore see that every such  $E$  provides an independent generator of the homology groups. Since $ \cfr_M(\xo,A)$ is by definition the number of such $E$, the result follows.
 \end{proof}

Given an open subset $U$ of $\mba$, we denote by $\mathcal{X}_{M,A}(U)$ the $\R$-vector space constituted by all the functions $u:U\cap M \to \R$ that are locally constant (locally in $U\cap M$, not in $U$) and that vanish in the vicinity of  $U\cap A$, and by $\lcs_{M,A}$ the sheaf resulting from this presheaf. Clearly, each element of the stalk $\lcs_{M,A,\xo}$ defines an element of $\chh_{M,A,\xo}^0$ which is a cocycle and conversely, every cocycle of $\chh_{M,A,\xo}^0$ defines a function-germ which is constant on every connected component of the germ of $M$ at $\xo$. As a matter of fact, the above proposition amounts to say that the sequence
 \begin{equation}\label{eq_chh_res}
  \lcs_{M,A}\to\chh _{M, A}^0  \to \chh _{M, A}^1  \to \chh _{M, A}^2\to \dots
 \end{equation}
is exact, and therefore that $\left(\chh _{M, A}^j\right)_{j\in \N} $ constitutes an acyclic resolution of the sheaf  $\lcs_{M,A}$. 

We will derive from Lemma \ref{lem_poinc}  that $\wca^j_p(M,A)$ also induces a resolution of  $\lcs_{M,A}$ (for $p$ large, see (\ref{eq_exact_sequence_plarge})). We first give extra properties of the collapsing homology that are analogous to those well-known for intersection homology \cite[section $4.3$]{ih1}.
  
 \begin{pro}\label{pro_normal_isom_cohomologie} Let $A$ be a subanalytic subset of $\delta M$.
 \begin{enumerate}
             \item \label{item_inclusion_collaps}  If $M$ is connected along $\delta M\setminus A$ then for all $j$ the natural mapping induced by restriction of cochains yields an isomorphism
             \begin{equation}\label{eq_incl_collaps}
               \hch^j(M,A)\simeq  H^j(\mba,A).
             \end{equation}
           
             \item \label{item_mc_collapse} If $h:\mc\to M$ is a normalization of $M$, then we have for all $j$ $$\hch^j(M,A) \simeq\hch^j(\mc,\adh h^{-1}(A)) \overset{(\ref{eq_incl_collaps})}\simeq H^j\left(\adh\mc,\adh h^{-1}(A)\right),$$
 \end{enumerate}
where $\adh h$ is the extension of $h$ (see Proposition \ref{pro_normal_existence} (\ref{item_unique})).
 In particular, the collapsing homology groups of $(M,A)$ are finitely generated.
 \end{pro}
\begin{proof}We first justify (\ref{eq_incl_collaps}).
Thanks to (\ref{eq_cch_same_hom}), (\ref{eq_chh_res}), and  standard arguments of sheaf cohomology \cite[section II.4]{bredon}, it suffices to check that when $M$ is connected along $\delta M\setminus A$, $\left(\C^j_{M,A}\right)_{j\in \N}$ constitutes a resolution of the sheaf $\lcs_{M,A}$, which is clear since  definable sets are locally contractible.

To prove (\ref{item_mc_collapse}), let $h:\mc \to M$ be a normalization of $M$ (and $\adh h:\adh \mc\to \mba$ its continuous extension) and 
let  $\F^j:= \adh h_*\chh_{\mc,\adh h^{-1}(A)}^j$ be the direct image of the sheaf $\chh_{\mc,\adh h^{-1}(A)}^j$, i.e.  let $\F^j(U):= \chh_{\mc,\adh h^{-1}(A)}^j(\adh h^{-1}(U))$, for any open subset $U$ of $\mba$.
  We  again just  have to check that the $\F^j$'s constitute a resolution of the sheaf $\lcs_{M,A}$. It actually follows from the definition of the direct image sheaf that  for $\xo\in \mba$
  $$ H^j(\F^\bullet_{\xo})\simeq \oplus_{i=1}^k H^j(\chh_{\mc,\adh h^{-1}(A),x_i}^\bullet),$$
 where $x_1,\dots,x_k$ are the elements of $\adh h^{-1}(\xo)$ and the number $k$ must coincide with the number of connected components of $\bou(\xo,\ep)\cap M$ for $\ep>0$ small  (see \cite[Proposition $3.2$]{trace}). The desired fact thus follows after applying Proposition \ref{pro_loc_hch} to $\mc$ at every $x_i$, $i\ge 1$.
\end{proof}

\end{subsection}


\begin{subsection}{Proof of Theorems \ref{thm_derham} and \ref{thm_lefschetz_duality}}\label{sect_proof}
Lemma \ref{lem_poinc} holds {\it for $p\in [1,\infty]$  sufficiently large}  in the sense that there is $p_0\in [1,\infty)$ such that its statement holds for all $ p \in( p_0,\infty]$.  If we define $p_\xo$ as the smallest such real number $p_0$, this number of course depends on the geometry of $\mba$ near $x_0$ and may  vary on this set. We however have:
 \begin{equation}\label{eq_pm_borne}
   \sup_{x_0 \in \mba} p_\xo<\infty . 
 \end{equation} 
  \begin{proof}[Proof of (\ref{eq_pm_borne}).] By Theorem \ref{thm_existence_stratifications}, we know that there is a stratification of $\mba$ compatible with $M$, such that given any two points $x_0$ and $x_1$ in the same stratum, there is a (germ of) subanalytic bi-Lipschitz homeomorphism that maps the germ of $M$ at $x_0$ onto the germ of $M$ at $x_1$, which means (see (\ref{eq_pullback})) that the respective $L^p$ cohomology groups of the stalks must be isomorphic (for all $p$). Hence,  $p_\xo$ is constant on the strata of such a stratification and therefore can only take  finitely many values.
  \end{proof}

\begin{proof}[Proof of Theorem \ref{thm_derham}] Let $ \mathcal{G}^j_{p}$ be the sheaf on $\mba$ derived from the presheaf $U\mapsto  \wca^j_p(U\cap M , U\cap A)$.
 Lemma \ref{lem_poinc} and equality (\ref{eq_pm_borne}) establish that for $p$ sufficiently large the sequence
\begin{equation}\label{eq_exact_sequence_plarge}
\R_{M,A} \to  \mathcal{G}_{p}^0 \to  \mathcal{G}_{p}^1\to  \mathcal{G}_{p}^2\to \dots,  
\end{equation}
is exact, and therefore that $\left( \mathcal{G}_{p}^j \right)_{j\in \N}$ is a fine resolution of the sheaf $\R_{M,A}$.  Theorem \ref{thm_derham} then comes down from the exact sequence (\ref{eq_chh_res}) (see \cite[section II.4]{bredon}).
  \end{proof}

\begin{proof}[Proof of Theorem \ref{thm_lefschetz_duality}.]
Note that showing that $(\poi_{M,A}^j)_{j\in \N}$ is an isomorphism {\it for $p$ large}  establishes that the pairing $<\alpha,\beta>$, $\alpha \in H^j_{p} (M,A)$, $\beta\in H^{m-j}_{p'}(M,\delta M\setminus A)$, is nondegenerate for every such $p$, and therefore, in virtue of Corollary \ref{cor_finitely_generated}, that $H^{m-j}_{p'}(M,\delta M\setminus A)$ is finitely generated. It is well-known that in this situation, the homomorphisms derived from the pairing are then both isomorphisms, which means that proving that $(\poi_{M,A}^j)_{j\in \N}$ is an isomorphism for $p$ large  shows that $(\poi_{M,\delta M\setminus A}^j)_{j\in \N}$ is an isomorphism for $p$ close to $1$. We will therefore only focus on proving that $\poi_{M,A}^j $ is an isomorphism for $p$ large.

We perform the proof by induction on $m\ge 0$, the case $m=0$ being trivial. Taking a normalization if necessary, we can assume  $M$ to be normal. 
Set for $U$ open subset of $\mba$
$$\F^j_{p}(U):= \wca^j_{p,U}(U\cap M,U\cap \delta M \setminus  A), $$
 as well as
$$ \F_{p}^{*j}(U):=\Hom (\F^{m-j}_{p}(U),\R).$$
The spaces $\F_{p}^{*j}$, $j\in \N$, constitute a complex of flabby sheaves on $\mba$ (the $ \F^{j}_{p}$'s are sometimes called {\it cosheaves} in the literature, see for instance \cite[Propositions V.1.6 and V.1.10]{bredon}).   As $\poi^j_{M,A}$ is a chain map, by (\ref{eq_exact_sequence_plarge}), it suffices to show that the sequence
$$ \R_{M,A} \overset{\poi^0_{M,A}}\longrightarrow \F_{p'}^{*0}  \overset{d^*}\longrightarrow \F_{p'}^{*1}\overset{d^*}\longrightarrow\ \F_{p'}^{*2}\overset{d^*}\longrightarrow\dots,$$
is exact for $p$ large. 

   Lemma \ref{lem_poinc} yields $d\F_{p'}^{m-j-1}=\ker d_{| \F_{p'}^{m-j}}$ for all $j\ge 1$ (for $p$ large, see (\ref{eq_pm_borne})) and consequently $d^*\F_{p'}^{*j-1}=\ker d^*_{| \F_{p'}^{*j}}$. We thus simply need to deal with the image of the first arrow.
Since $M$ is normal, the elements of $\R_{M,A}$ are constant near each $\xo\in \mba$, which means that
 it suffices to prove that for $p$ large
  the mapping
 $$\Phi: H^m_{p',\mba^\ep_\xo} (M^\ep_\xo,Z_\xo^\ep)\to \R, \quad \omega\mapsto \int_{M^\ep_\xo} \omega,$$ is an isomorphism 
 for $\ep>0$ small, if $Z:= \delta M \setminus  A$ and $Z^\ep_\xo,M_\xo^\ep$, as well as $\mba^\ep_\xo$ are like explained in (\ref{eq_Zxoep}).

To simplify notations, we will assume $\xo=\orn$ and will no longer match $\xo$, so that we will be in the setting of section \ref{sect_the_retraction_r_s_and_R_t} (see (\ref{eq_metaneta})). The mapping $\Phi$ is obviously surjective.
To show injectivity, take $\omega\in \wca^m_{p',\mba^\ep}(M^\ep,Z^\ep)$ satisfying $\int_{M^\ep} \omega=0$.
  By (\ref{eq_r_homot_operator}),   $\rfr \omega$ is an antiderivative of $\omega$ which, by  Proposition  \ref{pro_r_preserve_vanishing}, must belong to $ \wca^{m-1}_{p'}(M^\ep,Z^\ep)$.

  It is however not enough to show that the cohomology class of $\omega$ is zero as $\rfr\omega$ may not have compact support in $\mba^\ep$.   We thus have to add to $\rfr\omega$ some closed form which is supported near $N^\ep:= \sph(0,\ep)\cap M$ in order to produce  an antiderivative of $\omega$ which vanishes near $\adh{\nep}$ (see (\ref{eq_antider_omega}) below).
 By Proposition \ref{pro_Rfr_2eme_forme} $(ii)$, we have for $(t,x)\in (0,1)\times N^\ep$:
\begin{equation}\label{eq_hstrfr}h^*\rfr\omega(t,x)= \int_0 ^t (h^*\omega)_{\pa_s}(s,x)ds, \end{equation}
where $h$ is as in (\ref{eq h pour K_nu}).
As $\omega$ is zero in the vicinity of $\adh {N^\ep}$,  $h^*\rfr\omega(t,x)$ is thus constant with respect to $t$ when $t$ is close to $1$.  Moreover, we have for such $t$:
$$\int_{\{t\}\times\nep}h^*\rfr\omega\overset{(\ref{eq_hstrfr})}=\int_{N^\ep}\int_0 ^t (h^*\omega)_{\pa_s}(s,x)ds=\int_{M^{t\ep}}\omega=\int_{M^\ep}\omega=0. $$
By induction on $m$, this entails that (for $p$ large) there is $\theta\in \wca^{m-2}_{p'}( N^\ep,Z\cap \sph(0,\ep))$ such that $d \theta(x)=h^*\rfr\omega(t,x)$ for $x\in N^\ep$ and $t\in (0,1)$ close to $1$, regarding  $\theta$ as a form on $(0,1)\times N^\ep$, constant with respect to the first variable.  As $\trd^{p'} \theta$ vanishes on $Z\cap \sph(0,\ep)$ and is constant with respect to  the first variable, it vanishes on $(0,1)\times \left(Z\cap \sph(0,\ep)\right)$.

Note that as $h$ is bi-Lipschitz on the complement of every neighborhood of the origin, the restriction of  $h^{-1*}\theta$ to   $M^\ep\setminus M^\epd$ is $L^{p'}$.  Hence, if $\phi$ is a smooth function on $\mep$ which equals $1$ near $N^\ep$ and  zero on $M^\epd$, the form
\begin{equation}\label{eq_antider_omega}\beta:=\rfr\omega-d (\phi h^{-1*} \theta)\end{equation}
 is $L^{p'}$, has compact support in $\mba^\ep$, vanishes on $Z^\ep$, and satisfies $d \beta=\omega$. This yields the injectivity of $\Phi$.
\end{proof}

\end{subsection}

\begin{subsection}{Hodge theory on subanalytic manifolds}\label{sect_hodge}
For $\alpha$ and $\beta$ in $L^{j}_p(M)$, $0\le j\le m$, $p\in (1,\infty)$, we set (see (\ref{eq_*_p}) for $*_p$)
\begin{equation}\label{eq_bracket}[\alpha,\beta]_p:=<\alpha,*_p \beta>.\end{equation}
In particular, $[\cdot,\cdot]_2$ is the usual $L^2$ inner product.
It is easy to check that  $(L^j_p(M),[\cdot,\cdot]_p)$ is a {\bf homogeneous semi-inner product space} \cite{giles,lumer}  ((\ref{eq_cauchyschwartz}) below follows from H\"older's inequality),
i.e.
for all $\alpha$ and $\beta$  in $L^j_p(M)$:
\begin{enumerate}\label{se}	
		\item  $[\alpha , \beta]_p$ is linear with respect to $\alpha$.
	\item $[\alpha,\alpha]_p^{1/2} =||\alpha ||_{L^p(M)}$.
		\item\label{eq_cauchyschwartz}$[\alpha,\beta]_p \le ||\alpha||_{L^p(M)}||\beta||_{L^p(M)} .$
	 \item(homogeneity)  $[\alpha , \lambda\beta]_p=\lambda[\alpha , \beta]_p$ .
		\end{enumerate}
It follows from \cite[section $3$]{giles} that	this  homogeneous semi-inner product space is {\bf continuous}, in the sense that for any couple $\alpha$ and $\beta$ 
\begin{enumerate}[(5)]	\item  $[\alpha , \beta+\lambda \alpha]_p$ tends to $[\alpha , \beta]_p$  as $\lambda\to 0$.		\end{enumerate}
Continuous homogeneous semi-inner product spaces share many features with inner product spaces, especially when they are uniformly convex Banach spaces \cite{giles}. It follows from Clarkson's inequalities \cite{clarkson} that $L^j_p(M)$ is uniformly convex for all $p\in(1,\infty)$ and $j\le m$, which is enough to ensure that, if $V$ is a {\it closed linear subspace} of $L^j_p(M)$  then there is a unique continuous mapping $P_V:L^j_p(M)\to V$  such that for all $\omega\in L^j_p(M)$, $||\omega-P_V(\omega)||_{L^p(M)}$ realizes the distance to $V$ \cite{fort}, i.e.
\begin{equation}\label{eq_proj_inf}||\omega-P_V(\omega)||_{L^p(M)}=\inf_{\alpha\in V} ||\omega-\alpha||_{L^p(M)}.\end{equation}
 The vector $P_V(\omega)$ is also the unique element of $ V$ such that for all $\alpha\in V$ \cite[Theorem 2]{giles}
\begin{equation}\label{eq_proj_sca}[\alpha,\omega - P_V(\omega)]_p=0 .\end{equation}
One major difference with inner product spaces is however that, since $[\cdot,\cdot]_p$ is not bi-linear,   $P_V:L^j_p(M)\to V$ may fail to be linear and
$$V^{\perp_p}:= \{\beta\in L^j_p(M):\forall \alpha \in V,  [\alpha ,\beta]_p=0\} $$
 is not necessarily a vector subspace of $L_p^j(M)$.
 In virtue of (\ref{eq_proj_sca}), we however have $V^{\perp_p}=P_V^{-1}(0)$. A consequence of these facts that will be useful for our purpose is that when a subspace $V$ of $ L^j_p(M)$ is closed,
 we have
\begin{equation}\label{eq_proj_sum}
 L^j_p(M)=V \oplus V^{\perp_p}
\end{equation} ($V^{\perp_p}$  being possibly nonlinear, the symbol $\oplus$ just means that every element of $ L^j_p(M)$ can be uniquely decomposed $v+w$, with $v \in V$ and $w \in V^{\perp_p}$).

\begin{proof}[Proof of Theorem \ref{thm_hodge} and Corollary \ref{cor_hodge}.]
 Note first that for all $j$ we have for every $\alpha\in \wca^{j-1}_p(M,A)$ and $\beta\in \wsc_j^p(M,\delta M\setminus  A)$ (see (\ref{eq_wscpj_A}) for $\wsc_j^p$) if $p\in (1,\infty)$ is sufficiently large or close to $1$
\begin{equation}\label{eq_lsp_crochet}
[d\alpha,\beta]_p=<d\alpha,*_p\beta>\overset{(\ref{eq_lsp_A})}=(-1)^j<\alpha,d*_p\beta>\overset{(\ref{eq_deltap})}= [\alpha, \delta_p\beta]_p.
\end{equation}

Fix $j\le m$, and denote for simplicity by $Z^j_p(M,A)$ the kernel of the restriction of $d$ to $\wca^j_p(M,A)$. It is clearly a closed subspace of $L^j_p(M)$. We claim that for $p$ large or close to $1$:
\begin{equation}\label{eq_zjp_perp}Z_p^{j}(M,A)^{\perp_p}= cE^j_p (M,\delta M\setminus A), \end{equation}
as subspaces of $L^j_p(M)$.
Indeed,
 if $\omega\in cE^j_p(M,\delta M\setminus A)$ then,  since $\omega=\delta_p\beta$ for some $\beta\in \wsc_{j+1}^p(M,\delta M\setminus A)$, we have 
  for all $\alpha \in Z^j_p(M,A)$ (for $p$ large or close to $1$):
  $$[\alpha,\omega]_p=[\alpha,\delta_p\beta]_p\overset{(\ref{eq_lsp_crochet})}=[d\alpha, \beta]_p=0.$$
   Conversely, if a form $\omega\in L^j_p(M)$ satisfies $[\alpha,\omega]_p=0$ for all $\alpha\in Z^j_p(M,A)$ then $*_p\omega$ is a closed form that belongs to $\wca^{m-j}_{p'}(M,\delta M\setminus A)$ (since $<d\theta,*_p\omega>=0$ for all $\theta \in \wca^{j-1}_p (M,A)\supset \wca^{j-1}_{p,\mba \setminus A} (M)$ ) and that must be in the kernel of Lefschetz-Poincar\'e duality homomorphism, which, by Theorem \ref{thm_lefschetz_duality}, is enough to ensure that $*_p\omega=d\gamma$ for some $\gamma\in \wca^{m-j-1}_{p'}(M,\delta M \setminus  A)$, if $p\in (1,\infty)$ is sufficiently large or close to $1$. By definition of $\delta_p$, this implies that $\omega=\pm\delta_p *_{p'} \gamma$, which, since $*_{p'} \gamma \in \wsc^p_{j+1}(M,\delta M\setminus  A)$, yields (\ref{eq_zjp_perp}).
   
 Thanks to (\ref{eq_proj_sum}), we immediately  derive from (\ref{eq_zjp_perp})  that:
\begin{equation}\label{eq_zce}
 \wca^j_p(M)=Z^j_p(M,A) \oplus cE^j_p (M,\delta M\setminus  A).
\end{equation}
We now claim that  (for $p$ large or close to $1$)
\begin{equation}\label{eq_HEperp}
	Z^j_p(M,A)\cap E^j_p(M,A)^{\perp_p}=\mathscr{H}^j_p(M,A).
\end{equation}
Indeed, if a closed form  $\omega\in \wca^j_p(M,A)$
satisfies $[d\varphi,\omega]_p=0$ for every $\varphi \in  \wca^{j-1}_{p,\mba \setminus  A}(M) \subset\wca^{j-1}_{p}(M,A)$, then  $\delta_p \omega=0$ and $\omega\in \wsc_j ^p(M,\delta M \setminus  A)$, which means that $\omega\in \mathscr{H}_p^j(M,A)$. Conversely, it directly follows from (\ref{eq_lsp_crochet}) that the elements of $\mathscr{H}^j_p(M,A)$ are $\perp_p$ to all the elements of $E^j_p(M,A)$, yielding (\ref{eq_HEperp}).

 Thanks to (\ref{eq_proj_sum}) again, we immediately can deduce from (\ref{eq_HEperp}) ($E^j_p(M,A) $ is closed for $p$ large or close to $1$, by Corollary \ref{cor_image_fermee})
\begin{equation}
 Z^j_p(M,A)= E^j_p(M,A)\oplus \mathscr{H}^j_p(M,A),
\end{equation}
which, by (\ref{eq_zce}), yields the decomposition claimed in Theorem \ref{thm_hodge}.

We now prove Corollary \ref{cor_hodge}. Since  $\mathscr{H}^j_p(M,A)$ only consists of closed forms, there is a natural mapping $\iota:\mathscr{H}^j_p(M,A)\to H^j_p(M,A)$ induced by inclusion.  We are going to show that $\iota$ is a homeomorphism for every  $p\in (1,\infty)$ sufficiently large or  close to $1$.  

By Corollary \ref{cor_image_fermee}, $E_p^j(M,A)$ is closed for $p$ sufficiently large or small.
Let  $P:L^j_p (M) \to E_p^j(M,A) $ denote  the (nonlinear) $\perp_p$ projection onto this space (defined as explained in (\ref{eq_proj_inf}) and (\ref{eq_proj_sca})). Denote then by $Q:Z^j_p(M,A)\to  \mathscr{H}^j_p(M,A)$  the mapping defined as $Q(\omega):= \omega-P(\omega)$ (see (\ref{eq_HEperp})). It directly follows from (\ref{eq_proj_inf}) that $P(\omega +\alpha)=P(\omega)+\alpha$ for each $\alpha \in  E_p^{j}(M,A) $, which implies that $Q(\omega +\alpha)=Q(\omega)$. Hence, $Q$ induces a continuous mapping on the quotient $Z_p^j(M,A) / E_p^j(M,A)=H^j_p(M,A)$, say $\adh Q:H^j_p(M,A)\to  \mathscr{H}^j_p(M,A)$.

 We claim that $\adh Q\circ \iota: \mathscr{H}^j_p(M,A)\to  \mathscr{H}^j_p(M,A)$ and $\iota \circ\adh Q: H^j_p(M,A)\to H^j_p(M,A)$ are both equal to the identity map. Let us denote by $\adh \omega$ the class in $H^j_p(M,A)$ of an element $\omega\in Z^j_p(M,A)$. By definition of $\adh Q$ and $\iota$, we have for $\omega \in   \mathscr{H}^j_p(M,A)$:
 $$\adh Q(\iota(\omega))=\adh Q(\adh \omega)=\omega-P(\omega)\overset{(\ref{eq_HEperp})}=\omega, $$
 which establishes that $\adh Q\circ \iota$ is the identity. Moreover, for every $\omega\in Z^j_p(M,A)$, we have
   $$\iota(\adh Q(\adh \omega))=\iota(\omega-P(\omega)) =\adh\omega, $$
   since $P(\omega)$ is exact.  As $\iota$ and $\adh Q$ are continuous, we conclude that $\iota$ is a homeomorphism and thus that $\mathscr{H}^j_p(M,A)$ is a $\cc^0$ manifold of dimension $\dim H^j_p(M,A)$ (which is finite by Corollary \ref{cor_finitely_generated}). That $\adh Q(\omega)$ is the unique element minimizing the class follows from  (\ref{eq_proj_inf}).
\end{proof}

\end{subsection}
\end{section}

\begin{section}{Further results}\label{sect_further}
We are going to derive some extra consequences of our density theorem (Theorem \ref{thm_dense_formes}). We fix for all this section
	 a definably bi-Lipschitz trivial stratification $\Sigma$  of $\mba$ of which $M$ is a stratum (see Remark \ref{rem_M_stratum}).

	 Recall that $\wca^j_p(M,A)$ was defined in section \ref{sect_trd} as the space of forms $\omega\in \wca^j_p(M)$ satisfying $\trd^p \omega\equiv 0$ on $A\subset \delta M$, where $\trd^p$ is a trace operator defined implicitly (by duality with $L^{p'}$ forms). In section \ref{sect_trace_forms}, we will make this condition more  concrete (Theorem \ref{thm_trace_formes}) by defining an {\it explicit} trace operator (naturally induced by restriction of forms), denoted $\tra_S$, and showing that for $p$ sufficiently large $\trd^p \omega \equiv 0$ on  $A$ amounts to $\tra_S \omega=0$ for all $S\subset A, S\in \Sigma$ (Corollary \ref{cor_trace_formes}), which can be seen as a generalization of Lemma \ref{lem_traTra} to differential forms.
	In the case where $p$ is close to $1$, the situation is more delicate for ``a residue phenomenon'' may appear near the singularities (see Example \ref{exa_f_hol}). We will define residues of $L^1$ forms (see (\ref{eq_dfn_res}) and Theorem \ref{thm_welldefined}) and give a residue formula for integration by parts  (Theorem \ref{thm_residue_formula}), from which we will derive an explicit characterization of $\wca^j_p(M,A)$ in terms of vanishing of residues for $p$ close to $1$ (Corollary \ref{cor_vanishing_res}). This will enable us to provide a Hodge decomposition theorem without boundary condition (Corollary \ref{cor_hodge_deltaM_ptit}).
	
	Before entering the precise mathematical content, let us give a simple example that will shed light on the just above described results and   will stress the difference between the vanishing of $\trd^p \omega$ in the case $p$ large or close to $1$. 
	\begin{exa}\label{exa_f_hol}
		Let $f$ be an $L^1$ meromorphic function-germ on a neighborhood $U$ of the origin in $\mathbb C$. We can regard such a function as a differential $1$-form on $U\setminus \{0\}$, which, by Cauchy-Riemann's equations, is closed, and, by definition of $\trd^p$, satisfies
		 $$\trd^1 f\equiv 0  \; \mbox{ on }\;  \{0\} \;  \;\Leftrightarrow\; \res_0 f=0,$$
		 where $\res_0f$ denotes the residue of $f$ at the origin.
		  This fact can also be observed for $\trd^p f$ when $p$ is close to $1$.  Note on the other hand that for $p$ large, 
		  $f$ is $L^p$ if and only if it extends at the origin, and in this case 
		$$\trd^p f\equiv 0 \; \mbox{ on }\;  \{0\} \;\Leftrightarrow \; f(0)=0.$$
%
	\end{exa}

 Furthermore, it is not difficult to show (see Remark \ref{sect_concluding} (\ref{item_res_l2})) that for $p$ close to $1$,  forms $\omega\in \wca^j_p(M)$ that satisfy $\trd^p \omega\ne 0$ on a nonempty open subset $A$ of $\delta M$ must go to infinity fairly fast  as we are drawing near the points of $A$  and therefore cannot be approximated by forms of $\cc^{j,\infty}(\mba)$, which means that
  Theorem \ref{thm_dense_formes} (or Corollary \ref{cor_densite_non_normal}) is not true  in the case ``$p$ close to 1''.
 We will prove a weaker density result (Theorem \ref{thm_pprime}) which yet will be sufficient to establish Corollary \ref{cor_vanishing_res}.


		\begin{subsection}{Density in the case $p$ close to $1$.}\label{sect_pprime} Although we are unable to approximate in this case elements of $\wca^j_p(M,A)$ by forms that are smooth on $\mba$,   we can approximate them by smooth forms (on $M$) that vanish near $A$, which is satisfying in many situations:

	\begin{thm}\label{thm_pprime} Let $A$ and $E$ be  definable subsets of $\delta M$. For every $p\in [1,\infty)$ sufficiently close to $1$, the space $\wca^j_{p}(M)\cap \cc^{j,\infty}_{\mba\setminus A\cup E} (M)$
		is dense in $\wca^j_{p}(M,A)$  for every $j<m-\dim E-1$.
	\end{thm}
  In order to avoid repeating the sophisticated approximation process involved in the proof of Theorem \ref{thm_dense_formes} (this seems yet to be a possible proof, replacing $\rfr$ with $\Rfr$ in it), the idea will be to derive Theorem \ref{thm_pprime} from the latter one by means of a duality argument (see Proposition \ref{pro_densite}).
 It is worth mentioning here that, in the case where $A$ is a union of connected components of $\delta M$, we can spare Theorem \ref{thm_dense_formes} and our argument will  provide a direct short proof of the above theorem that holds for all $p$ (Proposition \ref{pro_densite_A_deltaM}).
 The following lemma, which characterizes  $\wca^j_p(M)'$, will be needed.
		\begin{lem}\label{lem_repres_T}
			If $T:\wca^j_p(M)\to \R $ is a continuous linear functional, $j\le m$,  $p\in [1,\infty)$, there are $\theta_0\in L^{m-j}_{p'}(M)$ and $\theta_1\in L^{m-j-1}_{p'}(M)$ such that $T=\theta_0+d^*\theta_1$, as functionals, i.e., for all $\alpha\in  \wca^j_p(M)$ we have:
			\begin{equation*}\label{eq_repres_T}T(\alpha)=<\theta_0,\alpha> +<\theta_1,d\alpha>.\end{equation*}
		\end{lem}
		\begin{proof}
			Let $G:\wca^j_p(M) \to L^j_p(M)\times L_p^{j+1}(M)$ be the linear mapping defined by $G(\alpha):= (\alpha,d \alpha)$. Clearly, $G$  induces an isomorphism of Banach spaces  from $\wca^j_p(M) $ onto its (closed) image. Its dual mapping $G': L^{m-j}_{p'}(M)\times L_{p'}^{m-j-1}(M)\to \wca^j_p (M)'$ is therefore surjective, which means that for every $T\in \wca^j_p (M)'$ there are $\theta_0\in L^{m-j}_{p'}(M)$ and $\theta_1\in L^{m-j-1}_{p'}(M) $ satisfying the property of the lemma.
		\end{proof}
		This lemma was needed to establish the following fact.
		\begin{pro}\label{pro_densite}
			Let $A,B$, and $E$ be  subsets of $\delta M$ partitioning this set. For each $p\in (1,\infty)$ and each $j$, the following statements are equivalent:
			\begin{enumerate}[(i)]
				\item  (\ref{eq_lsp_j}) holds for all $\alpha \in \wca^j_p(M,A)$ and all $\beta\in \wca^{m-j-1}_{p'}(M,B)$.
				\item $\cc^{j,\infty}_{\mba \setminus (A\cup E)} (M)\cap \wca^j_p(M)$ is dense in $\wca^j_p(M,A)$.
				\item $\cc^{m-j-1,\infty}_{\mba \setminus (B\cup E)} (M)\cap \wca^{m-j-1}_{p'}(M)$ is dense in $ \wca^{m-j-1} _{p'}(M,B)$.
			\end{enumerate}
		\end{pro}
		\begin{proof} If  $(ii)$ is equivalent to $(i)$  for all $j$ and all such sets $A,B$, and $E$ then so is  $(iii)$, since we can interchange $A$ with $B$ and $j$ with $(m-j-1)$. 
			Observe also that the implication  $(ii) \Rightarrow (i)$ directly comes down from the definition of the space  $\wca^j_p(M,A)$, since $ \mba\setminus (A\cup E)= M \cup B$ (the space $\cc^{j,\infty}_{\mba \setminus (A\cup E)} (M)\cap \wca^j_p(M)$ is dense in $\wca^j_{p,\mba \setminus (A\cup E)}(M) $ by the properties of de Rham's regularization operators \cite{derham, goldshtein}).  We thus just need to show that $(i)$ implies $(ii)$. 
			
			Since we already know that $\cc^{j,\infty}_{\mba \setminus (A\cup E)} (M)\cap \wca^j_p(M)$ is dense in $\wca^j_{p,\mba \setminus (A\cup E)}(M) $, we just need to prove, assuming $(i)$, that every continuous linear functional $T:\wca^j_p(M,A)\to \R$ that vanishes on $ \wca^j_{p,\mba \setminus (A\cup E)}(M)$ is identically zero.
			
			By Lemma \ref{lem_repres_T}, such a functional $T$ can be written $\theta_0 +d^*\theta_1$, with $\theta_0 \in L^{m-j}_{p'}(M)$ and $\theta_1 \in L^{m-j-1}_{p'}(M)$. As $T\equiv 0$ on $ \wca^j_{p,\mba \setminus (A\cup E)}(M)$, we must have for all $\alpha$ in this space \begin{equation}\label{eq_alpu}<\theta_0,\alpha>=-<\theta_1,d \alpha>,\end{equation}
			which entails that $\theta_0=(-1)^{m-j-1}\,d \theta_1$,  and consequently that $\theta_1\in \wca^{m-j-1}_{p'}(M)$.  Equality (\ref{eq_alpu}) then shows that  (\ref{eq_lsp_j}) holds for every pair $\alpha,\theta_1$, as soon as $\alpha\in \wca^j_{p,\mba \setminus (A\cup E)}(M)=\wca^j_{p,M\cup B}(M)$, which exactly means that $\theta_1 \in \wca^{m-j-1}_{p'}(M,B)$. By $(i)$, we deduce that
		 (\ref{eq_lsp_j}) must actually hold for each pair $\alpha,\theta_1$, when  $\alpha\in \wca^j_p(M,A)$, and consequently $$T(\alpha)=(-1)^{m-j-1}<d \theta_1, \alpha>+<\theta_1, d \alpha>\overset{(\ref{eq_lsp_j})}=0,$$
			for all such $\alpha$.
		\end{proof}
		\begin{rem}\label{rem_p=1_densite}
		Equivalence $(i)\Leftrightarrow (ii)$ still holds true for $p=1$, with the same proof (since Lemma \ref{lem_repres_T} holds in this case).
		\end{rem}
		
\begin{proof}[Proof of Theorem \ref{thm_pprime}] Fix $j\le m$ and let $B:=\delta M\setminus (A\cup E)$.
		Since Corollary \ref{cor_densite_non_normal} establishes that the space $\cc^{m-j-1,\infty}_{\mba \setminus (B\cup E)} (M)\cap \wca^{m-j-1}_{p'}(M)$ is dense in $ \wca^{m-j-1} _{p'}(M,B)$ for all  $p\in (1,\infty)$ sufficiently close to $1$ and $m-j-1>\dim E$,	 Proposition \ref{pro_densite} implies that  $\cc^{j,\infty}_{\mba \setminus (A\cup E)} (M)\cap \wca^j_p(M)$ is dense in $\wca^j_p(M,A)$ for all such $p$ and $j$.	 The case $p=1$ follows from Remarks \ref{rem_p=infty} and \ref{rem_p=1_densite}.	\end{proof}

	\end{subsection}


\begin{subsection}{The case where  $A$ is both open (in $\delta M$) and closed.}  When $A$ is the union of some connected components of $\delta M$, it is possible to give a simple proof (i.e. which does not involve Theorem \ref{thm_dense_formes}) of Theorem \ref{thm_pprime} (with $E=\emptyset$) and Corollary \ref{cor_lsp_A}, which then hold for all $p$. We indeed have:

	\begin{pro}\label{pro_densite_A_deltaM} Let $A$ be an open subset of $\delta M$ and let $K$ be a compact subset of $A$. For each $j$ and $p\in [1,\infty)$, equality
		(\ref{eq_lsp_A}) holds for every $\alpha \in \wca^j_p(M,A)$ and $\beta\in   \wca^{m-j-1}_{p'}(M,\delta M\setminus K)$.
		As a matter of fact,  if $A$ is both open and closed in $\delta M$ then    $\cc_{\mba \setminus A}^{j,\infty}(M)\cap \wca_p^j(M)$ is dense in $\wca_p^j(M,A)$ for all $j\le m$ and the conclusion of Corollary \ref{cor_lsp_A} is valid, for all $p\in [1,\infty)$.
	\end{pro}
	\begin{proof}  Let $\alpha\in \wca_p^j(M,A)$ and $\beta \in \wca_{p'}^{m-j-1}(M,\delta M\setminus K)$. Since we can use a partition of unity, it suffices to show, given $\xo\in \delta M$, that  (\ref{eq_lsp_A}) holds when  these two forms are supported in a small neighborhood of $\xo$. If $\xo\in A$ and $\beta$ is supported in the vicinity of $\xo$ then, since $A$ is open, we have $\beta\in  \wca_{p',M\cup A}^{m-j-1}(M)$ (if the support of $\beta$ is sufficiently small), which means that (\ref{eq_lsp_A}) follows from the definition of $ \wca^j_p(M,A)$. In the case where $\xo \in \delta M\setminus A\subset \delta M\setminus K$, since $\delta M\setminus K$ is open in $\delta M$, we have $\alpha\in  \wca_{p, M\cup K}^j(M)$ (if the support of $\alpha$ is sufficiently small), so that the needed fact comes down from the definition of $ \wca_{p'}^{m-j-1}(M,\delta M\setminus K)$. The last sentence then follows from   Proposition \ref{pro_densite} (and Remark \ref{rem_p=1_densite} in the case $p=1$), for we can choose $K=A$ if $A$ is closed.
	\end{proof}

	\end{subsection}


	\subsection{In the case  $p$ large, the trace  as a current.}\label{sect_trace_forms} As  explained in the introduction of this section, our aim is to provide explicit descriptions of $\wca^j_p(M,A)$ (Corollaries \ref{cor_vanishing_res} and \ref{cor_trace_formes}). We need for this purpose an explicit trace operator $\tra_S$ for $L^p$ differential forms (if $S$ is a stratum), which will be defined in this subsection (Theorem \ref{thm_trace_formes}), and a residue operator, which will be defined in the next one (Theorem \ref{thm_welldefined}).

	 For simplicity, we assume in Theorem \ref{thm_trace_formes}  that $M$  is normal. In the case where this assumption fails, one can always define a multi-valued trace using a normalization, as done in \cite{trace} in the case of Sobolev spaces of functions. We write $\dsc^j(M)$  for the space of $j$-currents on $M$.

We recall that throughout section \ref{sect_further}, $\Sigma$ stands for a definably bi-Lipschitz trivial stratification  of $\mba$ of which $M$ is a stratum, and that $\Sigma_{\delta M}$ then stands for the stratification induced on $\delta M$ (see Definition \ref{dfn_stratifications}).

\begin{thm}\label{thm_trace_formes} If $M$ is normal then
	for every $p\in [1,\infty)$ sufficiently large and  $S\in \Sigma_{\delta M}$, the mapping $\cc^{j,\infty } (\mba) \to  \cc^{j,\infty } (S) $, $\omega\mapsto   \omega_{|S} $, uniquely extends to a continuous mapping  $\tra_S:\wca^j_p (M)  \to  \dsc^j (S) $.
\end{thm}
\begin{proof}
Since we can use a partition of unity, it suffices to prove that every $\xo\in S\in \Sigma_{\delta M}$ has a neighborhood $U$ in $\mba$ such that $\omega\mapsto \omega_{|S\cap U}$ induces a continuous mapping from $(\cc_0^{j,\infty}(U),||\cdot||_{\wca^j_p (U\cap M) })$ to   $\dsc^j(U\cap S)$ (for $p\in [1,\infty)$ sufficiently large). Uniqueness of the extension will then follow from Theorem \ref{thm_dense_formes}.  

Fix $x_0\in S\in \Sigma_{\delta M}$.   By definition of definable bi-Lipschitz triviality, $\xo$ has a neighborhood $\uxo$ in $\mba$ for which there is a definable bi-Lipschitz homeomorphism $\Lambda: U_\xo \to (\pi^{-1}(\xo)\cap \adh{M})\times W_\xo$, where $\pi:U_\xo \to S$ is a definable $\cc^\infty$ retraction and $W_\xo$ is a neighborhood of the origin in $S$ (here $\dim S$ may be $0$, in which case $\Lambda$ is the identity).

Up to a coordinate system of $S$,  we may assume that $W_\xo$ is a neighborhood of the origin in $\{0\} \times \R^k$ (where $k=\dim S$, we will sometimes regard $W_\xo$ as a subset of $\R^k$).  We let $F:=\pi^{-1}(0)\cap M$, as well as $F^\eta:=F\cap \bou (0,\eta)$ and
$G^\eta:=F\cap \sph (0,\eta)$.

Let $\omega\in \cc_0^{j,\infty}(U_\xo)$ and $\varphi\in \cc_0^{k-j,\infty}(W_\xo)$.  We are going to show that (if $p$ is large) \begin{equation}\label{eq_claim_tra_form}
	|\int_{W_\xo} \omega\wedge \varphi|\le C ||\omega||_{\wca^j_p(U_\xo \cap M)} \cdot  ||\varphi||_{\wca^{k-j}_{\infty}(W_\xo )}, 
\end{equation}
for some constant $C$ independent of $\omega$ and $\varphi$, which will clearly yield the desired fact (we will actually just need a continuity property of $\omega$ at points of $S$, so that this inequality indeed holds for $\omega$  differentiable stratified $j$-form  on $U_\xo$,  see (\ref{eq_stratified_form})).

 Let $\beta$ denote the form provided by applying Lemma \ref{lem_beta}   to the manifold $F$. We will regard this form as defined on $F\times W_\xo$, identifying it with $\mu^*\beta$, where $\mu:F\times W_\xo \to F$ is the canonical projection. We first check that
\begin{equation}\label{eq_om_moins_piom}
	\lim_{\eta \to 0} \int_{G^\eta \times W_\xo} \beta(x)\wedge \Lambda^{-1*}(\omega-\pi^*\omega_{|S})(x,y) \wedge \pi^*\varphi(x,y)=0.
\end{equation}
For $(x,y)\in F\times W_\xo$, write for simplicity
$$\Lambda^{-1*}(\omega-\pi^*\omega_{|S})\wedge \pi^*\varphi(x,y) =\omega_1(x,y)+\omega_2(x,y)\, dy_1\wedge \dots \wedge d y_k,  $$
with
  \begin{equation}\label{eq_om2_geta}\omega_{1,e_1\otimes\dots \otimes e_k} \equiv 0,\end{equation} 
  where $e_1,\dots, e_k$ is the canonical basis of $\{0\}\times \R^k\supset S$.
As $\beta_{e_i}\equiv 0$ for all $i=1,\dots k$, (\ref{eq_om2_geta}) entails (since $\dim G^\eta=k-1$) 
\begin{equation}\label{eq_beta_omega1_0}\left( \beta \wedge \omega_1\right)_{|G^\eta \times W_\xo}\equiv 0 .\end{equation}
Moreover, since $\omega$ is smooth and $\Lambda$ commutes with $\pi$, $\omega_2(x,y)$ must tend to zero as $x$ tends to $0$, so  that, by (\ref{eq_beta_S_geta})  we see that
$\lim_{\eta\to 0}	\int_{G^\eta \times W_\xo} \beta\wedge \omega_2\wedge dy_1\wedge \dots \wedge d y_k=0$,
which, together with (\ref{eq_beta_omega1_0}), yields (\ref{eq_om_moins_piom}).

We are ready to establish (\ref{eq_claim_tra_form}). Observe first that it follows from the definition of trivializations that $\Lambda^{-1*}\pi^*\omega_{|S}=\pi^*\Lambda^{-1*}\omega_{|S}=\pi^*\omega_{|S}$. 
As $\varphi$ is smooth,  $\pi^*\varphi(x,y)$ tends to  $\pi^*\varphi(0,y)$ as $x$ goes to zero, which implies that for almost every $y$
\begin{eqnarray}\label{eq_unif_conv_a}
	\lim_{x\to 0, x\in F} \Lambda^{-1*} \pi^*\omega_{|S}(x,y)\wedge \pi^*\varphi(x,y) =\omega_{|S}\wedge \varphi(y),
\end{eqnarray}
and hence,
 \begin{eqnarray*}
 \lim_{\eta\to 0}	\int_{G^\eta \times W_\xo} \beta \wedge\Lambda^{-1*}\pi^*\omega_{|S}\wedge \pi^*\varphi &=& \lim_{\eta\to 0}	\int_{G^\eta } \beta(x) \wedge\int_{y\in W_\xo} \Lambda^{-1*}\pi^*\omega_{|S}\wedge \pi^*\varphi(x,y)\\
 &=&  \int_{ W_\xo}\omega\wedge \varphi \quad \mbox{(using (\ref{eq_beta_S_geta}) and (\ref{eq_unif_conv_a})),} \end{eqnarray*}
which, together with (\ref{eq_om_moins_piom}),  entails:
\begin{equation}\label{eq_om}
	\lim_{\eta\to 0} \int_{G^\eta \times W_\xo} \beta \wedge\Lambda^{-1*}\omega\wedge \pi^*\varphi =\int_{ W_\xo}\omega\wedge \varphi.
\end{equation}
Note now that  $P^\eta:=(F^\ep \setminus F^\eta ) \times W_\xo $ is a manifold with boundary  $G^\eta  \times W_\xo $  on which $\beta\wedge \Lambda^{-1*}\omega\wedge \pi^*\varphi  $ induces a compactly supported stratifiable form (see section \ref{sect_stokes_formula}).  Stokes' formula (\ref{eq_stokes_stratified_leaf}) (applied to the interior part of this manifold with boundary) gives us 
$$\int_{G^\eta \times W_\xo} \beta \wedge\Lambda^{-1*}\omega\wedge \pi^*\varphi =\int_{P^\eta}  d(\beta \wedge\Lambda^{-1*}\omega\wedge \pi^*\varphi), $$
and consequently, using H\"older's inequality,
$$|\int_{G^\eta \times W_\xo} \beta \wedge\Lambda^{-1*}\omega\wedge \pi^*\varphi |\lesssim ||\beta||_{\wca^{m-k-1}_{p'}(U_\xo)}\cdot ||\omega||_{\wca^j_p(U_\xo)}\cdot ||\varphi||_{\wca^{k-j}_\infty(U_\xo)}, $$
which, via (\ref{eq_om}),  gives (\ref{eq_claim_tra_form}), after passing to the limit as $\eta \to 0$.
\end{proof}

\begin{rem}\label{rem_invariance_trace}
 Inequality (\ref{eq_claim_tra_form}) is true as soon as $\omega$ is a stratified form (as emphasized right after (\ref{eq_claim_tra_form})). It means that the identity $\tra_S \omega=\omega_{|S}$ actually holds for every stratified form $\omega$ on $\mba$ (see (\ref{eq_stratified_form})).
	\end{rem}


\subsection{Residues of $L^1$ forms}\label{sect_residues}We recall that, throughout section \ref{sect_further}, $\Sigma$ is a definably bi-Lipschitz trivial stratification  of $\mba$ of which $M$ is a stratum. The situation is more complicated in the case where $p$ is close to $1$ (see Example \ref{exa_f_hol}).  In this case, the vanishing of a form on a stratum of $\delta M$ will be characterized by the vanishing of a residue operator that we define in this subsection.  The residue of a form  $\omega \in \wca^{m-j-1}_1(M)$ (i.e. of codimension $(j+1)$ in $M$) on a stratum $S$ will be  a $(k-j)$-current on $S$, where $k:=\dim S$ (i.e. of codimension $j$ in $S$). 

 Let $S\in \Sigma_{\delta M}$  (see Definition \ref{dfn_stratifications} for $\Sigma_{\delta M}$) and let $\pi:U \to S$ be a $\cc^2$  definable retraction and $\rho :U \to \R$ a $\cc^2$ definable function satisfying $\rho^{-1}(0)=S$, where $U\supset S$ is an open subset of $\R^n$ (see  \cite[section 2.4]{livre} for existence of definable tubular neighborhoods).  
 
 Given  $\omega \in \wca^{m-j-1}_1(M)$ and $\varphi\in \cc^{j,\infty}_0(S)$, we set $V^\ep_{\rho}:=\{x\in M:\rho(x)=\ep\}$ as well as
\begin{equation}\label{eq_dfn_res}
	<\res_{_S} \omega,\varphi>:= \underset{\ep \to 0}{\esslim} \int_{V^\ep_{\rho} } \omega\wedge \pi^*\varphi  ,\end{equation}
where $\esslim$ stands for the essential limit.  This definition of course requires to show that this limit exists and  is independent of $(\pi,\rho)$. When this happens for all $\varphi$ supported in some neighborhood $U$ of a point $x_0\in S$, we will say that $\res_{_S} \omega$ is {\bf well-defined near $x_0$}.

We will show that $\res_{_S}  \omega$ is well-defined as soon as $\res_{_Y} \omega$ is well-defined as well and is in addition $L^1$ for each stratum $Y$ satisfying $S\prec Y$ (Theorem \ref{thm_welldefined}, see Remark \ref{rem_frontier} for $\prec$). Intuitively speaking, this means that we are able to define the residue of a form on a stratum provided it does not oscillate too much on the higher strata. We give a counterexample that shows that  the limit (\ref{eq_dfn_res}) does not always exist when this condition is not met (Remark \ref{sect_concluding} (\ref{item_resl1})).
\begin{dfn}\label{dfn_l1op}
	Given $Y\in \Sigma$, we will say that a linear functional $T:\cc_0^{j,\infty}(Y)\to \R$ is $L^1$ near a point $\xo\in \adh Y$  if this point has a neighborhood $U$ in $\R^n$ such that $T$ coincides with an element of $\alpha\in L^{l-j}_1(Y\cap U)$, $l=\dim Y$, on $\cc_0^{j,\infty} (Y\cap U)$, i.e. $T(\varphi)=<\alpha,\varphi>$ for all $\varphi\in \cc_0^{j,\infty} (Y\cap U)$.\end{dfn}

It is worthy of notice that in Example \ref{exa_f_hol}, i.e. in the case of a holomorphic function that has a pole of order $1$, the above definition of residues coincides with the usual notion. In this example,  the only stratum of the boundary of the domain near the origin being $\{0\}$, it is maximal (for the relation $\prec$), so that, unlike in (\ref{item_well_dfn}) of the theorem below, there is no integrability condition to put.
\begin{thm}\label{thm_welldefined}  Let $S\in \Sigma_{\delta M}$ and  $\xo\in S$. 
	\begin{enumerate}
		\item\label{item_well_dfn_max}	If $S$ is a maximal stratum of $\delta M$ then $\res_{_S} \omega$ is well-defined for all $\omega\in\wca_1^{m-j-1} (M)$. 
		\item\label{item_well_dfn}  More generally, if $\res_{_Y} \omega$ is (well-defined and) $L^1$ near $\xo$ for all   $S\prec Y \prec M$  then $\res_{_S} \omega$ is well-defined near $\xo$. 
		\item\label{item_res_cur} 	If we let 
$$ \dor^j_S:=\{\omega \in \wca^j_1 (M): \res_{_Y} \omega \mbox{ is (well-defined and) $L^1$ for all $S\prec Y\prec M$} \}, $$
 then  (\ref{eq_dfn_res}) provides a linear mapping $\res_{_S} :\dor_S^{m-j-1}\to \dsc^{k-j}(S),$ where $k:=\dim S. $  
	\end{enumerate}
	
\end{thm}

The  operator $\res_{_S}$ will give us a natural and efficient way to check whether integrations by parts are licit. We give below a  formula providing information on this issue. We will say that a form $\omega\in \wca^j_1 (M)$ has {\bf $L^1$ residues} at $\xo\in \mba$ if  $\res_{_S}\omega$ is $L^1$ near $\xo$, for all $S\in \Sigma_{\delta M}$.
\begin{thm}\label{thm_residue_formula} (Residue formula) Let  $\varphi \in \cc^{j,\infty} (\mba)$. If  $\omega \in  \wca^{m-j-1}_1 (M) $ has $L^1$ residues  near every point of $\supp_\mba \varphi$ then
\begin{equation}\label{eq_residue_formula}
	<d\omega,\varphi>= (-1)^{m-j}<\omega,d\varphi> +\sum_{S\in \Sigma _{\delta M}} <\res_{_S} \omega, \varphi_{|S}>.
\end{equation}
\end{thm}
\begin{proof}[Proof of Theorems \ref{thm_welldefined} and \ref{thm_residue_formula}.] We first focus on the definition of residues, i.e. on (\ref{item_well_dfn_max}) and (\ref{item_well_dfn}) of  Theorem \ref{thm_welldefined}. As the statement of this theorem suggests, we shall argue by induction with respect to the partial strict order relation $\prec$. Moreover,
	the statement of this theorem being local, we may argue up to a finite partition of unity  subordinate to a covering of $\mba$, which means that it suffices to show the desired statement for forms that are supported in a small neighborhood of a point of $ \delta M$ (the statements are obvious near a point of $M$). 
	
	 We thus fix $\xo \in \delta M$ belonging to a stratum $S\in \Sigma$, take $\omega \in \wca^{m-j-1}_1 (M) $ supported in some small neighborhood $U_\xo$ of $\xo$  in $\R^n$, and will assume  $\res_{_Y} \omega$ to be well-defined and $L^1$ for every $Y\in \Sigma_{\delta M}$ satisfying $S\prec Y$ (this condition being empty if $S$ is maximal in $\delta M$, the first step of this induction is vacuous).
	
	If
	 $\supp_\mba\omega$ (that we can choose as small as we please) is sufficiently small, it will only meet strata that contain $S$ within their closure,  so that $\res_{_Y} \omega$ will be well-defined and $L^1$ for every such stratum by the inductive assumption. We will be proving (\ref{eq_residue_formula}) simultaneously.  Since this formula can also be proved up to a partition of unity, it suffices to prove it for a form $\omega$ as in the above paragraph,  and, as we can  also argue by induction on the  strata that meet $\supp_\mba \omega \cap \supp_\mba \varphi$ (with respect to the partial order relation $\prec$), we will assume that  (\ref{eq_residue_formula}) holds for all  pairs $\omega$, $\varphi$ supported in $\uxo$ for which $\supp_\mba\omega \cap \supp_\mba \varphi \cap S=\emptyset$.

  To check that the limit (\ref{eq_dfn_res}) exists, take  a tubular neighborhood $(\pi,\rho)$ of $S$ as well as $\varphi\in \cc_0^{j,\infty}(
  U_\xo)$  (although we just need to check that this limit  exists for a form $\pi^*\varphi$, with $\varphi\in \cc_0^{j,\infty}(U_\xo\cap S)$, we will carry out our computations for a form $\varphi\in \cc_0^{j,\infty} (U_\xo)$, which will be useful to establish (\ref{eq_residue_formula})).
 
 Let for $\eta$ and $\mu$ positive $\psi^S_{\mu,\eta}(x):=\psi_{\mu,\eta} (\rho(x))$, $x\in \uxo\cap M$, where $\psi_{\mu,\eta}(t):=\psi_\eta (t-\mu)$ (with $\psi_\eta$ as in (\ref{eq_psieta})), and 	set for simplicity 
		 \begin{equation*}
		f(\mu,\eta)=(-1)^{m-j}<\omega, d \psi^S_{\mu,\eta}\wedge \varphi>.
	\end{equation*}
 We first check that for almost every $\mu>0$ small
\begin{equation}\label{eq_ab}
	\lim_{\eta \to 0} f(\mu,\eta) =\int_{V_{\rho}^\mu} \omega\wedge \varphi, \quad \mbox{where }\;\; V_{\rho}^\mu:= \rho^{-1}(\mu)\cap M.
\end{equation}
 As $\rho$ is Lipschitz, by \cite[Theorem 7.6.3 and Proposition  7.6.5]{gvhandbook} or \cite[Corollary 3.2.12 and Proposition 3.2.14]{livre},  there are $\delta>0$ and a definable  homeomorphism  $$\Phi:\{x\in\uxo\cap M:\rho(x)<\delta\} \to V_{\rho}^\delta \times  (0,\delta)$$ of type $x\mapsto (\phi(x),\rho(x))$, for some mapping $\phi$, which is bi-Lipschitz on $\{x\in \uxo \cap M:\sigma<\rho(x)<\delta\}$ for every $\sigma\in(0,\delta)$.  We  denote by  $\tilde{\omega}$, $\tilde{\varphi},\tilde{\psi}^S_{\mu,\eta} $, the respective push-forwards under $\Phi$ of $\omega,\varphi,\psi^S_{\mu,\eta}$,  and set for $\mu<\delta$ \begin{equation}\label{eq_vtildemu}\tilde{V}^\mu:=\Phi(V_{\rho}^\mu)= V_{\rho}^\delta \times \{\mu\}.\end{equation}
  In particular, if we decompose $\tilde{\omega}\wedge \tilde{\varphi}$ as $\tilde{\gamma}_1+ dt\wedge\tilde{\gamma}_2 $ (i.e. $\tilde{\gamma}_2:=(\tilde{\omega}\wedge \tilde{\varphi}){_{\pa_t}}$) we have for  $x\in \tilde{V}^\mu$, setting $\theta_\eta (t):= \frac{d\psi_{\eta}}{dt}(-t)$, 
\begin{equation}\label{eq_getamu}\int_{t=0} ^\ep d \tilde{\psi}^S_{\mu,\eta}(t)\wedge\tilde{\omega}\wedge \tilde{\varphi}(x,t)= \int_0^{\ep }\tilde{\gamma}_1(x,t)\frac{d\psi_{\eta}}{dt}(t-\mu)dt\overset{(\ref{mol})}{=}(\tilde{\gamma}_1*_1\theta_\eta) (x,\mu ).\end{equation}
 But since $\int\theta_\eta (t)dt=1$,  Lemma \ref{lem_conv} $(\ref{item_conv_lem})$ implies that for $\sigma \in (0,\delta)$, 
\begin{equation*}\lim_{\eta\to 0}||\tilde{\gamma}_1-\tilde{\gamma}_1*_1\theta_\eta ||_{L^1( V_{\rho}^\delta\times[\sigma,\delta ])}=0,  \end{equation*}
 which entails for almost every $\mu\in (0,\delta)$
$$	\lim_{\eta \to 0}\int_{\tilde{V}^\mu}\tilde{\gamma}_1*_1\theta_\eta= \int_{\tilde{V}^\mu}\tilde{\gamma}_1\overset{(\ref{eq_vtildemu})}= \int_{\tilde{V}^\mu}\tilde{\omega}\wedge \tilde{\varphi},  $$
and consequently for such $\mu $ (integrating (\ref{eq_getamu}) and passing to the limit as $\eta\to 0$)
\begin{equation}\label{eq_lim_conv}\lim_{\eta\to 0} \int_{\tilde{V}^\mu} \int_0 ^\ep d \tilde{\psi}^S_{\mu,\eta}(t)\wedge\tilde{\omega}\wedge \tilde{\varphi}(x,t)=\int_{\tilde{V}^\mu}\tilde{\omega}\wedge \tilde{\varphi}.  \end{equation}
  We now can write
\begin{equation*}
	\lim_{\eta \to 0}(-1)^{m-j} <\tilde\omega, d\tilde{\psi}^S_{\mu,\eta}\wedge \tilde\varphi>=(-1)^{m-j}	\lim_{\eta \to 0}\int_{\tilde{V}^\mu} \int_0 ^\ep \tilde\omega\wedge d\tilde{\psi}^S_{\mu,\eta}\wedge \tilde\varphi\overset{(\ref{eq_lim_conv})}{=}\int_{\tilde{V}^\mu}\tilde{\omega}\wedge \tilde{\varphi},
\end{equation*}
which, up to a pull-back by $\Phi$, is exactly (\ref{eq_ab}).
As $\Phi$ induces a definable bi-Lipschitz homeomorphism between $V_{\rho}^\mu$ and $\tilde{V}^\mu$ for every $\mu$, this yields  (\ref{eq_ab}).

 Since $\psi_{\mu,\eta}^S$ vanishes near $S$, by induction, we know that (\ref{eq_residue_formula}) holds for $\omega$ and  $\psi_{\mu,\eta}^S\varphi$, which entails that	(setting for simplicity $i:=m-j+1$)
	\begin{equation*}\label{eq_rec_f} f(\mu,\eta)=	<d\omega,\psi^S_{\mu,\eta}\varphi>+(-1)^{i}<\omega,\psi^S_{\mu,\eta} d\varphi>-\sum_{S\prec Y\prec M } <\res_{_Y} \omega, \psi^S_{\mu,\eta} \varphi_{|Y}>.\end{equation*}
	As we assume that $\res_{_Y} \omega$ is  $L^1$  and   $  \varphi$ is $\cc_0^\infty$, the form $\res_{_Y} \omega\wedge   \varphi_{|Y}$ is $L^1$. As $\psi^S_{\mu,\eta}$ is uniformly bounded and tends  (pointwise) as $\eta \to 0$ to the function  $\psi^S_{\mu,0} $ defined as $\psi^S_{\mu,0}(t):=0 $ if $t\le \mu$ and $\psi_{\mu,0}(t):=1$ otherwise,
		 	 passing to the limit as $\eta\to 0$ in the above equality, we get thanks to (\ref{eq_ab}) and Lebesgue's Dominated Convergence Theorem
		$$\int_{V_{\rho}^\mu} \omega\wedge \varphi =<d\omega,\psi^S_{\mu,0}\varphi>+(-1)^{i}<\omega, \psi^S_{\mu,0}d\varphi>-\sum_{ S\prec Y\prec M } <\res_{_Y} \omega,\psi^S_{\mu,0}  \varphi_{|Y}>.$$
 Since the necessary integrability conditions hold by assumption for $\omega, \varphi,d\omega,d\varphi, \res_{_Y} \omega$, and $ \varphi_{|Y}$, passing now to the limit as $\mu \to 0$ shows that $\lim _{\mu \to 0}\int_{V_{\rho}^\mu} \omega\wedge \varphi$ exists  and 
  \begin{equation}\label{eq_resint}
  \lim _{\mu \to 0}\int_{V_{\rho}^\mu}   \omega\wedge \varphi =<d\omega,\varphi>+(-1)^{i}<\omega, d\varphi>-\sum_{S\prec Y\prec M  } <\res_{_Y} \omega,  \varphi_{|Y}>,
  \end{equation} 
which is independent of $\rho$.  Moreover, this equality also shows that, in order to prove that the limit displayed in  (\ref{eq_dfn_res}) is independent of $\pi$, 
we just need to check that 
\begin{equation}\label{eq_only}
\lim _{\mu \to 0}\int_{V_{\rho}^\mu} \omega\wedge (\varphi-\pi^*\varphi_{|S})=0
\end{equation}
(for then the limit (\ref{eq_dfn_res}) will be equal to the right-hand-side of (\ref{eq_resint}) for all $\pi$).
Clearly, we can assume that $\varphi_{|S}=0$. Because we can argue up to a coordinate system of $S$ and since it was established that the considered limit  is independent of $\rho$, we will assume that $S$ is a neighborhood of the origin in $ \{ 0_{\R^{n-k}}\}\times \R^k$, that $\rho$ is  the euclidean distance to $S$, and that $\pi$ is the orthogonal projection onto $S$ (as $\pi$ is a submersion at points of $S$, there is such a coordinate system).  Set then $\tilde{\pi}(x,t)=tx+(1-t)\pi(x)$, as well as
$$\mathbf{A} \varphi(x):=\int_0^1 (\tilde{\pi}^* \varphi)_{\pa_t}(x,t)\,dt.  $$
By Cartan's magic formula, we have (recall that we assume $\varphi_{|S}=0$):
\begin{equation}\label{eq_cmf}\mathbf{A} d\varphi+d\mathbf{A} \varphi =\varphi.\end{equation}
Moreover, since $|\frac{\pa \tilde{\pi}}{\pa t}(x,t)|= |x-\pi(x)|$, we have:
\begin{equation}\label{eq_u}
|\mathbf{A} \varphi(x)| \lesssim  |x-\pi(x)|\cdot ||\varphi||_{L^\infty(U_\xo)}.
\end{equation}
Define then a function $\hat \psi_\eta$ as $\hat \psi_\eta(x):= \psi_\eta(|x|)$ (where $\psi_\eta$ is as in (\ref{eq_psieta})), and set 
$$\mathbf{B}_\eta \varphi:=\hat\psi_\eta \varphi  +  d\hat\psi_\eta \wedge\mathbf{A} \varphi  .  $$
 Since $|d\hat\psi_\eta| \overset{(\ref{eq_psieta_ineq})}{\lesssim} 1/|x-\pi(x)| $, by (\ref{eq_u}), $||\mathbf{B}_\eta \varphi||_{L^\infty(U_\xo)}$ is bounded independently of $\eta$ for every $\varphi$. As $\mathbf{B}_\eta \varphi$ tends to $\varphi$ (pointwise), by Lebesgue's Dominated Convergence Theorem 
 $$\lim_{\eta \to 0}< d\omega ,\mathbf{B}_\eta \varphi >=<d \omega , \varphi> \et \lim_{\eta \to 0}<\res_{_Y} \omega, (\mathbf{B}_\eta \varphi)_{|Y} >=<\res_{_Y} \omega ,\varphi_{|Y} >,$$ 
 for every $Y\succ S$.  Moreover, as $d\mathbf{B}_\eta \varphi \overset{(\ref{eq_cmf})}=\mathbf{B}_\eta d\varphi$ which is also bounded independently of $\eta$ for the same reason,
applying again Lebesgue's Theorem,  we see that
$$\lim_{\eta \to 0}< \omega , d\mathbf{B}_\eta \varphi >=< \omega , d \varphi >.$$
As $\mathbf{B}_\eta \varphi$ is a smooth form that vanishes near $S$, by induction,  equality (\ref{eq_residue_formula}) holds for the pair $\omega,\mathbf{B}_\eta \varphi$. Writing this equality and passing to the limit as $\eta\to 0$, we get, thanks to  the three just above limits,  
$$ <d\omega,\varphi>=(-1)^{m-j}<\omega,d\varphi>+ \sum_{S\prec Y\prec M} <\res_{_Y} \omega ,  \varphi_{|Y}>. $$
 By (\ref{eq_resint}), this shows that  $\lim_{\mu\to 0} \int_{V_{\rho}^\mu} \omega\wedge \varphi =0$, yielding (\ref{eq_only}) (we are assuming $\varphi_{|S}=0$, see the paragraph right after (\ref{eq_only})), and establishing (\ref{item_well_dfn_max}) and (\ref{item_well_dfn}) of Theorem \ref{thm_welldefined}.
 
 To prove Theorem \ref{thm_residue_formula}, observe now that 
 $$\lim _{\mu \to 0}\int_{V_{\rho}^\mu} \omega\wedge \varphi\overset{(\ref{eq_only}) }=\lim _{\mu \to 0}\int_{V_{\rho}^\mu} \omega\wedge\pi^*\varphi_{|S}\overset{(\ref{eq_dfn_res}) }=<\res_{_S}\omega , \varphi_{|S}>, $$
 which means that (\ref{eq_residue_formula}) immediately follows from (\ref{eq_resint}).
 
 Finally, to check (\ref{item_res_cur}) of Theorem \ref{thm_welldefined}, observe that if $\omega\in \dor^{m-j-1}_S$  then $\res_{_S} \omega$ is well-defined (in virtue of (\ref{item_well_dfn_max}) and (\ref{item_well_dfn}) that we already checked).  We thus just have to check the continuity of $\res_{_S}\omega$ in the sense of currents, for such $\omega$. One more time, we can assume $\omega$ to be  supported in a small neighborhood $U_\xo$ of a point $\xo\in S$. Take a function $\phi\in \cc^\infty_0 (\uxo)$ which is equal to $1$ on some neighborhood of $\supp_\mba \omega$ and notice that for every $\varphi\in \cc_0^{j,\infty}(U_\xo\cap S)$, $\phi\pi^*\varphi$ extends (by $0$) to a smooth form on $M$.  We thus have, in virtue of (\ref{eq_residue_formula}) for the pair $\omega,\phi \pi^*\varphi$
 $$<\res_{_S}\omega, \varphi>=<d\omega,\phi \pi^*\varphi>+(-1)^{m-j-1}<\omega,d(\phi\pi^*\varphi)>-\sum_{S\prec Y\prec M} <\res_{_Y}\omega, \phi\pi^*\varphi> ,$$
 and consequently
 $|<\res_{_S}\omega, \varphi>| \lesssim ||\varphi||_{\wca^j_\infty (S)}, $
 where the constant just depends on $\phi$, $\omega$, and hence is independent of $\varphi$.
\end{proof}

\begin{rem}\label{rem_indep_rho}
	\begin{enumerate}
		\item\label{item_indep_rho} 	The proof has established that, in the situation of (\ref{item_well_dfn}), the limit  (\ref{eq_dfn_res}) is independent of $\rho:U\cap (M\cup S) \to [0,\infty)$, as soon as $\rho$ is a subanalytic, Lipschitz, and satisfies $\rho^{-1}(0)=S\cap U$.
		\item\label{item_trace_res}  If $S\in \Sigma$ is an $(m-1)$-dimensional stratum then $\tra_S \omega =\res_{_S} \omega$ for every $\omega\in \wca^j_p(M)$, if $p$ is sufficiently large (see also Remark \ref{sect_concluding} (\ref{item_res_l2})). This comes down from the definitions.
		\end{enumerate}
\end{rem}

The assumption about integrability of   residues holds in many situations. We illustrate this fact by giving two corollaries that will rely on our density theorems. 

\begin{cor}\label{cor_vanishing_res}
 Assume that $M$ is normal and	let $(Z,\Sigma_Z)$ denote an open subspace of $(\delta M,\Sigma_{\delta M})$. For every $p\in [1,\infty)$  sufficiently close to $1$, we have 
	\begin{equation}\label{eq_wmu_res}
		\wca^j_p (M,Z)=\{\omega\in \wca^j_p (M): \res_{_S}\, \omega=0 \mbox{ for all } S\in \Sigma_Z \}.
	\end{equation}
\end{cor}
\begin{proof}
 For any $p\ge 1$, if $\omega\in \wca^j_p (M)$ has residues on $Z$ that are zero and $\varphi \in \cc^{m-j-1,\infty}_{M\cup Z}(\mba)$, then  (\ref{eq_residue_formula}) immediately yields (\ref{eq_lsp_j}) for the pair $\omega, \varphi$. If $p$ is sufficiently close to $1$, thanks to the density of $\cc^{m-j-1,\infty}_{M\cup Z}(\mba)$ in $\wca^{m-j-1}_{p',M\cup Z}(M)$ (Theorem \ref{thm_dense_formes}, see Remark \ref{rem_p=infty} if $p'=\infty$), this is enough to ensure $\omega \in \wca^j_p (M,Z)$.
 
 Conversely, take $\omega\in \wca^j_p (M,Z)$, $S\in \Sigma_Z$, and $\varphi\in \cc^{m-j-1,\infty}_0(S)$, as well as a definable tubular neighborhood $(\pi,\rho)$. As $Z$ is open, every point $\xo \in S$ has a neighborhood $U$ in $\mba$ that only meets $M$ and strata of   $ \Sigma_Z$.  As our problem is local, we may assume $\varphi$ and $\omega$ to be supported in this neighborhood. By Theorem \ref{thm_pprime}, there is a sequence of smooth forms $\omega_i\in \wca^{j}_{p,\mba \setminus Z}(M)$ that tends to $\omega$ (if $p$ is sufficiently close to $1$). Multiplying $\omega_i$ by a function which is $1$ on the support of $\omega$ if necessary, we can assume it to be compactly supported in $U$. As $U$ only meets $M$  and strata of $\Sigma_Z$, near which $\omega_i$ is zero, this means that $\omega_i$ is compactly supported in $M$, which, by Stokes' formula,  entails
 \begin{equation}\label{eq_omi}
 	\int_M d(\omega_i\wedge \pi^* \varphi) =0.
 \end{equation} 
Moreover, for $\eta>0$, applying Stokes' formula  on $\{x\in U\cap M:\rho(x)\ge \eta\}$, we have (for coherent orientations)
 \begin{equation*}
 	\int_{\rho=\eta} \omega_i\wedge \pi^* \varphi=\int_{\rho\ge \eta} d(\omega_i\wedge \pi^* \varphi)  \overset{(\ref{eq_omi})}=-\int_{\rho\le \eta} d(\omega_i\wedge \pi^* \varphi).
 \end{equation*}
Passing to the limit as $i\to \infty$, we get for almost every $\eta>0$ small:
\begin{equation*}
	\int_{\rho=\eta} \omega\wedge \pi^* \varphi=-\int_{\rho\le \eta} d(\omega\wedge \pi^* \varphi),
 \end{equation*}
which tends to zero as $\eta\to 0$, yielding $\res_{_S}\, \omega=0$ near $\xo$ (see (\ref{eq_dfn_res})).
\end{proof}
\begin{rem}
   Property (\ref{eq_specialization}) is not true for $p$ close to $1$, i.e. there is no specialization property of the vanishing of residues. Indeed, the counterpart of   (\ref{eq_specialization}) in the case  $p$ close to $1$, that may be derived from it, is, assuming $M$ to be connected along $A$,  $$\wca^j_p (M,A)=\wca^j_p (M,int_{\delta M}(A)),$$
where $int_{\delta M}(A)$ stands for the interior of $A$ in $\delta M$. This accounts for the fact that Corollary \ref{cor_vanishing_res} is stated for $Z$ {\it open} subspace of $(\delta M,\Sigma_{\delta M})$.
\end{rem}

 In the case ``$p$ large'', we have the following counterpart:

\begin{cor}\label{cor_trace_formes} 
 Assume that $M$ is normal. 	If $(A,\Sigma_A)$ is a  subspace of $(\delta M,\Sigma_{\delta M})$ then for every $p\in [1,\infty)$  sufficiently large
	\begin{equation}\label{eq_tra0}
		\wca^j_p (M,A)=\{\beta \in \wca^j_p (M): \tra_S \beta =0 \mbox{ for all } S\in \Sigma_A \},\end{equation}
	and consequently, $\cc^{j,\infty }_{\mba \setminus  \adh A} (\mba)$ is dense in $\bigcap_{S\in \Sigma_A} \ker \tra_S$.
\end{cor}

\begin{proof}
	If $\beta\in  \wca^j_p(M,A)$ then by Theorem \ref{thm_dense_formes}, $\beta$ can be approximated (for $p$ sufficiently large) by forms that are smooth up to the boundary and vanish near $A$. We immediately get $\tra_S \beta=0$ for $S\in \Sigma_A$, by definition of $\tra_S$, which shows the first inclusion of (\ref{eq_tra0}).
	
	Conversely, let $\beta\in  \wca^j_p(M)$ be satisfying $\tra_S \beta=0$, as current, for all $S\in \Sigma_A$. By (\ref{eq_specialization}), it suffices to show that  $\beta\in  \wca^j_p(M,A')$, with $A'$ dense in $A$, which reduces to show that  $\beta\in  \wca^j_p(M,S)$ for every maximal stratum $S$ of $\Sigma_A$ (for the relation $\prec$). Let $S$ be such a stratum and $\alpha\in
\wca^{m-j-1}_{p',M\cup S}(M)$.

Since our problem is local, we fix $\xo\in S$.
    There are a neighborhood $U_\xo$ of this point and a definable bi-Lipschitz homeomorphism $\Lambda:U_\xo\cap M \to F\times W_\xo$, with $F$ smooth manifold  and $W_\xo$ neighborhood of the origin in $\R^k$, $k=\dim S$.   Since $\Lambda$ identifies stratified forms, it identifies the respective Sobolev spaces of forms (see (\ref{eq_pullback})) as well as the kernels of the respective trace operators (see Remark \ref{rem_invariance_trace}).  We thus can identify $\uxo\cap M$ with $F\times W_\xo$, $S$ with $W_\xo$, and $\xo$ with $\orn$. 
We can also assume  $\alpha$ to be supported in $U_\xo$ (see Remark \ref{rem_vanishing_local}). Let $\alpha_\sigma:=  \alpha*_k \varphi_\sigma$, where $*_k$ is the convolution product defined in (\ref{mol}) and $\varphi_\sigma $ is as in Lemma \ref{lem_conv} $(\ref{item_conv_lem})$.  

  We claim that if $U_\xo$ is sufficiently small then $\alpha_\sigma$  has  $L^1$ residues  on every stratum for $\sigma$ small. Indeed,  as $S$ is maximal in $\Sigma_A$, it must be open in $A$, and therefore, if $U_\xo$ is chosen sufficiently small, it will not meet any other stratum of $\Sigma_A$ than $S$.
Hence, $\alpha$ vanishes near $\delta M\setminus S$ (because it is compactly supported in $(M\cap U_\xo)\cup S$), and thus, by  (\ref{mol}),  so does $\alpha_\sigma$, for $\sigma>0$ small enough. It means that $\alpha$ and $\alpha_\sigma$ have residues $0$ (and hence $L^1$) on the  strata that are different from $S$. It thus only remains to check that $\res_{_S} \alpha_\sigma$ is $L^1$, and we are indeed going to show that it belongs to $\cc_0^\infty(S)$. Thanks to Theorem \ref{thm_welldefined}, we already know that $\res_{_S} \alpha_\sigma$ and $\res_{_S} \alpha$ are well-defined and continuous in the sense of currents. As a matter of fact, $\res_{_S} \alpha_\sigma=(\res_{_S} \alpha)*\varphi_\sigma$  is a  smooth form on $S$. As $\res_{_S} \alpha_\sigma$ is compactly supported on $S$, this means that $  \alpha_\sigma$ has
  $\cc^\infty_0$ residue on $S$, yielding the claimed fact.

  By Theorem \ref{thm_dense_formes}, for $p$ sufficiently large, there is a sequence $\beta_i\in \cc^{j,\infty}(\mba)$  tending to $\beta$ in $\wca^j_p (M)$.
      The preceding paragraph having established that  $\res_{_Y}\alpha_\sigma=0$ for all $Y\ne S$ and $\res_{_S} \alpha_\sigma\in L^1(S)$, the residue formula  (\ref{eq_residue_formula})  entails
$$<d\alpha_\sigma,\beta_i>+(-1)^{m-j-1}<\alpha_\sigma,d\beta_i>=<\res_{_S} \alpha_\sigma,  \beta_{i|S}> . $$
Since we have actually shown $\res_{_S} \alpha_\sigma \in \cc^{k-j,\infty}_0(S) $ and because $\tra_S$ is continuous in the topology of currents (Theorem \ref{thm_trace_formes}), passing to the limit as $i\to \infty$ we get
$$<d\alpha_\sigma,\beta>+(-1)^{m-j-1}<\alpha_\sigma,d\beta>=<\res_{_S} \alpha_\sigma, \tra_S \beta>= 0. $$  
Passing then to the limit as $\sigma \to 0$  yields (\ref{eq_lsp_j}), from which we  conclude that $\beta\in\wca^j_p (M,A)$,  establishing (\ref{eq_tra0}). The density of $\cc^{j,\infty }_{\mba\setminus  \adh A} (\mba)$ then follows from Theorem \ref{thm_dense_formes}.
  \end{proof}

\begin{rem}		Corollary \ref{cor_trace_formes} 
	amounts to say that for $p$ large $\trd^p \omega =0$ on $S$ is equivalent to $\tra_S \omega=0$, which is  the generalization of Lemma \ref{lem_traTra} to differential forms. 
\end{rem}

 Let us now give a Hodge decomposition theorem which will be a byproduct of the latter two corollaries and Theorem \ref{thm_hodge}. For $p$ sufficiently large, it directly follows from the definition of  $\tra_S$ (Theorem \ref{thm_trace_formes}) that  for every $S\in \Sigma_{\delta M}$ such that $\dim S<j$ 
\begin{equation}\label{eq_tra_0}
	\tra_S \omega=0.\end{equation}
This, together with 
Corollary \ref{cor_trace_formes}. accounts for 
the fact that in Theorem \ref{thm_dense_formes} we are able to assume that our approximations vanish in the vicinity of a prescribed definable set $E$ of dimension $<j$. Similarly,
for $p$ close enough to $1$, it follows from the definition of residues (\ref{eq_dfn_res}) that for every $\omega \in \wca^j_p(M)$ and  every stratum $S\subset \delta M$ satisfying $\dim S< m-j-1$ 
\begin{equation}\label{eq_res_0} \res_{_S} \omega =0.\end{equation}
This accounts for the fact that Theorem \ref{thm_pprime} is able to provide  approximations that  vanish in the vicinity of a given definable subset $E$ of $\delta M$ of dimension $<m-j-1$. 
 These two observations lead us to the following Hodge decomposition theorem on singular sets, without any boundary condition (here $\wsc_{j+1}^p (M)=\wsc_{j+1}^p (M,\emptyset)$ and $E^j_p (M)=E^j_p (M,\emptyset)$):
\begin{cor}\label{cor_hodge_deltaM_ptit}
	For all $p\in (1,\infty)$ sufficiently large, we have for all $j>\dim \delta M+1$:
	\begin{equation}\label{eq_hc_cor}
		\wca^j_p(M)=E^j_p (M)\oplus cE^j_p (M) \oplus \mathscr{H}_p^j (M),
	\end{equation}
	where
	$$\mathscr{H}_p^j (M):=\{ \,\omega \in \wca^j_p(M):\;d\omega=0,\;\; \delta_p \omega=0\,\} $$
	and
	\begin{equation*}\label{eq_cE}
		cE^j_p(M):=\{\omega\in \wca^j_p(M): \exists \beta \in \wsc_{j+1}^p(M), \,\, \omega=\delta_p \beta\}.
	\end{equation*} 
Moreover, if $p\in (1,\infty)$ is close enough to $1$ then (\ref{eq_hc_cor}) holds for all $j<m-\dim \delta M-1$.
\end{cor}
\begin{proof}
	Apply Theorem \ref{thm_hodge} with $A=\delta M$. By (\ref{eq_tra_0}) and Corollary \ref{cor_trace_formes}, the spaces that sit on the right-hand-side of (\ref{eq_hc_cor}) coincide with those that appear on  right-hand-side of (\ref{eq_hc}). In the case ``$p$ close to $1$'', just replace Corollary  \ref{cor_trace_formes} with Corollary  \ref{cor_vanishing_res}, and (\ref{eq_tra_0}) with (\ref{eq_res_0}).
\end{proof}

\begin{rems}\label{sect_concluding} We conclude with a few facts about traces and residues:
\begin{enumerate}	
		\item It is  easy to produce examples of  $\wca^j_p$ forms (with $p$ arbitrarily large) whose trace is not an $L^p$ differential form. Let $M:= (-1,1)\times (0,1)$ and let us examine the trace on $S:=(-1,1)\times \{0\}$ near $x_0=(0,0)$
		 of the exact $1$-form $\omega(x,y):=df(x,y), $
		where $f(x,y):=(x^2+y^2)^{\frac{1}{2}(1-\frac{1}{p})}= |(x,y)|^{1-\frac{1}{p}}$.
Clearly, $|\omega(x,y)|\lesssim |(x,y)|^{-\frac{1}{p}}$ is $L^p$ on $M$ but
		   $\tra_S\omega(x)=\pm (1-\frac{1}{p})|x|^{-\frac{1}{p}}$ is not $L^p$ on $S$. 
		
		  The case $j>0$ is thus less nice than the case  $j=0$, where  Theorem \ref{thm_trace} established that the trace is always an $L^p$ function if $p$ is sufficiently large. Theorem \ref{thm_trace_formes} and Corollary \ref{cor_trace_formes} may however be regarded as partial generalizations of  $(ii)$ and $(iii)$  of Theorem \ref{thm_trace}.
		
		\item\label{item_resl1} It is also easy to produce examples of  $\wca^j_1$ forms that do not possess well-defined residues. Let $M:=\{(x,y)\in (0,1)^2: y< x^2\}$ and let $\omega(x,y):=yf'(x)dx+f(x)dy$, where $f(x)=\frac{1}{x^2}\sin \frac{1}{x}$. We have $\omega\in \wca^1_p(M)$ ($\omega$ is closed) for $p\in [1,\frac{3}{2})$. However, $\esslim_{\eta \to 0^+}\int_{x=\eta}\omega$ does not exist, i.e. the residue at the origin is not well-defined.
		
		This is due to the fact that the residue  on the $x$-axis is not $L^1$.
				We thus cannot spare the assumption ``$\res_Y\omega$ is  $L^1$ for all $S\prec Y\prec M$'' in Theorem \ref{thm_welldefined} (\ref{item_well_dfn}), i.e., we have to exclude forms that oscillate too much on higher strata.

		 \item\label{item_res_l2}  $L^2$ forms do not leave residues on singular strata, which means that if $S\in \Sigma$ satisfies $\dim S<m-1$ then $\res_{_S} \omega=0$ for all $\omega \in  \wca^j_1(M)\cap L^j_2(M)$ (the proof of this is given below). In the case $\dim \delta M<m-1$, i.e. if $\mba$ is a pseudomanifold, this means that $\wca^j_2(M)\subset \wca^j_p (M,\delta  M)$, for all $p$ close to $1$. 
	\end{enumerate}

\begin{proof}[Proof of (\ref{item_res_l2}).]
	Let $\omega \in  \wca^j_1(M)\cap L^j_2(M)$ and $S\in \Sigma$.  As our problem is local, we will assume $S=\{0\} \times\R^k$,  $\rho(x,y)=|y| $, and will identify $\uxo\cap M$ with $F\times W_\xo$, where $\uxo,F,$ and $W_\xo$ are like in the proof of Theorem \ref{thm_trace_formes}.  We set for $\eta>0$, $V^\eta:=\{x\in \uxo\cap M:\rho(x)=\eta\}$ as well as  $U^\eta:=\{x\in \uxo\cap M:\rho(x)\le\eta\}$, and let $F^\eta$ and $G^\eta$ be as in the proof of Theorem \ref{thm_trace_formes}.
 Note first that since  for almost every $y\in S$, we have for $\ep>0$ small $$\int_0^\ep||\omega||_{L^2(G^\eta\times\{y\})}^2d\eta \overset{(\ref{eq_coarea_sph})}\le ||\omega||_{L^2(F^\ep\times\{y\})}^2,$$  we get, after integrating with respect to $y$:
 \begin{equation}\label{eq_resl2_veta}
 	\int_0^\ep||\omega||_{L^2(V^\eta)}^2d\eta \le ||\omega||_{L^2(U^\ep)}^2.
 \end{equation}
	  Since we can argue by induction on the codimension of $S$, we can assume $\res_{_Y}\omega=0$ for every $S\prec Y\prec M$, which, by Theorem \ref{thm_welldefined}, means that $\res_{_S} \omega$ is well-defined. By definition of the essential limit,  for every $\varphi\in \cc_0^{m-j-1,\infty}(S)$, there thus must be  a negligible set $E\subset \R$ such that $\lim_{\eta \to 0^+, \eta\notin E}\int_{V ^\eta} \omega\wedge \pi^*\varphi $ exists,  $\pi$ being the orthogonal projection onto $S$  (see (\ref{eq_dfn_res})). Assume this limit to be nonzero for some $\varphi$. We deduce  that for $\eta\notin E$ positive small, $ \int_{V ^\eta} |\omega|$ is bounded below by some positive constant $\sigma$. By H\"older, we get ($\dim S<m-1$ forces $  \hn^{m-1}(V^\eta )\lesssim \eta$, see \cite[Proposition 4.3.4]{livre})
	$$\sigma  \le  \hn^{m-1}(V^\eta )^{1/2} \cdot ||\omega||_{L^2(V ^\eta)} \lesssim \sqrt{\eta}\cdot ||\omega||_{L^2(V ^\eta)}  ,$$
	 which, after squaring and integrating with respect to $\eta<\ep$, with $\ep>0$ small, gives
	$$\int_0^\ep \frac{\sigma^2}{\eta} d\eta\le \int_0^\ep||\omega||_{L^2(V ^\eta)}^2d\eta \overset{(\ref{eq_resl2_veta})}\le ||\omega||_{L^2(U^\ep)}^2  ,$$ 
	which contradicts that $\omega$ is $L^2$ (the left-hand-side is infinite).
\end{proof}
\end{rems}

	\end{section}

\end{document}